\documentclass[12pt]{amsart}
\usepackage[utf8]{inputenc}
\usepackage{mathrsfs} 
\usepackage{amssymb}
\usepackage{color}
\usepackage{bm}
\usepackage{graphicx}
\usepackage[centertags]{amsmath}
\usepackage{amsfonts}
\usepackage{amsthm}
\usepackage{enumerate}
\usepackage{tocvsec2}
\usepackage{xcolor}
\usepackage[margin=.75in]{geometry}

\newtheorem{theorem}{Theorem}[section]
\newtheorem*{theorem*}{Theorem}

\newtheorem{corollary}[theorem]{Corollary}
\newtheorem{lemma}[theorem]{Lemma}
\newtheorem{rem}[theorem]{Remark}

\newtheorem{proposition}[theorem]{Proposition}

\newtheorem{claim}{Claim}
\newtheorem{claimp}{Claim}
\theoremstyle{definition}

\newcommand{\rr}{\mathbb{R}}
\newcommand{\nn}{\mathbb{N}}
\newcommand{\ee}{\varepsilon}
\newcommand{\supp}{\mathrm{supp}}

\newcommand{\upp}{\upharpoonright}
\newcommand{\ran}{\text{ran}}

\begin{document}

\title[Uniform Upper Estimates]{Uniform upper estimates and the \\ repeated averages hierarchy}

\begin{abstract} We use the repeated averages hierarchy to prove a Ramsey theorem regarding uniform upper estimates of convex block sequences of weakly null sequences.  The base case of the theorem recovers a result of Freeman.

\end{abstract}

\author{R.M. Causey}
\address{Department of Mathematics, Miami University, Oxford, OH 45056, USA}
\email{causeyrm@miamioh.edu}

\thanks{2010 \textit{Mathematics Subject Classification}. Primary: 46B03; Secondary: 46B06.}
\thanks{\textit{Key words}: Uniform boundedness, Ramsey theory, repeated averages hierarchy}

\maketitle

\section{Introduction}

In \cite{KO},\cite{KO2} Knaust and Odell proved the following.

\begin{theorem} Let $X$ be a Banach space. \begin{enumerate}[(i)]\item\cite{KO} If every normalized, weakly null sequence in $X$ has a subsequence dominated by the canonical $c_0$ basis, then there exists a constant $C$ such that every normalized, weakly null sequence in $X$ has a subsequence $C$-dominated by the canonical $c_0$ basis. \item \cite{KO2} For $1<p<\infty$, if every normalized, weakly null sequence in $X$ has a subsequence dominated by the canonical $\ell_p$ basis, then there exists a constant $C$ such that every normalized, weakly null sequence in $X$ has a subsequence $C$-dominated by the canonical $\ell_p$ basis.
\end{enumerate}

\label{veryspec}
\end{theorem}

In \cite{F}, Freeman proved the fully general extension.

\begin{theorem} Let  $X$ be a Banach space and $(g_n)_{n=1}^\infty$  a seminormalized Schauder basic sequence.  Suppose  that  every weakly null sequence  in $X$ has a subsequence dominated by $(g_n)_{n=1}^\infty$.  Then there exists a constant $C$ such that any weakly null sequence in $B_X$ has a subsequence which is $C$-dominated by $(g_n)_{n=1}^\infty$.

\label{spec1}
\end{theorem}

The  Mazur lemma states that a sequence in some Banach space is weakly null if and only if every subsequence of the sequence has a norm null convex block sequence.  A Banach space is said to have the  \emph{weak Banach-Saks property} if each of its weakly null sequences has a subsequence whose Cesaro means converge to zero in norm. Having the weak Banach-Saks property is equivalent to the property that every weakly null sequence in the space has a norm null convex block sequence such that for each convex combination, the convex coefficients are equal.  Schreier \cite{S} gave an example of a Banach space lacking the weak Banach-Saks property, prompting the question of quantifying the complexity of supports and coefficients required to witness weak nullity via the Mazur lemma. In \cite{AMT}, Argyros, Mercourakis, and Tsarpalias solved this problem by introducing the repeated averages hierarchy. The repeated averages hierarchy is an ordinal-indexed collection of families of convex coefficients such that blockings with respect to the zero level simply corresponds to taking subsequences, blockings with respect to the first level corresponds to Cesaro means, blockings with respect to the second level corresponds to Cesaro means of the Cesaro means, etc. In \cite{AMT}, the authors defined the Banach-Saks index of a weakly null sequence, corresponding to the minimum level of the hierarchy required to witness weak nullity via the Mazur lemma.  Furthermore, they showed that for any ordinal $\xi<\omega_1$, there is a weakly null sequence whose Banach-Saks index exceeds $\xi$, a result the $\xi=1$ case of which corresponds to Schreier's example. Since the introduction of the repeated averages hierarchy, a number of classical results, such as Rosenthal's characterization \cite{R} of when a weakly null sequence admits a subsequence generating a spreading model isomorphic to $\ell_1$ and Elton's theorem \cite{E} on near unconditionality, have seen transfinite generalizations using the repeated averages hierarchy. We recall that Elton's theorem states that every normalized, weakly null sequence has a \emph{nearly unconditional} subsequence. We also recall that the subsequences of a given sequence are simply the level zero blockings of that sequence with respect to the repeated averages hierarchy.  Argyros and Gasparis \cite{AG} proved an ordinal quantified verison of Elton's theorem such that the statement corresponding to the ordinal $\xi$ replaces the level zero blocking with the level $\xi$ blocking.

The goal of this work is to provide such a treatment to  Freeman's theorem.  Freeman's theorem has hypotheses and conclusions concerning sequences and their subsequences, which corresponds to level zero blockings with respect to the repeated averages hierarchy. We wish to prove in full generality the corresponding result for level $\xi$ blockings. 
The case $\xi=0$ of Theorem \ref{spec1} recovers the theorem of Freeman, as discussed in Section \ref{ex}.  Our proof  avoids the intermediate use of $C(K)$ spaces which was present in the arguments of both \cite{KO} and \cite{F}. 
Furthermore, it provides ordinal-quantified information not contained in those proofs. By this, we mean that if uniform estimates fail, they must fail in a quantifiable way at a countable ordinal. Moreover, it turns out to be more convenient to work in more generality than in the class of normalized, weakly null sequences. We now discuss the general setting in which we will work.

For a Banach space $X$, $x\in X$, and $n\in\nn$, we let $x\otimes e_n$ denote the sequence whose $n^{th}$ term is $x$, and each other term is zero. We denote by the formal series $\sum_{n=1}^\infty x_n\otimes e_n$ the sequence $(x_1, x_2, \ldots, )$.   We let $c_{00}(X)$ denote the span of $\{x\otimes e_n: x\in X, n\in\nn\}$ in $\ell_\infty(X)$.  For a Banach space $X$, we say a Banach space $(R,r)$ with $c_{00}(X)\subset R\subset \ell_\infty(X)$ is a \emph{subsequential space on} $X$ provided that, with $B_R=\{\varsigma\in R: r(\varsigma)\leqslant 1\}$,  \begin{enumerate}[(i)]\item $B_R\subset B_{\ell_\infty(X)}$, %\item if $\varsigma, \varrho$ are two sequences in $\ell_\infty(X)$ which have the same non-zero terms in the same order, $\varrho\in B_E$ if and only if $\varsigma\in B_E$, 

\item if $\varsigma\in B_R$, then every subsequence of $\varsigma$ is also in $B_R$. \end{enumerate}

Let us say a norm $s$ on $c_{00}(X)$ is \emph{bimonotone}  if \begin{enumerate}[(i)]\item for any $x\in X$ and $n\in\nn$, $s(x\otimes e_n)=\|x\|$,   \item for $\varsigma=(x_n)_{n=1}^\infty \in c_{00}(X)$, $$s(\varsigma)=\sup_{l\leqslant m} s\Bigl(\sum_{n=l}^m x_n\otimes e_n\Bigr)=\lim_m s\Bigl(\sum_{n=1}^m x_n\otimes e_n\Bigr) .$$  \end{enumerate}  In this case, we may define $S$ to be the subspace of $\ell_\infty(X)$ consisting of those $(x_n)_{n=1}^\infty\in \ell_\infty(X)$ such that $$\sup_m s\Bigl(\sum_{n=1}^m x_n\otimes e_n\Bigr)<\infty.$$  We may then extend $s$ to $S$ by letting $$s\bigl((x_n)_{n=1}^\infty \bigr)=\sup_m s\Bigl(\sum_{n=1}^m x_n\otimes e_n\Bigr)$$ for $(x_n)_{n=1}^\infty \in S\setminus c_{00}(X)$.  We refer to $S$ as the \emph{natural domain} of $s$. When a bimonotone norm $s$ on $c_{00}(X)$ is given, we will let $S$ denote the space constructed from $s$ in this way. We note that $(S,s)$ is a Banach space, and the inequality in $(ii)$ remains valid for any $\varsigma\in S$.

We are now ready to state the $\xi=0$  case of our main theorem, which generalizes Theorem \ref{spec1}.

\begin{theorem} Let $X$ be a Banach space, $(R, r)$ a subsequential space on $X$, and $s$ a bimonotone norm on $c_{00}(X)$ with natural domain $S$.  The following are equivalent. \begin{enumerate}\item Every member of $R$ has a subsequence which is a member of $S$. \item For any $\varsigma\in R$, there exist a constant $C$ and a subsequence $\varrho$ of $\varsigma$ such that every subsequence of $\varrho$ lies in $CB_S$. \item There exists a constant $C$ such that for every $\varsigma\in B_R$, there exists a subsequence $\varrho$ of $\varsigma$ such that every subsequence of $\varrho$ lies in $CB_S$.  \end{enumerate}

\label{main1}

\end{theorem}

We will also work in more generality than with the repeated averages hierarchy. Given an infinite subset $M$ of $\nn$, we let $[M]$ denote the infinite subsets of $M$. In Section 6 we recall all required definitions regarding $\xi$-homogeneous probability blocks. For the moment, let us simply recall the   property of a probability block $(\mathfrak{P}, \mathcal{P})$ which connects it to convex block sequences. If $P=(\mathfrak{P}, \mathcal{P})$ is a probability block, then $\mathfrak{P}$ is a collection $\{\mathbb{P}_{M,n}: M\in[\nn], n\in\nn\}$ of finitely supported probability measures on $\nn$ such that for any sequence $\varsigma=(x_n)_{n=1}^\infty$ in the Banach space $X$ and for any $M\in[\nn]$, the sequence $(\sum_{i=1}^\infty \mathbb{P}_{M,n}(i)x_i)_{n=1}^\infty$ is a convex block sequence of $(x_n)_{n\in M}$, where $\mathbb{P}_{M,n}(i)$ is the measure $\mathbb{P}_{M,n}(\{i\})$ of the singleton $\{i\}$.  For convenience, we denote the convex block  sequence $(\sum_{i=1}^\infty \mathbb{P}_{M,n}(i)x_i)_{n=1}^\infty$ of $\varsigma=(x_n)_{n=1}^\infty$ by $\mathbb{E}^P_M \varsigma$.

\begin{theorem} Fix $\xi<\omega_1$, let $P=(\mathfrak{P}, \mathcal{P})$ be a $\xi$-homogeneous probability block.  Let $X$ be a Banach space, $(R, r)$ a subsequential space on $X$, and $s$ a bimonotone norm on $c_{00}(X)$ with natural domain $S$.  The following are equivalent. \begin{enumerate}\item For every $\varsigma\in R$, there exists $M\in [\nn]$ such that for every $N\in[M]$, $\mathbb{E}^P_N\varsigma\in S$. \item For every $\varsigma\in R$, there exist $M\in[\nn]$ and a constant $C$ such that for every $N\in [M]$, $\mathbb{E}^P_N\varsigma\in CB_S$. \item There exists a constant $C$ such that for every $\varsigma\in B_R$, there exists $M\in [\nn]$ such that for every $N\in [M]$, $\mathbb{E}^P_N \varsigma\in CB_S$. \item For every $\varsigma\in R$ and $L\in[\nn]$, there exists $M\in [L]$ such that for all $N\in [M]$, $\mathbb{E}^P_N\varsigma\in S$. \item For every $\varsigma\in R$ and $L\in [\nn]$, there exist $M\in[L]$ and a constant $C$ such that for all $N\in[M]$, $\mathbb{E}^P_N \varsigma\in CB_S$. \item There exists a constant $C$ such that for every $\varsigma\in B_R$ and every $L\in [\nn]$, there exists $M\in [L]$ such that for all $N\in [M]$, $\mathbb{E}^P_N\varsigma\in CB_S$.  \end{enumerate}

\label{main2}
\end{theorem}

\begin{rem}\upshape Let $\Gamma(R,S,P)$ be the infimum of $C>0$ such that property $(6)$ of Theorem \ref{main2} holds, where $\Gamma(R,S, P)=\infty$ if there is no such $C$.  We will later show that if $P=(\mathfrak{P}, \mathcal{P})$ and $Q=(\mathfrak{Q}, \mathcal{Q})$ are any two $\xi$-homogeneous probability blocks, $\Gamma(R,S,P)=\Gamma(R,S,Q)$.  Therefore the six properties in Theorem \ref{main2} are properties of the ordinal $\xi$ which do not depend upon the particular $\xi$-homogeneous probability block $(\mathfrak{P}, \mathcal{P})$.  Therefore we can unambiguously define $\Gamma(R,S, \xi)$ to be $\Gamma(R,S, P)$, where $P=(\mathfrak{P}, \mathcal{P})$ is any $\xi$-homogeneous probability block.  The importance of this result is that in some instances, such as the proofs of \cite[Corollary $4.9$]{CN} and \cite[Corollary $2.9$]{CDP}, it is more convenient to use a probability block which is constructed from two others in order to prove results concerning convex block sequences of convex block sequences. Therefore having flexibility in choosing probability blocks is beneficial. 

We will also show that if $\upsilon\leqslant \xi<\omega_1$, $\Gamma(R,S,\xi)\leqslant \Gamma(R,S, \upsilon)$.  That is, if the pair $R,S$ satisfies the six equivalent properties in Theorem \ref{main2} for some $\upsilon$-homogeneous probability block, then for every $\upsilon<\xi<\omega_1$, the pair $R,S$ satisfies the same six conditions, with at least as small a uniform constant, for any $\xi$-homogeneous probability block.

\end{rem}

\begin{rem}\upshape Let us note that Theorem \ref{main1} is a special case, the $\xi=0$ case, of Theorem \ref{main2}. Let us discuss why Theorem \ref{main2} has six conditions, while Theorem \ref{main1} has only three. We observe that the last three conditions in Theorem \ref{main2} appear similar to the first three, except they state that the desired set $M$ not only exists, but any infinite subset $L$ of $\nn$ contains such a subset $M$. Since Theorem \ref{main1} deals with sequences and subsequences, and since $R$ and $B_R$ contain all subsequences of their members, the hypothesis that every member of $R$ (resp. $B_R$) has a subsequence with some certain property is the same as the hypothesis that every subsequence of a member of $R$ (resp. $B_R$) has a further subsequence with that property. Therefore in the case that we are dealing with subsequences rather than convex blocks, condition $(1)$ of Theorem \ref{main2} is equivalent to condition $(4)$, condition $(2)$ is equivalent to condition $(5)$, and condition $(3)$ is equivalent to $(6)$.  

However, when $(\mathfrak{P}, \mathcal{P})$ is a $\xi$-homogeneous probability block with $0<\xi<\omega_1$, the conditions on $\mathbb{E}^P_M\varsigma$ depend on convex coefficients, which themselves depend upon the positions of the vectors in the sequence $\varsigma$. Therefore, if we know $(1)$ holds and $\varsigma=(x_n)_{n=1}^\infty\in R$ and $L\in [\nn]$ are given, we know that $\varrho=(x_{L(n)})_{n=1}^\infty\in R$. Here, for an infinite subset $I$ of $\nn$, $I(n)$ denotes the $n^{th}$ smallest member of $I$. But an application of $(1)$ to the sequence $(x_{L(n)})_{n=1}^\infty$ yields $K\in[\nn]$ such that for all $N\in [K]$, $\mathbb{E}^P_K \varrho\in S$.   In the sequence/subsequence setting (that is, in the $\xi=0$ setting), this would mean that all subsequences of $(x_{L(K(n))})_{n=1}^\infty$ lie in $S$, and we could finish by letting $M=L(K)$ (that is, $M(n)=L(K(n))$).  However, in the $0<\xi<\omega_1$ case, since the convex blocks coming from $P$ depend upon the positions of the vectors in the sequence, $\mathbb{E}^P_K \varrho$ need not be equal to $\mathbb{E}^P_{L(K)} \varsigma$.  Therefore showing that the first three conditions of Theorem \ref{main2} imply the last three will involve using properties of $\xi$-homogeneous probability blocks to overcome this difficulty.  This requires a combinatorial result, Theorem \ref{god2}.

\end{rem}

We obtain the following transfinite analogue of Freeman's result.

\begin{corollary} Fix $0<\xi<\omega_1$ and let $P=(\mathfrak{P}, \mathcal{P})$ be a $\xi$-homogeneous probability block. Let $X$ be a Banach space and let $(g_n)_{n=1}^\infty$ be a seminormalized Schauder basis. If for every weakly null $\varsigma\in \ell_\infty(X)$, there exists $M\in [\nn]$ such that $\mathbb{E}^P_M\varsigma$ is dominated by $(g_n)_{n=1}^\infty$, then there exists a constant $C$ such that for any weakly null $\varsigma\in B_{\ell_\infty(X)}$ and $L\in [\nn]$, there exists $M\in [L]$ such that for all $N\in [M]$, $\mathbb{E}^P_N\varsigma$ is $C$-dominated by $(g_n)_{n=1}^\infty$.  Furthermore, if such a constant $C$ exists, then depends only on $\xi$, and not on the particular $\xi$-homogeneous probability block $(\mathfrak{P}, \mathcal{P})$.

\label{cor2}
\end{corollary}

\section{The Principle of Uniform Boundedness}

We first recall the Principle of Uniform Boundedness and recite a proof a the gliding hump argument. Our proof of Theorems \ref{main1} and \ref{main2} are analogous to this proof. The idea is quite simple: If uniform inequalities do not hold, we take vectors with worse and worse constants and then glue them together to find a single vector for which the inequality in question does not hold. We use the following argument as a painless introduction to the obstructions present in this argument, as well as give an indication to how we will eventually overcome these obstructions.

Let us suppose that $\mathfrak{A}$ is a collection of (continuous) operators from the Banach space $X$ into a Banach  space $Y$ such that for each $x\in X$, $\sup_{A\in \mathfrak{A}} \|Ax\|<\infty$.  Seeking a contradiction, assume that $\sup_{A\in \mathfrak{A}}\|A\|=\infty$.    Let us recursively choose $x_n\in B_X$, $A_n\in \mathfrak{A}$, and positive constants $C_n, D_n$ such that for all $n\in\nn$, \begin{enumerate}[(i)]\item $D_n>4^n$, \item $\sup_{A\in \mathfrak{A}} \|Ax_n\|=C_n<\infty$, \item $\|A_nx_n\|>D_n$, \item $\sum_{k=1}^{n-1} \frac{C_k}{D_k^{1/2}}<D_n^{1/2}/3$, \item for each $1\leqslant k<n$, $3\cdot 2^n \|A_k\|<D_k^{1/2}D_n^{1/2}$.  \end{enumerate}

As we discuss in the following paragraphs, it follows from these choices that for each $n\in\nn$, $\|A_n\frac{1}{D_n^{1/2}}x_n\|>D_n^{1/2}$, $\sum_{k=1}^{n-1} \|A_n \frac{1}{D_k^{1/2}} x_k\|<\frac{D_n^{1/2}}{3}$, and $\sum_{k=n+1}^\infty \|A_n \frac{1}{D_k^{1/2}} x_k\|<\frac{D_n^{1/2}}{3}$. Therefore with $x=\sum_{n=1}^\infty \frac{x_n}{D_n^{1/2}}$, $$\sup_{A\in \mathfrak{A}} \|Ax\|\geqslant \sup_n \|A_nx\| \geqslant \sup_n D^{1/2}_n-\frac{2D_n^{1/2}}{3}=\infty,$$ yielding the necessary contradiction and finishing the proof. We next discuss how to obtain these estimates, and the analogy to our eventual proof of Theorems \ref{main1} and \ref{main2}.

   We know from (i)-(iv) that $\|A_n\frac{1}{D_n^{1/2}}x_n\|>D_n^{1/2}$. We have to guarantee that the action of $A_n$ on $\frac{1}{D_k^{1/2}} x_k$, $k\neq n$, does not cancel out the action of $A_n$ on $\frac{1}{D_n^{1/2}}x_n$.    We estimate the action of $A_n$ on $\frac{1}{D_k^{1/2}} x_k$, $k\neq n$ differently in the cases $k<n$ and $k>n$. 
		
		For  $k<n$, we have the estimate $$\|A_n\frac{1}{D_k^{1/2}} x_k\|= \frac{1}{D_k^{1/2}}\|A_nx_k\|\leqslant \frac{C_k}{D_k^{1/2}},$$ so $$\sum_{k=1}^{n-1} \|A_n\frac{1}{D_k^{1/2}} x_k\|\leqslant \sum_{k=1}^{n-1}\frac{C_k}{D_k^{1/2}}<D_n^{1/2}/3.$$   Given that $x_k$ was chosen before $A_n$, we could not choose $A_n$ to satisfy anything better than the inequality $\|A_nx_k\|\leqslant C_k$.    In our proofs of Theorems \ref{main1} and \ref{main2}, the   part of the argument analogous to this, in which we use the hypothesis of pointwise estimates, will be  straightforward.

Now suppose $k>n$. Since for $k>n$, $A_n$ was chosen before $x_k$ was chosen, we are able to use the estimate $\|A_nx_k\|\leqslant \|A_n\|$ rather than the weaker estimate $\|A_n x_k\|\leqslant \sup_{A\in \mathfrak{A}} \|Ax_k\|=C_k$.   The stronger estimate yields
 that  $$\|A_n\frac{1}{D_k^{1/2}}x_k\|\leqslant \frac{\|A_n\|}{D_k^{1/2}} \leqslant \frac{D_n^{1/2}}{3\cdot 2^k},$$  from which it follows that $$\sum_{k=n+1}^\infty \|A_n\frac{1}{D_k^{1/2}}x_k\|\leqslant  \sum_{k=n+1}^\infty \frac{D_n^{1/2}}{3\cdot 2^k} \leqslant \frac{D_n^{1/2}}{3},$$ while the weaker estimate $\|A_n x_k\|\leqslant \sup_{A\in \mathfrak{A}} \|Ax_k\|=C_k$ yields the useless $$\|A_n \frac{1}{D_k^{1/2}}x_k\|\leqslant \frac{C_k}{D_k^{1/2}}>D_k^{1/2}.$$    The main obstruction to our proof of Theorems \ref{main1} and \ref{main2} will be to modify the estimate analogous to the $k>n$ case here.  

We will use the fine Schreier families $\mathcal{F}_\xi$, $\xi<\omega_1$ (defined in Section \ref{comb}), to introduce a quantified measure of the failure of uniform upper estimates. That is, if the uniform estimates desired in Theorems \ref{main1} and \ref{main2} do not hold, we will have an ordinal quantification such that for some sequence in $B_R$, the failure of the uniform estimates will be witnessed on sets in $\mathcal{F}_\xi$. Moreover, there will be a minimum countable ordinal $\gamma$ such that the desired estimates are infinitely bad on members of sets in $\mathcal{F}_\gamma$, but which are uniformly controlled by some constants $a_\zeta$ on members of sets in $\mathcal{F}_\zeta$ for $\zeta<\gamma$, and such that $\sup_{\zeta<\gamma} a_\zeta=\infty$. Therefore we can find sequences $\varsigma_n\in B_R$, analogous to $x_n$ above, such that the $D_n$-badness of $\varsigma_n$ can be witnessed on sets in $\mathcal{F}_{\zeta_n}$, and for the later sequences $\varsigma_k$, we have uniform control $a_{\zeta_n}$ over how bad $\varsigma_k$ can be on sets in $\mathcal{F}_{\zeta_n}$.  The finite quantity $a_{\zeta_n}$ will then play the role of $\|A_n\|$ in our proof of the Principle of Uniform Boundedness.    

Let us give two examples.  In our first example, we consider uniform domination of normalized, weakly null sequences in $\ell_2$ by the canonical $c_0$ basis. Of course, uniform upper estimates do not hold.  But note that for each finite $k$, there exists a constant $a_k$, which in this case is equal to $k^{1/2}$, such that for each $C>a_k$ and any normalized, weakly null sequence $(x_n)_{n=1}^\infty$ in $\ell_2$, there exists a subsequence $(y_n)_{n=1}^\infty$ of $(x_n)_{n=1}^\infty$ such that for each set $G$ with $|G|\leqslant k$ and any scalars $(a_n)_{n\in G}$, $$\|\sum_{n\in G}a_ny_n\|_{\ell_2}\leqslant C \|\sum_{n\in G}a_n e_n\|_{c_0}.$$  Thus while we do not have upper estimates of the form $$\|\sum_{n\in G}a_ny_n\|_{\ell_2}\leqslant C\|\sum_{n\in G}a_ne_n\|_{c_0}$$ for all finite $G\subset \nn$, we have this estimate for all $G$ with $|G|\leqslant k$ (that is, $G$ in the fine Schreier family $\mathcal{F}_k$) and a constant $C$ depending on $k$.    Although it is unnecessary to perform such an involved computation in this particular example, we can find weakly null sequences $\varsigma_n\in B_E$, integers $m_1<m_2<\ldots$, and constants $C_n$ such that for $k>n$, the sequence $\varsigma_k$ admits $c_0$ upper estimates with constant $m_n^{1/2}+\ee$ for linear combinations of not more than $m_n$ vectors, and such that any subseqence of $\varsigma_n$ does not exhibit $c_0$ upper estimates with any constant less than $m_n^{1/2}-\ee$ on linear combinations of $m_n$ vectors.  Therefore the lack of $c_0$ upper estimates on $\varsigma_n$ can specifically be witnessed on linear combinations of vectors such that the support of the linear combination lies in $\mathcal{F}_{m_n}$, and we do have uniform upper estimates on such linear combinations.   Thus we can maintain uniform, but worse and worse, upper estimates on linear combinations with supports in $\mathcal{F}_{\zeta_n}$ where $\zeta_n$ approaches the breaking point $\omega$. 

For our second example, let us consider Tsirelson's space $T$ (\cite{T}).  This space is reflexive and infinite dimensional, and therefore it cannot have the property that every normalized, weakly null sequence has a subsequence dominated by the $c_0$ basis.  How do we witness that this space does not have such upper estimates? Tsirelson's space has the property that for any $C>2$, any $m\in\nn$,  and any normalized, weakly null sequence $(x_n)_{n=1}^\infty$ in $T$, there exists a subsequence $(y_n)_{n=1}^\infty$ of $(x_n)_{n=1}^\infty$ such that for any $G$ with $|G|\leqslant m$ and any scalars $(a_n)_{n\in G}$, $$\|\sum_{n\in G}a_ny_n\|_T\leqslant C\|\sum_{n\in G}a_ne_n\|_{c_0}.$$   By diagonalizing, it follows that we can have the upper estimates $$\|\sum_{n\in G}a_ny_n\|_T\leqslant C\|\sum_{n\in G} a_ne_n\|_{c_0}$$ for all $G\in \mathcal{F}_\omega$.  More generally, it is known that for any normalized, weakly null sequence, $m\in\nn$, and $C>2^m$, we can find a subsequence $(y_n)_{n=1}^\infty$ of $(x_n)_{n=1}^\infty$ such that for any $G\in \mathcal{F}_{\omega^m}$ and scalars $(a_n)_{n\in G}$,  $$\|\sum_{n\in G}a_n y_n\|_T \leqslant C\|\sum_{n\in G}a_n e_n\|_{c_0}.$$   But it is known that for any normalized, weakly null sequence $(x_n)_{n=1}^\infty$, $m\in\nn$, and $\ee>0$,  we can find $G\in \mathcal{F}_{\omega^{m+1}}$ and scalars $(a_n)_{n\in G}$ such that $$\|\sum_{n\in G} a_nx_n\|_T >(2^m-\ee)\|\sum_{n\in G} a_ne_n\|_{c_0}.$$   Thus we can find ordinals $\zeta_n\uparrow \omega^\omega$ and weakly null sequences $\varsigma_n\subset B_T$ such that the badness of $\varsigma_n$ is witnessed by linear combinations of vectors whose supports are members of $\mathcal{F}_{\zeta_n}$, on which we do have uniform upper estimates by some constant depending on $\zeta_n$.   Thus we can maintain uniform, but worse and worse, upper estimates on linear combinations with supports in $\mathcal{F}_{\zeta_n}$ where $\zeta_n$ approaches the breaking point $\omega^\omega$.

\section{Domination by a basis and subsequences}\label{ex}

In this section, we show how to deduce Corollary  \ref{cor2} from Theorems \ref{main1} and \ref{main2}.   First we note that any seminormalized Schauder basis is equivalent to a normalized, bimonotone Schauder basis. Therefore it suffices to prove Corollary \ref{cor2} under the stronger hypothesis that $(g_n)_{n=1}^\infty$ is normalized and bimonotone.

Assume $(g_n)_{n=1}^\infty$ is normalized and bimonotone, from which it follows that the sequence of coordinate functionals $(g^*_n)_{n=1}^\infty$ is also a normalized, bimonotone Schauder basis for its closed span in $G^*$.  Let us define the norm $s$ on $c_{00}(X)$ by letting $$s((x_n)_{n=1}^\infty)=\sup_{x^*\in B_{X^*}} \|\sum_{n=1}^\infty x^*(x_n) g_n^*\|_{G^*}.$$   For a fixed sequence $(x_n)_{n=1}^\infty$ and $p\in \nn$, let $I_p$ denote the map from $\text{span}\{g_n:n\leqslant p\}$ to $X$ taking $g_n$ to $x_n$.   Then \begin{align*} \|I_p\| & = \sup\Bigl\{ \bigl\|\sum_{n=1}^p a_n x_n\bigr\|_X: \bigl\|\sum_{n=1}^p a_ng_n\bigr\|_G \leqslant 1\Bigr\} \\ & = \sup\Bigl\{ \bigl|x^*\bigl(\sum_{n=1}^p a_n x_n\bigr)\bigr|: x^*\in B_{X^*},  \bigl\|\sum_{n=1}^p a_ng_n\bigr\|_G \leqslant 1\Bigr\} \\ & = \sup\Bigl\{ \bigl|\sum_{n=1}^p a_n x^*(x_n)\bigr|: x^*\in B_{X^*},  \bigl\|\sum_{n=1}^p a_ng_n\bigr\|_G \leqslant 1\Bigr\} \\ &  =\sup\Bigl\{ \bigl|\bigl(\sum_{n=1}^p x^*(x_n)g_n^*\bigr)\bigl(\sum_{n=1}^p a_n g_n\bigr)\bigr|: x^*\in B_{X^*},  \bigl\|\sum_{n=1}^p a_ng_n\bigr\|_G \leqslant 1\Bigr\} \\ & = \sup\Bigl\{ \bigl\|\sum_{n=1}^p x^*(x_n)g_n^*\bigr\|_{G^*} : x^*\in B_{X^*}\Bigr\}=s\bigl(\sum_{n=1}^p x_n\otimes e_n\bigr).\end{align*}  It is also easy to see that since  $(g^*_n)_{n=1}^\infty$ is normalized and bimonotone, the norm $s$ on $c_{00}(X)$ is bimonotone.   Furthermore, $\varsigma=(x_n)_{n=1}^\infty\in \ell_\infty(X)$ is dominated by $(g_n)_{n=1}^\infty$ if and only if $\sup_p \|I_p\|=\sup_p s(\sum_{n=1}^p x_n\otimes e_n)<\infty$, which happens if and only if $\varsigma$ is in the natural domain $S$ of $s$, and in this case $s(\varsigma)$ is the domination constant.    Therefore applying Theorem \ref{main1} with this choice of $s$ recovers the result of Freeman. Applying Theorem \ref{main2} with this choice of $s$ yields Corollary \ref{cor2}, which is new.

\section{Combinatorics}\label{comb}

Given a set $\Lambda$, we let $\Lambda^{<\omega}$ denote the set of finite sequences whose members lie in $\Lambda$. This includes the empty sequence, which we denote by $\varnothing$.  Given $t\in \Lambda^{<\omega}$, we let $|t|$ denote the length of $t$. If $t\in \Lambda^{<\omega}$ and if $s$ is any (finite or empty) sequence whose members lie in $\Lambda$, we let $t\preceq s$ denote the relation that $t$ is an initial segment of $s$, and $t\prec s$ denotes the relation that $t$ is a proper initial segment of $s$.  A subset $T$ of $\Lambda^{<\omega}$ is called a \emph{tree on} $\Lambda$ (or just a \emph{tree}) provided that for any $s\prec t\in T$, it follows that $s\in T$. We let $s\smallfrown t$ denote the concatenation of $s$ and $t$.  Given a tree $T$, we define the set $MAX(T)$ to be the set of $\preceq$-maximal members of $T$, sometimes called the \emph{leaves} of $T$. We define the derivative $T'$ of $T$ by $$T'=T\setminus MAX(T).$$ We note that $T'\subset T$ is also a tree. We then define the transfinite derivatives $$T^0=T,$$ $$T^{\xi+1}=(T^\xi)',$$ and if $\xi$ is a limit ordinal, $$T^\xi=\bigcap_{\zeta<\xi}T^\zeta.$$   If there exists an ordinal $\xi$ such that $T^\xi=\varnothing$, we say $T$ is \emph{well-founded}, and otherwise we say $T$ is \emph{ill-founded}. If $T$ is well-founded, the \emph{rank} of $T$, denoted by $\text{rank}(T)$, is the miniumum ordinal $\xi$ such that $T^\xi=\varnothing$.   The following properties are standard, so we omit the proof. However, we isolate them for future reference.

\begin{proposition} Let $\Lambda$ be a countable set. Then either $T$ is well-founded and  $\text{\emph{rank}}(T)$ is countable or $T$ is ill-founded.  
\label{tree}
\end{proposition}

Throughout, we identify subsets of $\nn$ with strictly increasing sequences of members of $\nn$ by identifying a subset of $\nn$ with the sequence obtained by listing its members in strictly increasing order. Because of this identification, we use set and sequence notation interchangeably. For $E,F\subset \nn$, we let $E<F$ denote the relation that either $E=\varnothing$, $F=\varnothing$, or $\max E<\min F$. We let $n\leqslant E$ denote the relation that $(n)\leqslant E$ and $E<n$ denote the relation that $E<(n)$.

Given an infinite subset $M$ of $\nn$, we let $[M]$ denote the set of infinite subsets of $M$ and $[M]^{<\omega}$ denote the set of finite subsets of $M$. Given $\mathcal{F}\subset [\nn]^{<\omega}$ and $M\in[\nn]$, we let $\mathcal{F}\upp M$ denote the members of $\mathcal{F}$ which are subsets of $M$. If $\mathcal{F}$ is hereditary, this is the same as $$\mathcal{F}=\{F\cap M: F\in \mathcal{F}\}.$$  
 For $(l_n)_{n=1}^t, (m_n)_{n=1}^t \in [\nn]^{<\omega}$, we say $(l_n)_{n=1}^t$ is a \emph{spread} of $(m_n)_{n=1}^t$ if $m_n\leqslant l_n$ for all $1\leqslant n\leqslant t$.  We will also topologize the power set $2^\nn$ of $\nn$ with the \emph{Cantor topology}, which is the topology making the identification $2^\nn \ni E\leftrightarrow 1_E\in \{0,1\}^\nn$ a homeomorphism, where  $\{0,1\}^\nn$ has the product topology.   We say $\mathcal{F}\subset [\nn]^{<\omega}$ is \begin{enumerate}[(i)]\item \emph{spreading} if $\mathcal{F}$ contains all spreads of its members, \item \emph{hereditary} if $\mathcal{F}$ contains all subsets of its members, \item \emph{compact} if it is compact with respect to the Cantor topology, \item \emph{regular} if it is spreading, hereditary, and compact. \end{enumerate}

We say a regular family $\mathcal{F}$ is \emph{nice} if \begin{enumerate}[(i)]\item it contains all singletons, \item for any $F\in \mathcal{F}$, either $F\in MAX(\mathcal{F})$ or $F\smallfrown (n)\in \mathcal{F}$ for all $n>F$. \end{enumerate}

If $\mathcal{F}\subset [\nn]^{<\omega}$ is regular and non-empty,  then for any $M\in[\nn]$,  we let $M|\mathcal{F}$ denote the maximal initial segment of $M$ which is a member of $\mathcal{F}$.  Since $\mathcal{F}$ is compact, there must exist some $n\in\nn$, and therefore there must exist some minimal $n\in\nn$, such that $(M(1), \ldots, M(n))\notin \mathcal{F}$.  We then let $M|\mathcal{F}=(M(1), \ldots, M(n-1))$, where if $n=1$, $(M(1), \ldots, M(n-1))=\varnothing$. We note that $M|\mathcal{F}$ is possibly empty, and $M|\mathcal{F}$ need not be a member of $MAX(\mathcal{F})$. However, if $\mathcal{F}$ is nice, then $M|\mathcal{F}$ is non-empty and a member of $MAX(\mathcal{F})$. 

Given a nice family $\mathcal{P}$, we use the notation $F=_\mathcal{P} \cup_{n=1}^t F_n$ to mean \begin{enumerate}[(i)]\item $F=\cup_{n=1}^t F_n$, \item $F_1<\ldots <F_t$, \item for each $1\leqslant n\leqslant t$, $F_n\in MAX(\mathcal{P})$. \end{enumerate}    This notation will be used heavily throughout.

Since we identify sets with sequences, every hereditary set $\mathcal{F}\subset [\nn]^{<\omega}$ is naturally identified with a tree on $\nn$.  In this case,  each derivative is also hereditary. In particular, for each ordinal $\xi$,  either $\varnothing\in \mathcal{F}^\xi$ or $\mathcal{F}^\xi=\varnothing$. From this it follows that the rank of $\mathcal{F}$ cannot be a limit ordinal. We also remark here that for a spreading, hereditary set $\mathcal{F}$, a member $F\in \mathcal{F}$ is maximal in $\mathcal{F}$ with respect to the tree order $\preceq$ if and only if it is maximal in $\mathcal{F}$ with respect to inclusion.  Moreover, if $F\in \mathcal{F}$ is not maximal in $F$, then there exists $p\in\nn$ such that $F\cup (n)$ for all $n\geqslant p$. From this it follows that for any regular $\mathcal{F}$ and $M\in[\nn]$, $MAX(\mathcal{F})\upp M=MAX(\mathcal{F}\upp M)$.

Given $M\in[\nn]$ and $n\in\nn$, we let $M(n)$ denote the $n^{th}$ smallest member of $M$, so that $M=(M(n))_{n=1}^\infty$. If $F\in [\nn]^{<\omega}$ and $1\leqslant n\leqslant |F|$, we let $F(n)$ denote the $n^{th}$ smallest member of $F$.   Given $M\in[\nn]$ and $F\subset \nn$ (finite or infinite), $M(F)=(M(n))_{n\in F}$, and note that $M(F)$ is a spread of $F$.     Given $M\in [\nn]$ and $\mathcal{F}\subset [\nn]^{<\omega}$, we let $\mathcal{F}(M)=\{M(F):F\in\mathcal{F}\}$.     We note that if $\mathcal{F}$ is hereditary, then $\mathcal{F}(M)$ is also hereditary, and either $\mathcal{F}, \mathcal{F}(M)$ are both ill-founded or both well-founded with the same rank.  

Note that $\mathcal{F}(M)$ need not be spreading, since  a spread of a member of $\mathcal{F}(M)$ need not be a subset of $M$, but $\mathcal{F}(M)$ is \emph{spreading relative to} $M$. That is, if $E\in \mathcal{F}(M)$ and $F\in [M]^{<\omega}$ is a spread of $E$, then $F\in \mathcal{F}(M)$. To see this, note that since $E\in \mathcal{F}(M)$ and $ F\in [M]^{<\omega}$, we can write $E=M(G)$ and $F=M(H)$ for some $G\in\mathcal{F}$ and $H\in[\nn]^{<\omega}$. Then for each $1\leqslant n\leqslant |E|$, $$M(G(n))=E(n)\leqslant F(n) = M(H(n)),$$ from which it follows that $G(n)\leqslant H(n)$. That is, $H$ is a spread of $G$, and therefore $H\in \mathcal{F}$ and $M(H)\in \mathcal{F}(M)$.

We recall the \emph{fine Schreier families}, $(\mathcal{F}_\xi)_{\xi<\omega_1}$,  and the \emph{Schreier families}, $(\mathcal{S}_\xi)_{\xi<\omega_1}$.   We define $$\mathcal{F}_0=\{\varnothing\},$$ $$\mathcal{F}_{\xi+1} = \{\varnothing\}\cup \{(n)\cup F: n<F\in \mathcal{F}_\xi\},$$ and if $\xi<\omega_1$ is a limit ordinal, we fix $\xi_n\uparrow \xi$ and define $$\mathcal{F}_\xi=\{F: \exists n\leqslant F\in \mathcal{F}_{\xi_n}\}.$$  Note that we employed a choice of $\xi_n\uparrow \xi$ when $\xi$ is a limit ordinal, but none of our results depend upon making any particular choice of $\xi_n$.  We also note that for $n<\omega$, $\mathcal{F}_n$ is the set of subsets of $\nn$ with cardinality not exceeding $n$. For convenience, we let $\mathcal{F}_{\omega_1}=[\nn]^{<\omega}$, the set of all finite subsets of $\nn$. Of course, $\mathcal{F}_{\omega_1}$ is spreading and hereditary, but ill-founded and non-compact.  Moreover, $\mathcal{F}_\xi\subset \mathcal{F}_{\omega_1}$ for all $\xi\leqslant \omega_1$.

We define $$\mathcal{S}_0=\mathcal{F}_1=\{\varnothing\}\cup \{(n): n\in\nn\},$$ $$\mathcal{S}_{\xi+1}= \Bigl\{\bigcup_{n=1}^t F_n: F_1<\ldots <F_t, F_n \in \mathcal{S}_\xi, t\leqslant F_1\Bigr\},$$ and if $\xi<\omega_1$ is a limit ordinal, fix $\xi_n\uparrow \xi$ and define $$\mathcal{S}_\xi= \{F: \exists n\leqslant F\in \mathcal{S}_{\xi_n}\Bigr\}.$$

  We note that the choice of $\xi_n$ in the definition of $\mathcal{S}_\xi$ need not be the same as the choice of $\xi_n$ in the definition of $\mathcal{F}_\xi$. Furthermore, some of the results we cite below for $\mathcal{S}_\xi$ assumed that when $\xi<\omega_1$ is a limit ordinal,  $\xi_n\uparrow \xi$ is chosen so that for all $n\in\nn$, $\xi_n$ is a successor and $\mathcal{S}_{\xi_n+1}\subset \mathcal{S}_{\xi_{n+1}}$.    It is known (see, for example, the presentation in \cite{Con}) that such a choice of $\xi_n\uparrow \xi$ exists.

Given a spreading, hereditary family $\mathcal{F}$ and a regular family $\mathcal{P}$, we let $$\mathcal{F}[\mathcal{P}]= \{\varnothing\}\cup \Bigl\{\bigcup_{n=1}^t F_n: F_1<\ldots <F_t, \varnothing\neq F_n\in \mathcal{P}, (\min F_n)_{n=1}^t \in \mathcal{F}\Bigr\},$$ and for $M\in[\nn]$, we let \begin{align*} \mathcal{F}^M[\mathcal{P}] & =\{\varnothing\}\cup \Bigl\{\bigcup_{n=1}^t F_n: F_1<\ldots <F_t, \varnothing\neq F_n\in \mathcal{P}\upp M, (\min F_n)_{n=1}^t\in \mathcal{F}(M)\Bigr\}.\end{align*} Note that $\mathcal{F}^\nn[\mathcal{P}]=\mathcal{F}[\mathcal{P}]$ and if $M\in[\nn]$ and $N\in[M]$, the spreading property of $\mathcal{F}$ yields that $\mathcal{F}(N)\subset \mathcal{F}(M)$, so $\mathcal{F}^N[\mathcal{P}] \subset \mathcal{F}^M[\mathcal{P}]$.  

\begin{rem}\upshape Note that if $\mathcal{F}$ is regular, $\mathcal{P}$ is nice, $M\in[\nn]$, and $F=_\mathcal{P}\cup_{n=1}^t F_n\in [M]^{<\omega}$, then $F\in \mathcal{F}^M[\mathcal{P}]$ if and only if $(\min F_n)_{n=1}^t \in \mathcal{F}(M)$. Indeed, if $F=_\mathcal{P}\cup_{n=1}^t F_n$ and $(\min F_n)_{n=1}^t\in \mathcal{F}(M)$, then $F_n\in \mathcal{P}$ by the definition of $=_\mathcal{P}$ and $F_n\subset F\in [M]^{<\omega}$. From this it follows that $F\in \mathcal{F}^M[\mathcal{P}]$.   For the converse,  suppose $F=_\mathcal{P}\cup_{n=1}^t F_n$ and $F=\cup_{n=1}^r E_n$ for some $E_1<\ldots <E_r$, $\varnothing\neq E_n\in \mathcal{P}\upp M$, and $(\min E_n)_{n=1}^r\in \mathcal{F}(M)$.  Then since $E_1, \ldots, E_s$ are successive members of $\mathcal{P}$ and $F_1, \ldots, F_t$ are successive, maximal members of $\mathcal{P}$, there exist $0=l_0< l_1<\ldots <l_t\leqslant r$ such that for each $1\leqslant j< t$,  $$\bigcup_{n=1}^{l_j} E_n \preceq \bigcup_{n=1}^j F_n\preceq \bigcup_{n=1}^{l_j+1}E_n.$$   From this it follows that $\min E_{l_{j-1}+1} \leqslant \min F_j$ for each $1\leqslant j\leqslant t$, and $(\min F_n)_{n=1}^t$ is a spread of a subset of $(\min E_n)_{n=1}^r \in \mathcal{F}(M)$. Since $\mathcal{F}(M)$ is hereditary,  $\mathcal{F}(M)$ is spreading relative to $M$, and $(\min F_n)_{n=1}^t\in [M]^{<\omega}$ is a spread of a subset of $(\min E_n)_{n=1}^r$, $(\min F_n)_{n=1}^t\in \mathcal{F}(M)$, and $F=_\mathcal{P}\cup_{n=1}^t F_n\in \mathcal{F}^M[\mathcal{P}]$. 

\label{triv}
\end{rem}

In the next proposition, we recall the following facts from \cite[Propositions $3.1$, $3.2$]{Con} regarding regular families, and specifically the fine Schreier and Schreier families. Here, we note that, as shown in that source, the rank of a regular family is the same as its Cantor-Bendixson index as a topological space.

\begin{proposition} \begin{enumerate}[(i)]\item  For any two regular families $\mathcal{F}, \mathcal{G}$, $\text{\emph{rank}}(\mathcal{F})\leqslant \text{\emph{rank}}(\mathcal{G})$ if and only if there exists $N\in[\nn]$ such that $\mathcal{F}(N)\subset \mathcal{G}$ if and only if for any $M\in[\nn]$, there exists $N\in[M]$ such that $\mathcal{F}(N)\subset \mathcal{G}$. \item For each $\xi<\omega_1$, $\mathcal{F}_\xi$ is regular with rank $\xi+1$, and $\mathcal{F}_\xi$ is nice for $0<\xi<\omega_1$. \item For $\xi<\omega_1$, $\mathcal{S}_\xi$ is nice with rank $\omega^\xi+1$. \item If $\mathcal{F}, \mathcal{G}$ is regular, then so is $$\mathcal{H}:= \{F\cup G: F<G, F\in \mathcal{F}, G\in \mathcal{G}\}.$$  Moreover, if $\text{\emph{rank}}(\mathcal{F})=\mu+1$ and $\text{\emph{rank}}(\mathcal{G})=\zeta+1$, then $\text{\emph{rank}}(\mathcal{H})=\zeta+\mu+1$. \item If $\mathcal{F}, \mathcal{P}$ are regular (resp. nice) families, so is $\mathcal{F}[\mathcal{P}]$. Moreover, if  $\text{\emph{rank}}(\mathcal{F})=\xi+1$ and $\text{\emph{rank}}(\mathcal{P})=\zeta+1$, then $\text{\emph{rank}}(\mathcal{F}[\mathcal{P}])=\zeta\xi+1$.    \item The fine Schreier families have the \emph{almost monotone property}. That is, for any $\zeta<\xi<\omega_1$, there exists $l\in\nn$ such that if $l<F\in \mathcal{F}_\zeta$, then $F\in \mathcal{F}_\xi$. \item The Schreier families have the almost monotone property.  \end{enumerate}

\label{reg}
\end{proposition}

\begin{rem}\upshape Note that for any regular $\mathcal{F}, \mathcal{P}$ and any $M\in[\nn]$, $\mathcal{F}[\mathcal{P}]$ and $\mathcal{F}^M[\mathcal{P}]$ are tree isomorphic to subtrees of each other, and therefore have the same rank.  The inclusion is a tree isomorphism of $\mathcal{F}^M[\mathcal{P}]$ onto a subtree of $\mathcal{F}[\mathcal{P}]$, while $F\mapsto M(F)$ is a tree isomorphism of $\mathcal{F}[\mathcal{P}]$ onto a subtree of $\mathcal{F}^M[\mathcal{P}]$.

\end{rem}

\begin{rem}\upshape We will frequently use the fact that if $M\in[\nn]$ and $N\in[M]$, and if $M(m)=N(n)$ for some $m,n\in\nn$, it follows that $m\geqslant n$. That is, the $n^{th}$ term of the subsequence $N$ of $M$ cannot occur before the $n^{th}$ position in $M$. For such $M,N$ and $F\in [\nn]^{<\omega}$, since $N(F)\in [M]^{<\omega}$, there exists $G\in [\nn]^{<\omega}$ such that $M(G)=N(F)$. It then follows from the first sentence of the remark that this $G$ must be a spread of $F$.

\label{gas}

\end{rem}

The next proposition concerns our standard diagonalization procedure. 

\begin{proposition} Suppose that $M_1\supset M_2\supset \ldots$ are infinite subsets of $\nn$ and $M(n)=M_n(n)$ for all $n\in\nn$. Fix a limit ordinal $\zeta<\omega_1$ and let $\zeta_k\uparrow \zeta$ be such that  $\mathcal{F}_\zeta= \{F: \exists k\leqslant F\in \mathcal{F}_{\zeta_k}\}$. Then $$\mathcal{F}_\zeta(M)\subset \bigcup_{k=1}^\infty \mathcal{F}_{\zeta_k}(M_k)$$ and for any regular family $\mathcal{P}\subset [\nn]^{<\omega}$, $$\mathcal{F}^M_\zeta[\mathcal{P}]\subset \bigcup_{k=1}^\infty \mathcal{F}_{\zeta_k}^{M_k}[\mathcal{P}].$$

\label{silly}
\end{proposition}

\begin{proof} For the first inclusion, if $\varnothing \neq F\in\mathcal{F}_\zeta$, we may fix $k\in\nn$ such that $k\leqslant F\in \mathcal{F}_{\zeta_k}$. Let $N=(M_k(1), \ldots, M_k(k-1), M(k), M(k+1))\in [M_k]$. Then $M(F)=N(F)=M_k(G)$ for some spread $G$ of $F$. The fact that $G$ is a spread of $F$ follows from the content of Remark \ref{gas}.  By the spreading property of  $\mathcal{F}_{\zeta_k}$, $G\in \mathcal{F}_{\zeta_k}$. Therefore $$M(F)=N(F)=M_k(G)\in \mathcal{F}_{\zeta_k}(M_k).$$  Since $F\in\mathcal{F}_\zeta$ was arbitrary, we have the first inclusion. 

For the second inclusion, we note that any set $H\in \mathcal{F}^M_\zeta[\mathcal{P}]$ can be written as $H=\cup_{n=1}^t F_n$, where $F_1<\ldots <F_t$, $\varnothing\neq F_n\in \mathcal{P}\upp M$, and $(\min F_n)_{n=1}^t\in \mathcal{F}_\zeta(M)$.  Then $(\min F_n)_{n=1}^t= M(F)$ for some $F\in \mathcal{F}_\zeta$. As in the preceding paragraph, we may choose some $k$ such that $k\leqslant F\in \mathcal{F}_{\zeta_k}$ and note that $(\min F_n)_{n=1}^t=M(F)=M_k(G)\in \mathcal{F}_{\zeta_k}$ for some spread $G$ of $F$.  This yields that $H\in \mathcal{F}_{\zeta_k}^{M_k}[\mathcal{P}]$. Since $H\in \mathcal{F}^M_\zeta[\mathcal{P}]$ was arbitrary, we deduce the second inclusion.

\end{proof}

We next recall a special case of the infinite Ramsey theorem, the proof of which was achieved in steps by Nash-Williams \cite{NW}, Galvin  and Prikry \cite{GP}, Silver \cite{Silver}, and Ellentuck \cite{Ellentuck}. The general form of the infinite Ramsey theorem holds for analytic sets, but we will only need it for closed sets.

\begin{theorem} If $\mathcal{V}\subset [\nn]$ is closed, then for any $L\in[\nn]$, there exists $M\in[L]$ such that either $[L]\subset \mathcal{V}$ or $[L]\cap \mathcal{V}=\varnothing$. 

\label{Ramsey}
\end{theorem}

The following dichotomy was shown in \cite{PR} by Pudl\'{a}k and R\"{o}dl. 

\begin{theorem} For any non-empty  regular family $\mathcal{P}$, any $N\in[\nn]$, and any finite partition $MAX(\mathcal{P})\upp N=\cup_{n=1}^r T_n$ of $MAX(\mathcal{P})\upp N$,  there exist $M\in[N]$ and $1\leqslant n\leqslant r$ such that $MAX(\mathcal{P})\upp M\subset T_n$. 

\label{pr}
\end{theorem}

\begin{rem}\upshape We will use a continuous version of Theorem \ref{pr}. Namely, for any $N\in[\nn]$, any non-empty regular family $\mathcal{F}$, any $\delta>0$, and any bounded function $h_0:MAX(\mathcal{F})\upp N\to \rr$, there exist $M\in[N]$ and $a,b\in\rr$ such that $0<b-a<\delta$ and $h_0(F)\in [a,b]$ for any $F\in \mathcal{F}\upp M$.  Indeed, we partition the range of $h_0$ into finitely many sets $A_1, \ldots, A_n$ of diameter less than $\delta$ and partition $MAX(\mathcal{F})\upp N$ into $T_1, \ldots, T_n$ by letting $T_i= h_0^{-1}(A_i)$. We then apply Theorem \ref{pr} to obtain $M\in[N]$ and $1\leqslant i\leqslant n$ such that $MAX(\mathcal{F})\upp M\subset T_i$.  We conclude by letting $a=\inf A_i$ and $b=\sup A_i$. 

\label{prremark}
\end{rem}

The next theorem was shown in \cite{G} by Gasparis.

\begin{theorem} If $\mathcal{F}, \mathcal{G}\subset [\nn]^{<\omega}$ are hereditary, then for any $L\in[\nn]$, there exists $M\in[L]$ such that either $\mathcal{F}\upp M\subset \mathcal{G}$ or $\mathcal{G}\upp M\subset \mathcal{F}$. 

\label{gasp}
\end{theorem}

We also include the following result, whose proof is similar to that of the proof of Theorem \ref{gasp}. 

\begin{proposition} Let $\mathcal{F}, \mathcal{G}$ be hereditary families such that $\mathcal{F}$ is nice. Then for any $L\in[\nn]$, there exists $M\in[L]$ such that either $\mathcal{F}\upp M\subset \mathcal{G}$ or $\mathcal{G}\cap MAX(\mathcal{F}\upp M)=\varnothing$.

\label{dich}
\end{proposition}

\begin{proof} If $\varnothing\notin \mathcal{F}$, then $\mathcal{F}\upp L =  \varnothing \subset \mathcal{G}$. Therefore in this case we can take $M=L$. Assume $\varnothing\in \mathcal{F}$. Recall that for $M\in[\nn]$, $M|\mathcal{F}$ is the maximal (possibly empty) initial segment of $M$ which lies in $\mathcal{F}$.   Note that $M\mapsto M|\mathcal{F}$ is locally constant. Indeed, since $\mathcal{F}$ is nice, for any $M\in[\nn]$, $M|\mathcal{F}\in MAX(\mathcal{F})$, and $M|\mathcal{F}=N|\mathcal{F}$ for any $N$ which has $M|\mathcal{F}$ as an initial segment. Since the set of such $N$ is clopen, it follows that $$\mathcal{V}=\{M\in[\nn]: M|\mathcal{F}\in \mathcal{G}\}$$ is closed. By Theorem \ref{Ramsey}, there exists $M\in[\nn]$ such that either $[M]\subset \mathcal{V}$ or $[M]\cap \mathcal{V}=\varnothing$. If $[M]\subset \mathcal{V}$, then fix $F\in MAX(\mathcal{F}\upp M)$ and fix $F\prec N\in[M]$. Note that since $F\in MAX(\mathcal{F}\upp M)$,  $F=N|\mathcal{F}$.  Since $N\in[M]\subset \mathcal{V}$,  $F=N|\mathcal{F}\in \mathcal{G}$.  This shows that $MAX(\mathcal{F}\upp M)\subset \mathcal{G}$. Since $\mathcal{G}$ is hereditary, $\mathcal{F}\upp M\subset \mathcal{G}$.  

If $[M]\cap \mathcal{V}=\varnothing$, for any $F\in MAX(\mathcal{F}\upp M)$, we fix $F\prec N\in [M]$ and note that, since $N\notin \mathcal{V}$, $F=N|\mathcal{F}\notin \mathcal{G}$. Therefore $MAX(\mathcal{F}\upp M)\cap \mathcal{G}=\varnothing$.

\end{proof}

We isolate here the following technical piece which will be of later use.

\begin{proposition} Let $\mathcal{P}, \mathcal{Q}$ be nice families such that $\text{\emph{rank}}(\mathcal{Q})\leqslant \text{\emph{rank}}(\mathcal{P})$. Then for any  $M,L\in[\nn]$, $K\in[L(M)]$, and $m\in\nn$, there exist $F,E\in[\nn]^{<\omega}$ such that \begin{enumerate}[(i)]\item $m<F\in MAX(\mathcal{Q})\upp M$, \item $L(F\setminus (\min F))=K(E)$, \item $E\in \mathcal{P}$. \end{enumerate}

\label{star}
\end{proposition}

\begin{proof} By Proposition \ref{reg}, there exists $N\in[\nn]$ such that $\{N(G):G\in\mathcal{Q}\}\subset \mathcal{P}$.  Note that since $K\in [L(M)]$, if $L(M(n))=K(k)$ for some $k,n\in\nn$, then $n\geqslant k$.  Choose $k_1\in\nn$ such that $m<k_1$. Since $K\in[L(M)]$, $K(k_1)=L(M(n_1))$ for some $n_1 \geqslant k_1$.    

Now assume that $k_1< \ldots< k_s$ and $ n_1< \ldots< n_s$ have been chosen such that $L(M(n_i))=K(k_i)$ for each $1\leqslant i\leqslant s$.  As stated in  the previous paragraph, $n_i\geqslant k_i$ for each $1\leqslant i\leqslant s$. Choose $k_{s+1}>N(M(n_s))$ and note that $k_{s+1} >N(M(n_s)) \geqslant n_s\geqslant k_s$.  Choose $n_{s+1}$ such that $K(k_{s+1})= L(M(n_{s+1}))$ and note that $n_{s+1} \geqslant  k_{s+1}>N(M(n_s))\geqslant n_s$.   This completes the recursive construction. 

Since $\mathcal{Q}$ is nice, there exists $t\in \nn$ such that $(M(n_i))_{i=1}^t\in MAX(\mathcal{Q})$.    Let $F=(M(n_i))_{i=1}^t$ and let $E=(k_i)_{i=2}^t$.  Then $$m<k_1\leqslant n_1\leqslant M(n_1) =\min F,$$ so $(i)$ is satisfied. For $(ii)$, note that $$L(F\setminus (\min F))= (L(M(n_i))_{i=2}^t=(K(k_i))_{i=2}^t=K(E).$$  For $(iii)$, we note that since $k_{i+1}>N(M(n_i))$ for all $i\in\nn$, $E=(k_i)_{i=2}^t=(k_{i+1})_{i=1}^{t-1}$ is a spread of $$(N(M(n_i)))_{i=1}^{t-1} \subset N(F)\in \{N(G): G\in\mathcal{Q}\}\subset \mathcal{P},$$ and $E\in\mathcal{P}$, since $\mathcal{P}$ is spreading and hereditary.

\end{proof}

We next collect some standard facts about functions on $\omega_1+=[0, \omega_1]$.

\begin{proposition} Suppose  $\Gamma:\omega_1+\times\omega\to [0,\infty]$ is a function such that for each $k<\omega$, $\zeta\mapsto \Gamma(\zeta, k)$ is non-decreasing.  Define $\Gamma:\omega_1+\to [0,\infty]$ by $\Gamma(\zeta)=\sup_{k<\omega} \Gamma(\zeta, k)$.  \begin{enumerate}[(i)]\item For $k<\omega$, the map $\zeta\mapsto \Gamma(\zeta, k)$ is continuous if and only if whenever $\zeta\leqslant \omega_1$ is a limit ordinal and $C\in (0,\infty)$ is such that $\sup_{\mu<\zeta}\Gamma(\mu, k)<C$, then $\Gamma(\zeta, k)\leqslant C$. \item If for each $k<\omega$, $\zeta\mapsto \Gamma(\zeta, k)$ is continuous, then so is $\Gamma:\omega_1+\to [0,\infty]$.  \item Suppose that $\Gamma:\omega_1+\to [0,\infty]$ is continuous,  $\Gamma(0, k)=0$ for all $k<\omega$, and $\Gamma(\zeta, k+p)\leqslant \Gamma(\zeta+p, k)\leqslant p+\Gamma(\zeta, k+p)$ for all $k,p<\omega$. Then either $\Gamma(\omega_1)<\infty$, or there exists a countable ordinal $\gamma$ such that $\Gamma(\gamma)=\infty$. In the case that $\Gamma(\omega_1)=\infty$,  if $\gamma$ is the minimum  ordinal such that $\Gamma(\gamma)=\infty$, then $\gamma$ is a countable limit ordinal and for any $p<\omega$,  $\{\Gamma(\zeta, p): \zeta<\gamma\}$ is an unbounded subset of $[0, \infty)$.  \end{enumerate}

\label{ord1}

\end{proposition}

\begin{proof}$(i)$ Continuity means that for any limit ordinal $\zeta\leqslant \omega_1$, $\lim_{\mu\uparrow \zeta}\Gamma(\mu, k)=\Gamma(\zeta, k)$. Since $\mu\mapsto \Gamma(\mu,k)$ is non-decreasing, $$\lim_{\mu\uparrow \zeta}\Gamma(\mu, k)=\sup_{\mu<\zeta}\Gamma(\mu, k)\leqslant \Gamma(\zeta, k).$$  Then $\mu\mapsto \Gamma(\mu, k)$ is continuous if and only if for all limit ordinals $\zeta\leqslant \omega_1$, $\sup_{\mu<\zeta}\Gamma(\mu, k) \geqslant \Gamma(\zeta, k)$ if and only for all limit ordinals $\zeta\leqslant \omega_1$ and $C\in\rr$ such that $\sup_{\mu<\zeta}\Gamma(\mu,k)<C$, $\Gamma(\zeta, k)\leqslant C$. 

$(ii)$ If each $\zeta\mapsto \Gamma(\zeta, k)$ is non-decreasing, then so is $\zeta\mapsto \Gamma(\zeta)$. We check continuity as in $(i)$. Since each $\zeta\mapsto \Gamma(\zeta, k)$ is continuous, then for any limit ordinal $\zeta\leqslant \omega_1$, \begin{align*} \Gamma(\zeta) & = \sup_{k<\omega} \Gamma(\zeta, k)=\sup_{k<\omega}\sup_{\mu<\zeta}\Gamma(\mu,k)= \sup_{\mu<\zeta}\sup_{k<\omega}\Gamma(\mu,k)= \sup_{\mu<\zeta}\Gamma(\mu).  \end{align*}

$(iii)$ If  $\Gamma:\omega_1+\to [0,\infty]$ is continuous and $\Gamma(\omega_1)=\infty$, then for each $l\in\nn$, $\min \Gamma^{-1}([l,\infty])<\omega_1$. Therefore $\gamma:=\sup_l \min \Gamma^{-1}([l,\infty])<\omega_1$, and since $\Gamma$ is non-decreasing, $\Gamma(\gamma)=\infty$.  

Now suppose that $\Gamma(\omega_1)=\infty$ and $\gamma$ is the minimum ordinal $\zeta$ such that $\Gamma(\zeta)=\infty$. By the previous paragraph, $\gamma$ is countable. Since $\Gamma(0)=0$, $\gamma\neq 0$.  We also note that $\gamma$ cannot be a successor. Indeed, if $\gamma=\zeta+1$, then $\Gamma(\zeta)<\infty$ and by the properties of $\Gamma$ assumed in $(iii)$, $$\Gamma(\gamma)=\sup_{k<\omega}\Gamma(\zeta+1, k) \leqslant 1+\sup_{k<\omega}\Gamma(\zeta,k+1) \leqslant 1+\Gamma(\zeta)<\infty.$$   Thus $\gamma$ is a countable limit ordinal.  By minimality of $\gamma$, $\Gamma(\zeta, p)\leqslant \Gamma(\zeta)<\infty$ for all $\zeta<\gamma$ and $p<\omega$.  Therefore for any $p<\omega$,  $\{\Gamma(\zeta, p): \zeta<\gamma\}$ is a subset of $[0,\infty)$. By continuity of $\Gamma$, $\sup_{\zeta<\gamma}\Gamma(\zeta)=\infty$. Therefore for any $0<D<\infty$, there exists $\zeta<\gamma$ such that $D<\Gamma(\zeta)$. This means there exists $k<\omega$ such that $D<\Gamma(\zeta, k)$. By the properties of $\Gamma$, $D<\Gamma(\zeta, k)\leqslant \Gamma(\zeta+k, 0)$.  This shows that $\{\Gamma(\zeta, 0): \zeta<\gamma\}$ is unbounded. Now for any $p<\omega$ and $0<D<\infty$, since $\{\Gamma(\zeta, 0): \zeta<\gamma\}$ is unbounded, we may find $\zeta<\gamma$ such that $\Gamma(\zeta, 0)>p+D$. Then $$D<\Gamma(\zeta, 0) - p\leqslant \Gamma(\zeta+p, 0) - p\leqslant \Gamma(\zeta, p).$$

\end{proof}

\section{Probability blocks}\label{rev}

Recall that for a regular family $\mathcal{P}$, a member of $\mathcal{P}$ is maximal in $\mathcal{P}$ with respect to inclusion if and only if it is maximal with respect to the initial segment ordering. Therefore,  $MAX(\mathcal{P})$ denotes the set of maximal members of $\mathcal{P}$ with respect to  either one of these orders.  Similarly, for any $N\in[\nn]$,  a member of $\mathcal{P}\upp N$ is maximal in $\mathcal{P}\upp N$ with respect to inclusion if and only if it is maximal in $\mathcal{P}\upp N$ with respect to the initial segment ordering, so $MAX(\mathcal{P}\upp N)$ is unambiguous.  Furthermore, $MAX(\mathcal{P}\upp N)=MAX(\mathcal{P})\upp N$.

Let us recall that if $\mathcal{P}$ is a nice family, it is spreading, hereditary, compact, contains all singletons, and for each $F\in\mathcal{P}$, either $F\in MAX(\mathcal{P})$ or $F\smallfrown (n)\in \mathcal{P}$ for all $F<n$.  For $M\in[\nn]$ and a nice family $\mathcal{P}$, $M|\mathcal{P}$ denotes the maximal initial segment of $M$ which is in $\mathcal{P}$. Since $\mathcal{P}$ contains all singletons, $M|\mathcal{P}\neq \varnothing$. Since for each $F\in \mathcal{P}$, either $F\in MAX(\mathcal{P})$ or $F\smallfrown (n)\in \mathcal{P}$ for all $F<n$, it follows that $M|\mathcal{P}\in MAX(\mathcal{P})$.   

In this section, we treat probability measures on $\nn$ as functions on $\nn$ by letting $\mathbb{P}(i)=\mathbb{P}(\{i\})$. For a probability measure $\mathbb{P}$ on $\nn$, we let $\supp(\mathbb{P})=\{i\in\nn: \mathbb{P}(i)\neq 0\}$.   Suppose that $\mathcal{P}$ is a nice family and $\mathfrak{P}=\{\mathbb{P}_{M,n}: M\in[\nn], n\in\nn\}$ is a collection of probability measures on $\nn$.  Then we say $(\mathfrak{P}, \mathcal{P})$ is a \emph{probability block} provided that \begin{enumerate}[(i)]\item for each $M\in[\nn]$, $\text{supp}(\mathbb{P}_{M,1})= M|\mathcal{P}$, \item for $M,N\in[\nn]$ and $r\in\nn$ such that $\supp(\mathbb{P}_{M,r})\prec N\in[\nn]$, it follows that $\mathbb{P}_{N,1}= \mathbb{P}_{M,r}$. \end{enumerate}

These properties together imply that for any $M\in[\nn]$,  $(\supp(\mathbb{P}_{M,n}))_{n=1}^\infty$ is the unique partition of $M$ into successive, maximal members of $\mathcal{P}$, and if two measures $\mathbb{P}_{M,m}$, $\mathbb{P}_{N,n}$ have equal supports then they are equal measures. From this it follows that for any $n\in\nn$, the map $M\mapsto \mathbb{P}_{M,n}$ is locally constant.   We refer to this as the \emph{permanence property}. By the permanence property, if $F=_\mathcal{P}\cup_{n=1}^t F_n$, then $\mathbb{P}_{M,n}=\mathbb{P}_{N,n}$ for each $1\leqslant n\leqslant t$ and any two $M,N\in[\nn]$ which have $F$ as an initial segment.  Therefore we can define for $F=_\mathcal{P}\cup_{n=1}^t F_n$ and each $1\leqslant n\leqslant t$ the measure $\mathbb{P}_{F,n}$ by letting $\mathbb{P}_{F,n}=\mathbb{P}_{M,n}$ for any $F\prec M\in[\nn]$, and the definition of $\mathbb{P}_{F,n}$ is independent of the particular choice of $M$.

If $(\mathfrak{P}, \mathcal{P})$ is a probability block, then any set which is a finite union of successive, maximal members of $\mathcal{P}$ is uniquely expressible as such. That is, if $F=\cup_{n=1}^s F_n=\cup_{n=1}^t G_n$, where $F_1<\ldots <F_s$, $G_1<\ldots <G_t$, and $F_m, G_n\in MAX(\mathcal{P})$ for all $1\leqslant m\leqslant s$ and $1\leqslant n\leqslant t$, then $s=t$ and $F_n=G_n$ for all $1\leqslant n\leqslant t$.  Therefore if $F=_\mathcal{P}\cup_{n=1}^s F_n$ and $F=_\mathcal{P} \cup_{n=1}^t G_n$, then $s=t$ and $F_n=G_n$ for all $1\leqslant n\leqslant t$.   Moreover, $F=_\mathcal{P}\cup_{n=1}^t F_n$ if and only if for some (equivalently, every) $M\in[\nn]$ such that $F\prec M$, $F=\cup_{n=1}^t \supp(\mathbb{P}_{M,n})$, and in this case $F_n=\supp(\mathbb{P}_{M,n})$ for each $1\leqslant n\leqslant t$.   

Given a Banach space $X$ and a sequence $\varsigma=(x_n)_{n=1}^\infty\subset X$,  $M\in[\nn]$, and a probability block $P=(\mathfrak{P}, \mathcal{P})$, we define $\mathbb{E}^P_M\varsigma$ to be the sequence whose $n^{th}$ term is $\sum_{i=1}^\infty \mathbb{P}_{M,n}(i)x_i=\sum_{i\in \supp(\mathbb{P}_{M,n})} \mathbb{P}_{M,n}(i)x_i$.   We denote this $n^{th}$ term by $\mathbb{E}^P_M \varsigma(n)$. The superscript $P$ in this notation is to refer that expectations are taken with respect to measures coming from the probability block $P=(\mathfrak{P}, \mathcal{P})$. When we are considering a second probability block $Q=(\mathfrak{Q}, \mathcal{Q})$, we use the notations $\mathbb{E}^Q_M\varsigma$, $\mathbb{E}^Q_M\varsigma(n)$.     When we are only considering a fixed probability block $P=(\mathfrak{P}, \mathcal{P})$ and no confusion can arise, we omit the superscript $P$ from the notation and write simply $\mathbb{E}_M \varsigma$, $\mathbb{E}_M\varsigma (n)$. 

If $X$ is a Banach space,  $F=_\mathcal{P} \cup_{n=1}^t F_n$, $\varsigma\in \ell_\infty(X)$, and $M,N\in[\nn]$ are such that $F\prec M$ and $F\prec N$, then $\mathbb{E}_M\varsigma(n)=\mathbb{E}_N\varsigma(n)$ for all $1\leqslant n\leqslant t$ by the permanence property.  For this reason, we can unambiguously define for $F=_\mathcal{P}\cup_{n=1}^t F_n$ and $\varsigma=(x_n)_{n=1}^\infty\subset X$ the sequence $\mathbb{E}_F^P \varsigma$ such that, with $F\prec M\in[\nn]$,  $\mathbb{E}_F^P\varsigma (n)=\mathbb{E}_M^P\varsigma(n)$ for all $n\leqslant t$ and $\mathbb{E}_F^P\varsigma(n)=0$ for all $n>t$.   That is, by the facts mentioned at the beginning of this paragraph, the definition $\mathbb{E}_F^P\varsigma(n)=\mathbb{E}_M^P\varsigma(n)$ for all $n\leqslant t$ is independent of the choice of $F\prec M\in[\nn]$.  When there is no potential for confusion, we write $\mathbb{E}_F \varsigma$ in place of $\mathbb{E}_F^P \varsigma$ and $\mathbb{E}_F\varsigma(n)$ in place of $\mathbb{E}_F^P \varsigma(n)$.

For a regular family $\mathcal{F}$ and $N\in[\nn]$, we define $$N\oplus \mathcal{F}=\{(n,F)\in\nn\times [\nn]^{<\omega}: n\in F\in \mathcal{F}\upp N\}$$ and $$N\ominus \mathcal{F}= \{(n,F)\in \nn\times [\nn]^{<\omega}: n\in F\in MAX(\mathcal{F}\upp N)\}.$$

Given $N\in[\nn]$ and a regular family $\mathcal{F}$, we say a function $h:N\ominus \mathcal{F}\to \rr$ is \emph{tail independent} provided that for any $n\in N$ and $F,G\in MAX(\mathcal{F})\upp N$ such that $[1,n]\cap F=[1,n]\cap G$, $h(n,F)=h(n,G)$.   This is equivalent to saying that for any $H\in \mathcal{F}$ and $F,G\in MAX(\mathcal{F})$ such that $H\preceq F$ and $H\preceq G$, $h(n,F)=h(n,G)$ for all $n\in H$.    In this case, we can extend $h$ to a function, which we also denote by $h$, defined on $N\oplus \mathcal{F}$. For $n\in H\in \mathcal{F}\upp N$, we fix $F\in MAX(\mathcal{F})\upp N$ such that $H\preceq F$ and let $h(n,H)=h(n,F)$. We note that, by tail independence, this definition is independent of the choice of maximal extension $F$ of $H$.   We refer to the extension of $h:N\ominus \mathcal{F}\to \rr$ to $h:N\oplus \mathcal{F}\to \rr$ as the \emph{natural extension} of $h$.  

If $f$ is a function defined on a subset $S$ of $\nn$ such that $F=_\mathcal{P} \cup_{n=1}^t F_n\subset S$, we define $$\mathbb{E}^P_F f = \sum_{n=1}^t \sum_{i\in F_n} f(i)\mathbb{P}_{F,n}(i)= \sum_{n=1}^t \mathbb{E}_{\mathbb{P}_{F,n}}f|_{F_n}.$$ We note that the last sum is a sum of expectations in the usual sense.  If $f:N\ominus \mathcal{F}\to \rr$ is tail independent and we extend $f$ to its natural extension on $N\oplus \mathcal{F}$, and if  $F=_\mathcal{P} \cup_{n=1}^t F_n$, $$\mathbb{E}_F^P f(\cdot, F) = \sum_{n=1}^t \sum_{i\in F_n} f(i, F) \mathbb{P}_{F,n}(i) = \sum_{n=1}^t \sum_{i\in F_n} f(i, \cup_{m=1}^n F_m) \mathbb{P}_{F,n}(i).$$ That is, if $i\in \cup_{m=1}^n F_m$, then $f(i, F)$ does not depend on $F_{n+1}, \ldots, F_t$.   Again, if no confusion can arise, we omit the superscript $P$ from the notation.

We next observe that, up to passing to infinite subsets, every function on $N\ominus \mathcal{F}$ is close to being tail independent.

\begin{proposition} Fix $N\in[\nn]$, a regular family $\mathcal{F}$ containing a singleton, and a bounded function $g:N\ominus \mathcal{F}\to \rr$. For any sequence $(\ee_n)_{n=1}^\infty$ of positive numbers, there exist $M\in[N]$ and  a tail independent function $h:M\ominus \mathcal{F}\to \rr$ such that if $(M(n), F)\in M\ominus \mathcal{F}$, then $|g(M(n), F)-h(M(n),F)|<\ee_n$.    Furthermore, $h$ may be taken to satisfy $h(n, F)\geqslant g(n,F)$ for all $(n,F)\in M\ominus \mathcal{F}$.

\label{tail}
\end{proposition}

\begin{proof}  Let us begin with an observation. Fix $F\in \mathcal{F}$ and let $$\mathcal{G}=\{G\in \mathcal{F}: F<G, F\cup G\in \mathcal{F}\}.$$  This is a regular family, possibly containing only $\varnothing$.    For any $L\in[\nn]$, any $\ee>0$, and  any bounded function $f:MAX(\mathcal{G})\to \rr$, as noted in Remark \ref{prremark}, we may choose $L_1\in [L]$ such that $\text{diam}\{f(G): G\in MAX(\mathcal{G})\upp L_1\}<\ee$.    Of course, $\text{diam}\{f(G): G\in MAX(\mathcal{G})\upp L_2\}<\ee$ for any $L_2\in [L_1]$.   

We now return to the proof, wherein we will apply the argument in the preceding paragraph.  For the base step of the recursion, we define $N_0=N$.  

Now assume that $N_0\supset N_{n-1}\in [\nn]$ and $m_1<\ldots <m_{n-1}$, $m_i\in N_i$ have been chosen.  If $n=1$, choose $m_1\in N_0$. If $n>1$, choose $m_n\in N_{n-1}$ such that $m_n>m_{n-1}$.   Now let $F_1, \ldots, F_t$ be an enumeration of those members $F$ of $\mathcal{F}$ such that $F\subset (m_1, \ldots, m_n)$ and $\max F=m_n$. If there are no such $F$, we simply let $N_n=N_{n-1}\cap (m_n, \infty)$.  Otherwise we apply the procedure from the first paragraph of the proof to obtain $N_{n-1}\supset L_1\supset \ldots \supset L_t$ such that for each $1\leqslant i\leqslant t$, $$\text{diam}\{g(m_n, F_i\cup G): F_i<G\in [L_i]^{<\omega}, F_i\cup G\in MAX(\mathcal{F})\}< \ee_n.$$  Let $N_n=L_t$.  This completes the recursive step.

Now let $M(n)=m_n$ for each $n\in\nn$, so $M\in [N]$. Fix $(m,F)\in M\ominus \mathcal{F}$ and define $$h(m,F)=\sup \{g(m, ([1,m]\cap F)\cup H):  m<H \in [M]^{<\omega}, ([1,m]\cap F)\cup H\in MAX(\mathcal{F})\}.$$ Note that this definition only depends on $[1,m]\cap F$. From this it follows that if $(m,F),(m,G)\in M\ominus \mathcal{F}$ and $[1,m]\cap F=[1,m]\cap G$, then  $h(m,F)=h(m,G)$. Thus $h$ is tail independent. If $(m,F)\in M\ominus \mathcal{F}$, we define $H_0=([1,m]\cap F)\setminus [1,m]$ and note that $m<H_0$ and $F=([1,m]\cap F)\cup H_0\in MAX(\mathcal{F})$.  From this it follows that \begin{align*} g(m,F) & = g(m, ([1,m]\cap F)\cup H_0) \\ & \leqslant \sup \{g(m, ([1,m]\cap F)\cup H):  m<H \in [M]^{<\omega}, ([1,m]\cap F)\cup H\in MAX(\mathcal{F})\} \\ & =h(m,F).\end{align*}  For $(m,F)\in M\ominus \mathcal{F}$, there exists some $n\in\nn$ such that $m=m_n$ and, with $F_1, \ldots, F_t$ as defined in the recursive construction, there exists some $1\leqslant i\leqslant t$ such that $[1,m]\cap F=F_i$.    Let $H_0=F\setminus F_i$, so that  $H_0\in [N_n]^{<\omega}\subset [L_i]^{<\omega}$, and note that since $$\text{diam}\{g(m_n, F_i\cup G): F_i<G\in [L_i]^{<\omega}, F_i\cup G\in MAX(\mathcal{F})\}< \ee_n,$$ it follows that \begin{align*} g(m,F) +\ee_n & =g(m_n, F_i\cup H_0) +\ee_n \\ & \geqslant \ee_n+ \inf\{g(m_n, F_i\cup H): m_n<H\in [L_i]^{<\omega}, F_i\cup H\in MAX(\mathcal{F})\} \\ & \geqslant \sup\{g(m_n, F_i\cup H): m_n<H\in [L_i]^{<\omega}, F_i\cup H\in MAX(\mathcal{F})\} \\ & \geqslant \sup\{g(m_n, F_i\cup H): m_n<H\in [M]^{<\omega}, F_i\cup H\in MAX(\mathcal{F})\}  \\ & = h(m,F).\end{align*}  The last inequality uses the fact that $m_n<H\in [M]^{<\omega}$ implies $m_n<H\in [L_i]^{<\omega}$.

\end{proof}

For a countable ordinal $\xi$, we say that a probability block $(\mathfrak{P}, \mathcal{P})$ is $\xi$-\emph{sufficient} provided that for any regular $\mathcal{F}$ with $\text{rank}(\mathcal{F}) \leqslant \omega^\xi$, for any $L\in[\nn]$,  and any $\ee>0$, there exists $M\in[L]$ such that $$\sup \{\mathbb{P}_{N,1}(E): N\in[M], E\in \mathcal{F}\} \leqslant \ee.$$   We say $(\mathfrak{P}, \mathcal{P})$ is $\xi$-\emph{homogeneous} provided that it is $\xi$-sufficient and $\text{rank}(\mathcal{P})=\omega^\xi+1$.

We note that there is only one nice family $\mathcal{P}$ with rank $\omega^0+1=2$, which is $\mathcal{F}_1=\{F\in[\nn]^{<\omega}: |F|\leqslant 1\}$.  From this it follows that if $(\mathfrak{P}, \mathcal{P})$ is $0$-homogeneous, then $\mathcal{P}=\mathcal{F}_1$ and $\mathfrak{P}$ is the collection of Dirac measures given by $\mathbb{P}_{M,n}=\delta_{M(n)}$.  We refer to this unique $0$-homogeneous probability block as the \emph{Dirac block}.   In this case,  if $\varsigma=(x_n)_{n=1}^\infty$, then for $M\in [\nn]$, $\mathbb{E}_M\varsigma=(x_{M(n)})_{n=1}^\infty$.   Therefore in the $\xi=0$ case, the hypothesis that for every $\varsigma$ in $R$ (or $B_R$) and every $L\in[\nn]$, there exists $M\in[L]$ such that $\mathbb{E}_M \varsigma$ has some specified property is precisely the hypothesis that every member of $R$ (or $B_R$) has a subsequence with that specified property. Therefore we can see how the sequence-subsequence hypothesis in Theorem \ref{main1} fits as a particular case of our sequence-$(\mathfrak{P}, \mathcal{P})$-convex block hypothesis of Theorem \ref{main2}.

We next observe that if $0<\xi<\omega_1$, $\xi$-homogeneous probability blocks can be taken to have small $c_0$ norms. 

\begin{proposition} Fix $0<\xi<\omega_1$ and let $P=(\mathfrak{P}, \mathcal{P})$ be a $\xi$-homogeneous probability block. For any sequence $(\delta_n)_{n=1}^\infty$ of positive numbers and any $L\in[\nn]$, there exists $M\in[L]$ such that for all $N\in[M]$, $n\in\nn$, and $i\in\nn$, $\mathbb{P}_{N,n}(i) \leqslant \delta_n$. 

\label{diag}

\end{proposition}

\begin{proof} Since $CB(\mathcal{F}_1)=2<\omega^\xi$, for any $M\in[L]$, by the definition of $\xi$-sufficient, for any $\delta>0$ and $L\in [\nn]$, there exists $M\in[L]$ such that $$\sup \{\mathbb{P}_{N,1}(i): i\in\nn, N\in[M]\}=\sup\{\mathbb{P}_{N,1}(F):F\in\mathcal{F}_1, N\in[M]\} \leqslant \delta.$$  We recursively select $M_1\supset M_2\supset \ldots$, $M_n\in [L]$ such that for all $n\in\nn$, $$\sup\{\mathbb{P}_{N,1}(i): N\in[M_n], i\in\nn\}\leqslant \delta_n.$$   Now let $M(n)=M_n(n)$ and fix $N\in[M]$. Let $N_n=N\setminus \cup_{i=1}^{n-1}\text{supp}(\mathbb{P}_{M,i})\in [M_n]$. Note that by the permanence property, for any $n\in\nn$, $\mathbb{P}_{N,n}=\mathbb{P}_{N_n, 1}$, so $$\sup\{\mathbb{P}_{N,n}(i):i\in\nn\}\leqslant \delta_n.$$

\end{proof}

The $\xi=0$ case of the following result is trivial. The $0<\xi<\omega_1$ case of the following was shown in \cite{CN}, which combines \cite[Corollary $4.10$]{Sch} and \cite[Lemma $3.12$]{CN}.

\begin{theorem} Let $P=(\mathfrak{P}, \mathcal{P})$ be a $\xi$-homogeneous probability block. Fix $K\in[\nn]$ and a bounded function $h_0:K\ominus \mathcal{P}\to \rr$. If for some $D\in \rr$ and each $F\in MAX(\mathcal{P})\upp K$, $\mathbb{E}_F h_0(\cdot, F)\geqslant D$, then for any $\delta>0$,  there exists $L\in[K]$ such that for each $F\in \mathcal{P}$, there exists $L(F)\subset G\in MAX(\mathcal{P})\upp K$ such that for each $n\in F$, $h_0(L(n), G)\geqslant D-\delta$.

\label{pt}
\end{theorem}

\begin{theorem} Fix $0<\zeta,\xi<\omega_1$ and assume $(\mathfrak{P}, \mathcal{P})$ is a $\xi$-homogeneous probability block. Assume $L\in[\nn]$, $D\in\rr$, and  $g:L\ominus \mathcal{F}_\zeta[\mathcal{P}]\to \rr$ are such that $g$ is a bounded function and for every $F\in MAX(\mathcal{F}_\zeta[\mathcal{P}])\upp L$, $\mathbb{E}^P_F g(\cdot, F)\geqslant D$. Then for any sequence $(\delta_n)_{n=1}^\infty$ of positive numbers, any $M\in[\nn]$,  any $0<\upsilon\leqslant \xi$, and any $\upsilon$-homogeneous probability block $(\mathfrak{Q}, \mathcal{Q})$, there exist $$F=_\mathcal{Q}\bigcup_{n=1}^t F_n \in \mathcal{F}_{\zeta+1}^M[\mathcal{Q}],$$ numbers $b_1, \ldots, b_t\in\rr$,   and  sets $H_1, \ldots, H_t\in \mathcal{P}\upp L$, $H\in MAX(\mathcal{F}_\zeta[\mathcal{P}])\upp L$, such that  \begin{enumerate}[(i)]\item  for each $1\leqslant n\leqslant t$ and each $i\in\nn$, $\mathbb{Q}_{F,n}(i) \leqslant \delta_n$, \item $L(F_n\setminus (\min F_n))\subset H_n$,  \item  $H=_\mathcal{P}\bigcup_{n=1}^t H_n$, \item for each $1\leqslant n\leqslant t$ and $m\in F_n\setminus (\min F_n)$, $g(L(m), H) \geqslant b_n-\delta_n$, \item $\sum_{n=1}^t b_n\geqslant D$.   \end{enumerate}

\label{god2}
\end{theorem}

\begin{proof} By replacing $(\delta_n)_{n=1}^\infty$ with a sequence of smaller numbers if necessary, we may assume this sequence is decreasing. By replacing $M$ with a subset thereof and appealing to Proposition \ref{diag}, we may assume that for any $n\in\nn$, $$\sup \{\mathbb{Q}_{N,n}(i): i\in\nn, N\in[M]\}\leqslant \delta_n.$$  Therefore $(i)$ will be satisfied by any choice of $F\in \mathcal{F}^M_{\zeta+1}[\mathcal{Q}]\upp M$.

Let $J\subset \rr$ be a compact interval containing the range of $g$.    By applying Proposition \ref{tail} with $\ee_n=\delta_n/2$, we can find $L_0\in [L(M)]$ and a bounded, tail independent function $h:L_0\ominus \mathcal{F}_\zeta[\mathcal{P}]\to\rr$ such that if $(L_0(n), F)\in L_0 \ominus \mathcal{F}_\zeta[\mathcal{P}]$, $$g(L_0(n), F)\leqslant h(L_0(n), F)\leqslant g(L_0(n), F)+\delta_n/2.$$   Since $MAX(\mathcal{F}_\zeta[\mathcal{P}])\upp L_0\subset MAX(\mathcal{F}_\zeta[\mathcal{P}])\upp L$ and $g|_{L_0\ominus \mathcal{F}_\zeta[\mathcal{P}]}\leqslant h|_{L_0\ominus \mathcal{F}_\zeta[\mathcal{P}]}$, it follows that for any $H\in MAX(\mathcal{F}_\zeta[\mathcal{P}])\upp L_0$, $$D\leqslant \mathbb{E}^P_H g(\cdot, H)\leqslant \mathbb{E}^P_H h(\cdot, H).$$    Since $L_0\in [L(M)]$, it follows that $L_0=L(M(T))$ for some $T\in[\nn]$.   Let $K_0=L_0(M(T))\in [L(M(T))]$.

Since $h$ is tail independent, we may consider the natural extension $h:K_0\oplus \mathcal{F}_\zeta[\mathcal{P}]\to J$.  Define $f_1:MAX(\mathcal{P})\upp K_0\to J$ by letting $\eta_1(F)=\mathbb{E}^P_F h(\cdot,F)$.  This is well defined because we have taken the natural extension. As noted in Remark \ref{prremark}, we may fix $a_1, b_1\in\rr$ such that $b_1-a_1<\delta_1/4$ and $L_1\in [K_0]$ such that $f_1(F)\in [a_1, b_1]$ for all $F\in MAX(\mathcal{P})\upp L_1$.   By considering the function $h_0:L_1\ominus \mathcal{P}\to J$ given by $h_0(n,F)=h(n,F)$, we may apply Theorem \ref{pt} to find $K_1\in [L_1]$ such that for any $F\in \mathcal{P}$, there exists $K_1(F)\subset G\in MAX(\mathcal{P})$ such that for any $n\in F$, $h(K_1(n),G)=h_0(K_1(n), G) \geqslant a_1-\delta_1/4\geqslant b_1-\delta_1/2$.   Since $K_1\in [L(M(T))]$, we may appeal to Proposition \ref{star} to find $F_1\in MAX(\mathcal{Q})\upp M(T)$ and $E_1\in \mathcal{P}$ such that $M(T(1))<F_1$ and $L(F_1\setminus (\min F_1))=K_1(E_1)$.    Since $E_1\in \mathcal{P}$, there exists $K_1(E_1)\subset H_1\subset MAX(\mathcal{P})\upp L_1$ such that for each $n\in E_1$, $h(K_1(n), H_1)\geqslant b_1-\delta_1/2$.  Note that the $n=1$ case of item $(ii)$ is satisfied with this choice.   Since for any $m\in F_1\setminus (\min F_1)$, $L(m)=K_1(n)$ for some $n\in E_1$, it follows that $h(L(m), H_1)\geqslant b_1-\delta_1/2$ for any $m\in F_1\setminus (\min F_1)$.    Since $M(T(1))<F_1\in [M(T)]^{<\omega}$ and $L_0=L(M(T))$, for any $m\in F_1\setminus (\min F_1)$, $L(m)=L_0(p)$ for some $p>1$, from which it follows that for any $H_1\preceq H\in MAX(\mathcal{F}_\zeta[\mathcal{P}])$,  \begin{align*} g(L(m), H) & =g(L_0(p), H)\geqslant h(L_0(p), H) -\delta_p/2 = h(L_0(p), H_1) -\delta_p/2 \geqslant h(L_0(p),H_1)-\delta_1/2 \\ & =h(L(m), H_1)-\delta_1/2 \geqslant b_1-\delta_1.\end{align*} Here we have used that $h$ is tail independent and $L_0(p)=L(m)\in H_1$.    Therefore $g(L(m), H)\geqslant b_1-\delta_1$ for each $m\in F_1\setminus (\min F_1)$ and $H_1\preceq H\in MAX(\mathcal{F}_\zeta[\mathcal{P}])$. Therefore for $n=1$, item $(iv)$ will be satisfied for our eventual choice of $H_1\preceq H\in MAX(\mathcal{F}_\zeta[\mathcal{P}])$

Now suppose that for some $s\in\nn$, $b_1, \ldots, b_s$, $F_1<\ldots<F_s$, $F_n\in MAX(\mathcal{Q})$, $L_1\supset K_1\supset \ldots \supset L_s\supset K_s$, $H_1<\ldots<H_s$, $H_n\in MAX(\mathcal{P})$ have been chosen such that $(\min H_n)_{n=1}^s\in \mathcal{F}_\zeta$. If $(\min H_n)_{n=1}^s\in MAX(\mathcal{F}_\zeta)$, we let $t=s$ and we are done. Otherwise we perform the following recursive step. Let $A_s=(\min H_n)_{n=1}^s$ and $B_s=\cup_{n=1}^s H_n$. Note that since $\mathcal{F}_\zeta$ is nice, $A_s\smallfrown (m)\in \mathcal{F}_\zeta$ for all $A_s<m$. Since $\max B_s\geqslant \max A_s$, it follows that $A_s\smallfrown (m)\in \mathcal{F}_\zeta$ for all $B_s<m$, and therefore $B_s\cup F\in \mathcal{F}_\zeta[\mathcal{P}]$ for all $B_s<F\in \mathcal{P}$.      We now choose $L_{s+1}\in [K_s]$ such that $B_s<L_{s+1}$ and define $f_{s+1}:MAX(\mathcal{P})\upp L_{s+1}\to J$ by $f_{s+1}(F)=\mathbb{E}^P_F h(\cdot, B_s\cup F)$.    Again using the fact stated in Remark \ref{prremark}, by passing to a subset of $L_{s+1}$ and relabeling, we may assume there exist $a_{s+1}, b_{s+1}$ such that $b_{s+1}-a_{s+1}<\delta_{s+1}/4$ and $f_{s+1}(F)\in [a_{s+1}, b_{s+1}]$ for all $F\in MAX(\mathcal{P})\upp L_{s+1}$.    By another application of Theorem \ref{pt} applied to the function $h_s:L_{s+1}\ominus \mathcal{P}\to J$ given by $h_s(n, F)=h(n,B_s\cup F)$, we can find $K_{s+1}\in [L_{s+1}]$ such that for any $F\in \mathcal{P}$, there exists $K_{s+1}(F)\subset G\in MAX(\mathcal{P})\upp L_{s+1}$ such that for each $n\in F$, $h(K_{s+1}(n), B_s\cup G) \geqslant a_{s+1}-\delta_{s+1}/4 \geqslant b_{s+1}-\delta_{s+1}/2$.    By another application of Proposition \ref{star}, we may find $F_{s+1}\in MAX(\mathcal{Q})\upp M(T)$ and $E_{s+1}\in \mathcal{P}$ such that $$\max\{\max F_s, M(T(s+1)), M(\min H_s)\}<F_{s+1}$$ and $L(F_{s+1}\setminus (\min F_{s+1}))=K_{s+1}(E_{s+1})$.    Since $E_{s+1}\in \mathcal{P}$, there exists $K_{s+1}(E_{s+1})\subset H_{s+1}\in MAX(\mathcal{P})\upp L_{s+1}$ such that for all $n\in E_{s+1}$, $h(K_{s+1}(n), B_s\cup H_{s+1}) \geqslant b_{s+1}-\delta_{s+1}/2$.  Then the $n=s+1$ case of item $(ii)$ is satisfied.   For any $m\in F_{s+1}\setminus (\min F_{s+1})$, there exists $n\in E_{s+1}$ such that $L(m)=K_{s+1}(n)$, and $h(L(m), B_s\cup H_{s+1})\geqslant b_{s+1}-\delta_{s+1}$.       Since $M(T(s+1))<F_{s+1}\in [M(T)]^{<\omega}$ and $L_0=L(M(T))$, for any $m\in F_{s+1}\setminus (\min F_{s+1})$, $L(m)=L_0(p)$ for some $p>s+1$, from which it follows that for any $B_s\cup H_{s+1}\preceq H\in MAX(\mathcal{F}_\zeta[\mathcal{P}])$,  \begin{align*} g(L(m),  H) & =g(L_0(p),  H)\geqslant h(L_0(p), H) -\delta_p/2 = h(L_0(p), B_s\cup H_{s+1}) - \delta_p/2 \\ & \geqslant h(L_0(p),B_s\cup H_{s+1})-\delta_{s+1}/2  =h(L(m), B_s\cup H_{s+1})-\delta_{s+1}/2 \geqslant b_{s+1}-\delta_{s+1}.\end{align*} Here we have used the fact that $h$ is tail independent and $L(m)=L_0(p)\in B_s\cup H_{s+1}$.  Therefore $g(L(m), B_s\cup H_{s+1})\geqslant b_{s+1}-\delta_{s+1}$ for each $m\in F_{s+1}\setminus (\min F_{s+1})$ and $B_s\cup H_{s+1}\preceq H\in MAX(\mathcal{F}_\zeta[\mathcal{P}])$.  Therefore the $n=s+1$ case of item $(iv)$ will be satisfied for our eventual choice of $B_s\cup H_{s+1}\preceq H\in MAX(\mathcal{F}_\zeta[\mathcal{P}])$.   This completes the recursive construction.

Since $\mathcal{F}_\zeta$ is nice, this process must eventually terminate when $(\min H_n)_{n=1}^t\in MAX(\mathcal{F}_\zeta)$.  As in the previous paragraph, we let $B_s=\cup_{n=1}^s H_n$ for $1\leqslant s\leqslant t$.  For convenience, we let $B_0=\varnothing$. Item $(i)$ is satisfied as noted in the first paragraph of the proof. Items $(ii)$ and $(iv)$ were verified in the recursive construction.  Let $H=\cup_{n=1}^t H_n$. It follows from the construction that $H_1<\ldots <H_n$, since for each $1\leqslant s<t$, $H_s\subset B_s <L_{s+1}\supset H_{s+1}$.  Since $H_1<\ldots <H_t$, $H_n\in MAX(\mathcal{P})$, and $(\min H_n)_{n=1}^t\in MAX(\mathcal{F}_\zeta)$, $H=_\mathcal{P}\cup_{n=1}^t H_n\in MAX(\mathcal{F}_\zeta[\mathcal{P}])$.     Therefore item $(iii)$ is satisfied. As was noted in the construction, since $B_s\cup H_{s+1}\preceq H\in MAX(\mathcal{F}_\zeta[\mathcal{P}])$ for $s=0,1, \ldots, t-1$, item $(iv)$  is satisfied. By the permanence property together with tail independence of $h$ and the fact that $H_s\in MAX(\mathcal{P})\upp L_s$ and $H\in MAX(\mathcal{F}_\zeta[\mathcal{P}])\upp L_0$, $$D\leqslant \mathbb{E}^P_H h(\cdot, H) = \sum_{s=1}^t \mathbb{E}^P_{H_s}h (\cdot, \cup_{n=1}^s H_n) \leqslant \sum_{s=1}^t b_s.$$  Here we have used the fact that for $s=0, \ldots, t-1$ and any $G\in MAX(\mathcal{P})\upp L_{s+1}$, $$\mathbb{E}^P_G h(\cdot, \cup_{n=1}^{s+1}H_n) = \mathbb{E}^P_G h(\cdot, B_s\cup G) = f_{s+1}(G)\leqslant b_{s+1}.$$

It remains to show that $F=_\mathcal{Q}\cup_{n=1}^t\in \mathcal{F}_{\zeta+1}^M[\mathcal{Q}]$.     We note that by construction, $F_1<\ldots <F_t$ and $F_n\in MAX(\mathcal{Q})\upp M$, so $F=_\mathcal{Q}\cup_{n=1}^t F_n$.  Note that we can write $(\min F_n)_{n=1}^t= M(G)$ for some $G\in [\nn]^{<\omega}$.     Since $\mathcal{F}_\zeta$ is hereditary and $(\min H_n)_{n=1}^t\in \mathcal{F}_\zeta$, $(\min H_n)_{n=1}^{t-1}\in \mathcal{F}_\zeta$.    Since for each $1\leqslant n<t$, $M(\min H_n)<F_{n+1}$, $$M(G\setminus (\min G))=(\min F_n)_{n=2}^t=(\min F_{n+1})_{n=1}^{t-1}$$ is a spread of $M((\min H_n)_{n=1}^{t-1})$, so that $G\setminus (\min G)$ is a spread of $(\min H_n)_{n=1}^{t-1}\in \mathcal{F}_\zeta$.    Therefore $G\setminus (\min G)\in \mathcal{F}_\zeta$, and $G=(\min G)\smallfrown (G\setminus (\min G))\in \mathcal{F}_{\zeta+1}$.  Since $F_1<\ldots <F_t$, $F_n\in \mathcal{Q}\upp M$, and $(\min F_n)_{n=1}^t=M(G)\in \mathcal{F}_{\zeta+1}(M)$, $F\in \mathcal{F}_{\zeta+1}^M[\mathcal{Q}]$.

\end{proof}

The proof in the case that $(\mathfrak{Q}, \mathcal{Q})$ is a $0$-homogeneous (or, more accurately, the $0$-homogeneous) probability block is  easier.

\begin{theorem} Fix $0<\zeta<\omega_1$ and assume $(\mathfrak{P}, \mathcal{P})$ is a $\xi$-homogeneous probability block. Assume $L\in[\nn]$, $D\in\rr$, and  $g:L\ominus \mathcal{F}_\zeta[\mathcal{P}]\to \rr$ are such that $g$ is a bounded function and for every $F\in MAX(\mathcal{F}_\zeta[\mathcal{P}])\upp L$, $\mathbb{E}^P_F g(\cdot, F)\geqslant D$. Then for any sequence $(\delta_n)_{n=1}^\infty$ of positive numbers, any $M\in[\nn]$,  and the $0$-homogeneous probability block $(\mathfrak{Q}, \mathcal{Q})$, there exist $L\in [L]$, $F =(m_n)_{n=1}^t\in \mathcal{F}_{\zeta+1}(M),$ numbers $b_1, \ldots, b_t\in\rr$,  and  sets $H_1, \ldots, H_t\in \mathcal{P}$, $H\in MAX(\mathcal{F}_\zeta[\mathcal{P}])$, such that  \begin{enumerate}[(i)]\item $L(m_n)\in H_n$,  \item  $H=_\mathcal{P}\bigcup_{n=1}^t H_n$, \item for each $1\leqslant n\leqslant t$, $g(L(m_n), H) \geqslant b_n-\delta_n$, \item $\sum_{n=1}^t b_n\geqslant D$.   \end{enumerate}

\label{demigod2}

\end{theorem}

\begin{proof} By replacing $(\delta_n)_{n=1}^\infty$ with a sequence of smaller numbers if necessary, we may assume this sequence is decreasing.

Let $J\subset \rr$ be a compact interval containing the range of $g$.    By applying Proposition \ref{tail} with $\ee_n=\delta_n/2$, we can find $L_0\in [L(M)]$ and a bounded, tail independent function $h:L_0\ominus \mathcal{F}_\zeta[\mathcal{P}]\to\rr$ such that if $(L_0(n), F)\in L_0 \ominus \mathcal{F}_\zeta[\mathcal{P}]$, $$g(L_0(n), F)\leqslant h(L_0(n), F)\leqslant g(L_0(n), F)+\delta_n/2.$$   Since $MAX(\mathcal{F}_\zeta[\mathcal{P}])\upp L_0\subset MAX(\mathcal{F}_\zeta[\mathcal{P}])$ and $g|_{L_0\ominus \mathcal{F}_\zeta[\mathcal{P}]}\leqslant h|_{L_0\ominus \mathcal{F}_\zeta[\mathcal{P}]}$, it follows that for any $H\in MAX(\mathcal{F}_\zeta[\mathcal{P}])\upp L_0$, $$D\leqslant \mathbb{E}^P_H g(\cdot, H)\leqslant \mathbb{E}^P_H h(\cdot, H).$$    Since $L_0\in [L(M)]$, it follows that $L_0=L(M(T))$ for some $T\in[\nn]$.   Let $K_0=L_0(M(T))\in [L(M(T))]$.

Since $h$ is tail independent, we may consider the natural extension $h:K_0\oplus \mathcal{F}_\zeta[\mathcal{P}]\to J$.  Define $f_1:MAX(\mathcal{P})\upp K_0\to J$ by letting $\eta_1(F)=\mathbb{E}^P_F h(\cdot,F)$.  This is well defined because we have taken the natural extension. As noted in Remark \ref{prremark}, we may fix $a_1, b_1\in\rr$ such that $b_1-a_1<\delta_1/2$ and $L_1\in [K_0]$ such that $f_1(F)\in [a_1, b_1]$ for all $F\in MAX(\mathcal{P})\upp L_1$.   Fix $H_1\in MAX(\mathcal{P})\upp L_1$.     Since $\mathbb{E}^P_{H_1}h(\cdot, H_1)\geqslant a_1\geqslant b_1-\delta_1/2$, there exists $n_1\in H_1$ such that $h(n_1, H_1)\geqslant b_1-\delta_1/2$.  Since $L_1\subset L(M)$, there exists $m_1\in M$ such that $n_1=L(m_1)$, and $h(L(m_1), H_1)=h(n_1, H_1)\geqslant b_1-\delta_1/2$.    Since $n_1\in L_1\subset L_0$, $n_1=L(M(T(p))=L_0(p)$ for some $1\leqslant p\in\nn$.    Then for any $H_1\preceq H\in MAX(\mathcal{F}_\zeta[\mathcal{P}])$, \begin{align*} g(L(m_1), H) & = g(L_0(p), H) \geqslant h(L_0(p), H) -\delta_1/2= h(L(m_1),H)=h(L(m_1), H_1)-\delta_1/2 \\ & \geqslant b_1-\delta_1.\end{align*}    

Assume that $m_1<\ldots <m_s$, $H_1<\ldots <H_s$, $L_1, \ldots, L_s$ have been chosen such that $(\min H_n)_{n=1}^s\in \mathcal{F}_\zeta$. If $(\min H_n)_{n=1}^s\in MAX(\mathcal{F}_\zeta[\mathcal{P}])$, we are done.  Otherwise we perform the following recursive step.    Let $B_s=\cup_{n=1}^s H_s$ and fix $L_{s+1}\in [L_s]$ such that $B_s<L_{s+1}$. Since $(\min H_n)_{n=1}^s\in \mathcal{F}_\zeta\setminus MAX(\mathcal{F}_\zeta)$, $B_s\cup F\in \mathcal{F}_\zeta[\mathcal{P}]$ for all $B_s<F\in \mathcal{P}$.    By again using Remark \ref{prremark}, passing to a subset and relabeling if necessary, we may also assume that for some $a_{s+1}<b_{s+1}$ with $b_{s+1}-a_{s+1}<\delta_{s+1}/2$, $\mathbb{E}^P_F h(\cdot, B_s\cup F)\in [a_{s+1}, b_{s+1}]$ for all $F\in MAX(\mathcal{P})\upp L_{s+1}$.   Fix $H_{s+1}\in MAX(\mathcal{P})\upp L_{s+1}$ and $n_{s+1}\in H_{s+1}$ such that $h(n_{s+1}, B_s\cup H_{s+1})\geqslant a_{s+1}\geqslant b_{s+1}-\delta_{s+1}/2$.    Since $n_{s+1}\in L_{s+1}\subset L(M)$, $n_{s+1}=L(m_{s+1})$ for some $m_{s+1}\in M$. Since $L(m_s)<L_{s+1}(1)\leqslant n_{s+1}=L(m_{s+1})$, it follows that $m_{s+1}>m_s$.  Since $m_1<\ldots <m_{s+1}$ and $m_n\in L_0$ for $n=1,\ldots, s+1$, it follows that $m_{s+1}=L_0(p)$ for some $p\geqslant s+1$.    Therefore for any $B_s\cup H_{s+1}\preceq H\in MAX(\mathcal{F}_\zeta[\mathcal{P}])$,  \begin{align*} g(L(m_{s+1}), H) & = g(L_0(p), H) \geqslant h(L_0(p), H) -\delta_{s+1}/2= h(L(m_{s+1}),H)=h(L(m_{s+1}), H_1)-\delta_{s+1}/2 \\ & \geqslant b_{s+1}-\delta_{s+1}.\end{align*}  Since $L(M(\min H_s))<L_{s+1}\leqslant n_{s+1}=L(m_{s+1})$, it follows that $M(\min H_s)<m_{s+1}$.    This completes the recursive step. The details are checked as in the proof of Theorem \ref{god2}.

\end{proof}

\section{Theorems \ref{main1} and  \ref{main2}}\label{full}

In this section, let $X$ be  a Banach space. Let $R$ be a subsequential space on $X$. We recall that this means that $c_{00}(X)\subset R\subset \ell_\infty(X)$, $R$ is a (not necessarily closed) subspace of $\ell_\infty(X)$ on which there exists a norm $r$ such that $(R, r)$ is a Banach space and, with $B_R=\{\varsigma\in R: r\leqslant 1\}$, $B_R\subset B_{\ell_\infty(X)}$ and if $\varsigma\in B_R$, then every subsequence of $\varsigma$ also lies in $B_R$.  We also fix a bimonotone norm $s$ on $c_{00}(X)$, which means that for each $x\in X$ and $n\in\nn$, $s(x\otimes e_n)=\|x\|$ and for any $\varsigma =(x_n)_{n=1}^\infty\in c_{00}(X)$, $$s=\sup_{l\leqslant m} s\Bigl(\sum_{n=l}^m x_n\otimes e_n\Bigr)=\lim_m s\Bigl(\sum_{n=1}^m x_n\otimes e_n\Bigr).$$  The \emph{natural domain} of a bimonotone norm is defined to be the space $S$ of all sequences $(x_n)_{n=1}^\infty\in \ell_\infty(X)$ such that $\sup_m s(\sum_{n=1}^m x_n\otimes e_n)<\infty$. The norm $s$ naturally extends by the formula $$s((x_n)_{n=1}^\infty)=\sup_m s\Bigl(\sum_{n=1}^m x_n\otimes e_n\Bigr)$$ to all of $S$, and $(S, s)$ is a Banach space.

Throughout this section, $0\leqslant \xi<\omega_1$ and   $P=(\mathfrak{P}, \mathcal{P})$ is a fixed, $\xi$-homogeneous probability block.  When $\mathbb{E}_M$ is written with no superscript, it is understood that the convex block sequence is taken with respect to this probability block $P=(\mathfrak{P}, \mathcal{P})$. If we wish to consider convex blocks coming from some other probability block $Q=(\mathfrak{Q}, \mathcal{Q})$, we include the superscripts $\mathbb{E}^P_M$ and $\mathbb{E}^Q_M$ to distinguish.

We will prove the following, further quantified theorem, and then deduce a large part of Theorem \ref{main2} and \ref{main1} as special cases. We now state this further quantified theorem, for which  recall that if $\mathcal{F}$ is regular and $M\in[\nn]$, $M|\mathcal{F}$ denotes the maximal initial segment of $M$ which is a member of $\mathcal{F}$. We also agree to the convention that $M|[\nn]^{<\omega}=M$.   

\begin{theorem} Fix $\zeta\leqslant \omega_1$, $\xi<\omega_1$, and let $P=(\mathfrak{P}, \mathcal{P})$ be a $\xi$-homogeneous probability block.  Let $X$ be a Banach space, $(R, \|\cdot\|_R)$ a subsequential space on $X$, and $s$ a bimonotone norm on $c_{00}(X)$ with natural domain $S$.  The following are equivalent. \begin{enumerate}\item For every $\varsigma\in R$, there exist $M\in[\nn]$ and a constant $C$ such that for every $N\in [M]$, $\mathbb{E}_{N|\mathcal{F}_\zeta[\mathcal{P}]}\varsigma\in CB_S$. \item There exists a constant $C$ such that for every $\varsigma\in B_R$, there exists $M\in [\nn]$ such that for every $N\in [M]$, $\mathbb{E}_{N|\mathcal{F}_\zeta[\mathcal{P}]}\varsigma\in CB_S$. \item For every $\varsigma\in R$ and $L\in [\nn]$, there exist $M\in[L]$ and a constant $C$ such that for all $N\in[M]$, $\mathbb{E}_{N|\mathcal{F}_\zeta[\mathcal{P}]} \varsigma\in CB_S$.  \item There exists a constant $C$ such that for every $\varsigma\in B_R$ and every $L\in [\nn]$, there exists $M\in [L]$ such that for all $N\in [M]$, $\mathbb{E}_{N|\mathcal{F}_\zeta[\mathcal{P}]}\varsigma\in CB_S$.  \end{enumerate}

\label{main3}
\end{theorem}

We also include here the relationship between these properties for two different probability blocks. 

\begin{theorem} Let $\zeta\leqslant \omega_1$ be a limit ordinal and fix $\xi\leqslant \upsilon<\omega_1$.  Let $P=(\mathfrak{P}, \mathcal{P})$ be a $\xi$-homogeneous probability block and suppose that $R,S$ and $(\mathfrak{P}, \mathcal{P})$ satisfy item $(4)$ of Theorem \ref{main3} with constant $C$. If $Q=(\mathfrak{Q}, \mathcal{Q})$ is any $\upsilon$-homogeneous probability block, then for any $C'>C$,  $R,S$ and $(\mathfrak{Q}, \mathcal{Q})$ satisfy item $(4)$ of Theorem \ref{main3} with constant $C'$.

\label{diff}

\end{theorem}

We will prove Theorem \ref{main3} by completing the implications $$(1)\Rightarrow (3)\Rightarrow (4)\Rightarrow (2)\Rightarrow (1).$$  It is obvious that $(4)\Rightarrow (2)\Rightarrow (1)$, from which it follows that we need only to prove $(1)\Rightarrow (3)\Rightarrow (4)$ to complete the proof of Theorem \ref{main3}.

Let us  note that conditions $(1), (2), (3),$ and $(4)$ of Theorem \ref{main3}, in the special case that $\zeta=\omega_1$, are respectively equivalent to $(2),(3), (5),$ and $(6)$, of Theorem \ref{main2}.  We will complete Theorem \ref{main2} by completing the implications $$(1)\Rightarrow(2) \Rightarrow(5)\Rightarrow(6)\Rightarrow(3)\Rightarrow (1)$$ and $$(5)\Rightarrow (4)\Rightarrow (1).$$  In Theorem \ref{main2}, the implications $(6)\Rightarrow (3)\Rightarrow (1)$ and $(5)\Rightarrow (4)\Rightarrow (1)$ are clear. Moreover, since conditions $(2)$, $(5)$, and $(6)$ of Theorem \ref{main2} are the $\zeta=\omega_1$ cases, respectively,  of conditions $(1)$, $(3)$, and $(4)$ in Theorem \ref{main3},  the implications $(2)\Rightarrow (5)\Rightarrow (6)$ in Theorem \ref{main2} will be special cases of the implications $(1)\Rightarrow (3)\Rightarrow (4)$ in Theorem \ref{main3}.   Therefore, once we complete the proof of Theorem \ref{main3}, the only implication remaining to complete the proof of Theorem \ref{main2} will be $(1)\Rightarrow (2)$, which is Proposition \ref{j}.  Let us also discuss why Theorem \ref{main2} has six conditions while Theorem \ref{main3} only has four.  Theorem \ref{main2} contains the $\zeta=\omega_1$ case of Theorem \ref{main3}, as well as two additional conditions, $(1)$ and $(4)$, the analogues of which do not appear for the $\zeta<\omega_1$ case. The reason for this is because, since $M|\mathcal{F}_\zeta[\mathcal{P}]$ is a finite set for any $M\in[\nn]$, any nice $\mathcal{P}$, and any countable $\zeta$, it follows that for any $\varsigma\in \ell_\infty(X)$ and any such $M$, $\mathcal{P}$, $\zeta$, $\mathbb{E}_{M|\mathcal{F}_\zeta[\mathcal{P}]}\varsigma \in c_{00}(X)\subset S$.

We also note that Theorem \ref{main1} is a special case of Theorem \ref{main2}. Namely, Theorem \ref{main2} when $P$ is the Dirac probability block is precisely Theorem \ref{main1}.

The proof of the next proposition is an application of the Ramsey theorem similar to an unpublished result of Johnson \cite{O}.

\begin{proposition} If for every $\varsigma\in R$, there exists $M\in[\nn]$ such that for all $N\in[M]$, $\mathbb{E}_N\varsigma\in S$, then for any $\varsigma\in R$, there exist $M\in[\nn]$ and $C>0$ such that for all $N\in[M]$, $\mathbb{E}_N\varsigma\in CB_S$. 

\label{j}

\end{proposition}

\begin{proof} By homogeneity, it is sufficient to prove the result for each $\varsigma\in B_R\subset B_{\ell_\infty(X)}$. Suppose that for some $\varsigma\in B_R$, no such $M$ and $C$ exist. For each $p,q\in\nn$, define $$\mathcal{V}_{q,p}=\{M\in[\nn]: s(\sum_{n=1}^p \mathbb{E}_M \varsigma(n)\otimes e_n)\leqslant 2q\}.$$   Since for each $n\in\nn$,  $M\mapsto \mathbb{E}_M \varsigma(n)$ is locally constant, it follows that $\mathcal{V}_{p,q}$ is a closed set, as is $$\mathcal{V}_q:=\bigcap_{p=1}^\infty \mathcal{V}_{p,q}.$$  Moreover, note that by the properties of $s$ and its definition on the natural domain $S$, for $N\in[\nn]$, $\mathbb{E}_N \varsigma\in 2qB_S$ if and only if $N\in \mathcal{V}_q$.   By hypothesis, there exists some $M_0\in [\nn]$ such that for all $N\in[M_0]$, $\mathbb{E}_N\varsigma\in S$. Since $\mathcal{V}_q$ is closed, we can apply Theorem \ref{Ramsey} recursively  to select $M_1\supset M_1\supset \ldots$ such that for each $q\in\nn$, either $[M_q]\subset \mathcal{V}_q$ or $[M_q]\cap \mathcal{V}_q=\varnothing$.  We note that for each $q$, the second alternative must occur. Indeed, if $[M_q]\subset \mathcal{V}_q$, then we set $M=M_q$ and $C=2q$ to obtain $M$ and $C$ as in the conclusion of the proposition. This contradicts our hypothesis that for this $\varsigma\in B_R$, no such $M$ and $C$ exist.   Therefore we have $M_0\supset M_1\supset M_2\supset \ldots$ such that for each $q\in\nn$, $[M_q]\cap \mathcal{V}_q=\varnothing$.

For each $q\in \nn$, let $I_q=\cup_{n=1}^{q-1}\text{supp}(\mathbb{P}_{M_q, n})$. Select $N(1)<N(2)<\ldots$ such that for each $q\in\nn$, $I_q<N(q)\in M_q$. Then $N\in[M_0]$, from which it follows that $\mathbb{E}_N\varsigma\in S$. Therefore there exists some $C\in \nn$ such that $\mathbb{E}_N\varsigma\in CB_S$. Fix $C<q\in\nn$ and define $L=I_q\cup (N\setminus \cup_{n=1}^{q-1} \text{supp}(\mathbb{P}_{N,n}))\in [M_q]$.  By the permanence property, $\mathbb{P}_{L,n}=\mathbb{P}_{N,n}$ for all $n\geqslant q$. Since we have assumed $\varsigma\in B_R\subset B_{\ell_\infty(X)}$, it follows that $\|\mathbb{E}_L\varsigma(n)\|\leqslant 1$ for all $n\in\nn$.  Since $L\in[M_q]$, there exists $p\in\nn$, which by the properties of $s$ we may assume exceeds $q$, such that $s(\sum_{n=1}^p \mathbb{E}_L\varsigma(n)\otimes e_n)>2q$.    Then \begin{align*} C & \geqslant s(\sum_{n=1}^p \mathbb{E}_N \varsigma(n)\otimes e_n) \geqslant s(\sum_{n=q}^p \mathbb{E}_N\varsigma(n)\otimes e_n) = s(\sum_{n=q}^p \mathbb{E}_L\varsigma(n)\otimes e_n) \\ & \geqslant s(\sum_{n=1}^p \mathbb{E}_L \varsigma(n)\otimes e_n)- s(\sum_{n=1}^{q-1} \mathbb{E}_L\varsigma(n)\otimes e_n) \\ & \geqslant  s(\sum_{n=1}^p \mathbb{E}_L\varsigma(n)\otimes e_n) - \sum_{n=1}^{q-1} \|\mathbb{E}_L\varsigma(n)\| \\ & >2q-q =q>C.\end{align*} This contradiction finishes the proof.

\end{proof}

For our proof of the implications $(1)\Rightarrow (3)\Rightarrow (4)$ of Theorem \ref{main3}, we will define a transfinite way of measuring the failure of uniform upper estimates.  Our next few technical definitions lay the necessary groundwork for accomplishing this. For each $k<\omega$, define the bimonotone norm $s_k$ on $c_{00}(X)$ by $$s_k\Bigl(\sum_{n=1}^\infty x_n\otimes e_n\Bigr)=s\Bigl(\sum_{n=1}^\infty x_n\otimes e_{n+k}\Bigr)$$ and let $S_k$ be the natural domain of $s_k$.

The definitions in the following paragraphs are made, and depend upon, our fixed probability block $P$. These definitions can be made with respect to any probability block, and later we will briefly wish to consider these notions for another probability block $Q=(\mathfrak{Q}, \mathcal{Q})$. When necessary, we will include in our notation a reference to the underlying probability block.

For $k<\omega$, $0\leqslant C\leqslant \infty$, and a sequence $\varsigma\in R$, let us say a finite subset $G$ of $\nn$ is $(k, C, \varsigma)$-\emph{good} provided that for any $F=_\mathcal{P}\cup_{n=1}^t F_n\subset G$, $s_k(\mathbb{E}_F \varsigma)\leqslant C$.  In the case that $P$ is the Dirac probability block, this is equivalent to the condition that, if $\varsigma=(x_n)_{n=1}^\infty$, then for any $(m_1, \ldots, m_t)\subset G$, $s_k(\sum_{n=1}^t x_{m_n}\otimes e_n)\leqslant C$.    Obviously any subset of a $(k, C, \varsigma)$-good set is also $(k, C, \varsigma)$-good.    Let $\mathfrak{G}(k,C, \varsigma)$ be the set of all $(k,C, \varsigma)$-good sets. Since any subset of a $(k, C, \varsigma)$-good set is also $(k, C, \varsigma)$-good, it follows that $\mathfrak{G}(k, C, \varsigma)$ is hereditary. 

For $\varsigma\in R$, $M\in[\nn]$, $\zeta\leqslant \omega_1$, $k<\omega$, and $0\leqslant C\leqslant \infty$, let us say that $(\varsigma, M)$ is $(\zeta, k, C)$-\emph{stable} if $\mathcal{F}_\zeta^M[\mathcal{P}]\subset \mathfrak{G}(k,C ,\varsigma)$. In the case that $P$ is the Dirac probability block, this is equivalent to the condition that $\mathcal{F}_\zeta(M)\subset \mathfrak{G}(k,C, \varsigma)$.  Let us say that $(\varsigma, L)$ is $(\zeta, k, C)$-\emph{Ramsey} if for any $M\in[L]$, there exists $N\in [M]$ such that $(\varsigma, N)$ is $(\zeta, k, C)$-stable.   For $\zeta\leqslant \omega_1$, $k<\omega$, $\varsigma\in B_R$, and  $L\in[\nn]$, let $\Gamma(\zeta, k, \varsigma, L)$ denote the infimum of $C>0$ such that $(\varsigma, L)$ is $(\zeta, k, C)$-Ramsey.   We recall the convention that $\inf \varnothing=\infty$, so that $\Gamma(\zeta, k, \varsigma, L)=\infty$ if no such $C$ exists.  

\begin{rem}\upshape \begin{enumerate}[(i)]\item For $\varsigma\in R$,  $(\varsigma, \nn)$ is $(\zeta, k, C)$-Ramsey if and only if for any $L\in[\nn]$, there exists $M\in[L]$ such that $(\varsigma, M)$ is $(\zeta, k ,C)$-stable (equivalently, for any $L\in[\nn]$, there exists $M\in[L]$ such that $\mathcal{F}_\zeta^M[\mathcal{P}]\subset \mathfrak{G}(k, C, \varsigma)$).   \item For $\varsigma\in R$, $(\varsigma, M)$ is $(\omega_1, k, C )$-stable if and only if $[M]^{<\omega}\subset \mathfrak{G}(\zeta, k, C)$ if and only if $\mathbb{E}_N\varsigma\in CB_{S_k}$ for all $N\in [M]$.    Similarly, $(\varsigma, \nn)$ is $(\omega_1, k, C)$-Ramsey if and only if for any $L\in [\nn]$, there exists $M\in [L]$ such that $[M]^{<\omega}\subset \mathfrak{G}(k, C, \varsigma)$ if and only if for any $L\in[\nn]$, there exists $M\in [L]$ such that for all $N\in [M]$, $\mathbb{E}_N\varsigma\in CB_{S_k}$.  \item If $(\varsigma, L)$ is $(\zeta, k, C)$-stable, so is $(\varsigma, M)$ for any $M\in[L]$, since $\mathcal{F}^M_\zeta[\mathcal{P}]\subset \mathcal{F}^L_\zeta[\mathcal{P}]$ when $M\in [L]$.   Similarly, if $(\varsigma, L)$ is $(\zeta, k, C)$-Ramsey, so is $(\varsigma, M)$ for any $M\in [L]$.  \item Since $(\varsigma, L)$ being $(\zeta, k, C)$-Ramsey implies that $(\varsigma, M)$ is $(\zeta, k, C)$-Ramsey for any $M\in [L]$, it follows that $$\Gamma(\zeta, k, \varsigma, \nn)=\sup_{L\in [\nn]}\Gamma(\zeta, k, \varsigma, L).$$ \end{enumerate}

\label{good remark}
\end{rem}

Note that our function $\Gamma$ takes four arguments: An ordinal $\zeta\leqslant \omega_1$, a non-negative integers $k<\omega$, a sequence $\varsigma\in B_R$, and a set $L\in[\nn]$. For brevity, we adopt the convention that if $\Gamma$ appears with one or more of these four arguments missing, this means that we have taken the supremum over the appropriate set of the missing arguments. For example,  $$\Gamma(\zeta, k, \varsigma)=\sup_{L\in [\nn]} \Gamma(\zeta, k, \varsigma, L),$$ $$\Gamma(\zeta)=\sup_{k<\omega}\sup_{\varsigma\in B_R}\sup_{L\in[\nn]} \Gamma(\zeta, k, \varsigma, L),$$ $$\Gamma=\sup_{\zeta\leqslant \omega_1} \sup_{k<\omega} \sup_{\varsigma\in B_R} \sup_{L\in[\nn]} \Gamma(\zeta, k, \varsigma, L),$$ etc.  We write $\Gamma^P$ and $\Gamma^Q$ if we need to distinguish between these functions defined for different probability blocks $P,Q$.

  We next deduce several properties of these functions.

\begin{lemma} \begin{enumerate}[(i)]\item For a fixed $k<\omega$,  $\varsigma\in B_R$, and $M_0\in [\nn]$,  the function $\zeta\mapsto \Gamma(\zeta, k, \varsigma,M_0)$ is non-decreasing from $\omega_1+$ to $[0,\infty]$. This implies that for fixed $k<\omega$ and $\varsigma\in B_R$, the functions $\zeta\mapsto \Gamma(\zeta, k, \varsigma)$,  $\zeta\mapsto \Gamma(\zeta, k)$, and $\zeta\mapsto \Gamma(\zeta)$ are non-decreasing.  

\item For a fixed $k<\omega$,  $\varsigma\in B_R$, and $M_0\in [\nn]$,  the function $\zeta\mapsto \Gamma(\zeta, k, \varsigma, M_0)$ is continuous from $\omega_1+$ to $[0,\infty].$ Therefore for a fixed $k<\omega$ and $\varsigma\in B_R$, the functions $\zeta\mapsto \Gamma(\zeta, k, \varsigma)$, $\zeta\mapsto \Gamma(\zeta, k)$, and $\zeta\mapsto \Gamma$ are continuous.

\item $\Gamma(0)=0$. 

\item For each $k,p<\omega$,  $\varsigma\in B_R$,  and $\zeta<\omega_1$, $\Gamma(\zeta, k+p, \varsigma)\leqslant \Gamma(\zeta+p, k, \varsigma)\leqslant p+\Gamma(\zeta, k+p, \varsigma)$. Moreover, for any $\zeta<\omega_1$ and $k,p<\omega$, $\Gamma(\zeta, k+p)\leqslant \Gamma(\zeta+p, k) \leqslant p+\Gamma(\zeta, k+p) $ and $\Gamma(\zeta+p) \leqslant p+\Gamma(\zeta)$. 

\item For any countable $\zeta$ and $p<\omega$, $\Gamma(\zeta)\leqslant n+p+\Gamma(\zeta,p),$ where $\zeta=\beta+n$ is the unique expression of $\zeta$ as a non-successor  ordinal $\beta$ and $n<\omega$. In particular, if $\beta\leqslant \omega_1$ is a  limit ordinal, $\Gamma(\beta)=\Gamma(\beta, 0)$.  \item For any ordinal $\zeta\leqslant \omega_1$, $\Gamma(\zeta)<\infty$ if and only if $\Gamma(\zeta, 0)<\infty$.  \end{enumerate}

\label{work2}
\end{lemma}

\begin{proof} In all parts of the proof, an inequality is trivial if the majorizing quantity is infinite. Therefore we prove the inequalities only in the cases that the majorizing quantity is finite. 

$(i)$ Fix $k<\omega$, $\mu\leqslant \zeta\leqslant \omega_1$, $\varsigma\in B_R$, $M_0\in [\nn]$, and assume $\Gamma(\zeta, k, \varsigma,M_0)<C$. We must show that for any $L\in[M_0]$, there exists $M\in [L]$ such that $\mathcal{F}^M_\mu[\mathcal{P}]\subset \mathfrak{G}(k, C, \varsigma)$.  But we know that for any such $L$, there exists $N\in [L]$ such that $\mathcal{F}^N_\zeta[\mathcal{P}]\subset \mathfrak{G}(k, C, \varsigma)$. By the almost monotone property of the fine Schreier families, there exists $l\in \nn$ such that $l<H\in \mathcal{F}_\mu$ implies $H\in \mathcal{F}_\zeta$.  Let  $M(n)=N(l+n)$ for all $n\in\nn$. Then $M\in [L]$ and $$\mathcal{F}^M_\mu[\mathcal{P}]\subset \mathcal{F}_\zeta^N[\mathcal{P}]\subset \mathfrak{G}(k, C, \varsigma).$$ The second inclusion follows from the properties of $N$.   To see the first inclusion, fix $$F=\bigcup_{n=1}^t F_n$$ with $\varnothing \neq F_1<\ldots <F_t$, $F_n\in \mathcal{P}\upp M\subset \mathcal{P}\upp N$, and $$(\min F_n)_{n=1}^t = (M(n))_{n\in G}$$ for some $G\in \mathcal{F}_\mu$. Then with $H=(l+n: n\in G)$, $l<H\in \mathcal{F}_\mu$, so $H\in \mathcal{F}_\zeta$.  Moreover, $(\min F_n)_{n=1}^t = (M(n))_{n\in G}=(N(l+n))_{n\in G}=(N(n))_{n\in H}\in \mathcal{F}_\zeta(N)$. From this it follows that $F=\cup_{n=1}^t F_n\in \mathcal{F}_\zeta^N[\mathcal{P}]$.    Thus we have shown that for any $L\in[M_0]$, there exists $M\in[L]$ such that $\mathcal{F}^M_\mu[\mathcal{P}]\subset \mathfrak{G}(k, C, \varsigma)$, from which it follows that $\Gamma(\mu, k ,\varsigma, M_0)\leqslant C$.   Since this holds for any $C>\Gamma(\zeta, k, \varsigma,M_0)$, $\Gamma(\mu, k, \varsigma,M_0)\leqslant \Gamma(\zeta, k, \varsigma,M_0)$.  The remaining parts of $(i)$ follow from taking the appropriate suprema.

$(ii)$ Fix $k<\omega$, $\varsigma\in B_R$, and $M_0\in [\nn]$. By Proposition \ref{ord1}, we must show that for any limit ordinal $\zeta\leqslant \omega_1$, if $\sup_{\mu<\zeta}\Gamma(\mu, k, \varsigma,M_0)<C$, then $\Gamma(\zeta, k, \varsigma,M_0)\leqslant C$.  We do this in two cases.   First suppose that $\zeta<\omega_1$ and let $\zeta_n\uparrow \zeta$ be such that $$\mathcal{F}_\zeta=\{E: \exists n\leqslant E\in \mathcal{F}_{\zeta_n}\}.$$   Assume $\sup_{\mu<\zeta}\Gamma(\mu, k, \varsigma,M_0)<C$ and fix $L\in[M_0]$.   As noted in Remark \ref{good remark}, we may find $L\supset M_1\supset M_2\supset \ldots$ such that for all $n\in\nn$, $\mathcal{F}_{\zeta_n}(M_n)\subset \mathfrak{G}(k, C, \varsigma)$.  Let $M(n)=M_n(n)$ and note that $$\mathcal{F}_\zeta^M[\mathcal{P}]\subset \bigcup_{n=1}^\infty \mathcal{F}_{\zeta_n}^{M_n}[\mathcal{P}]\subset \mathfrak{G}(k, C, \varsigma)$$ by Proposition \ref{silly}. Since $L\in [M_0]$ was arbitrary, $\Gamma(\zeta, k, \varsigma,M_0)\leqslant C$.   This completes the $\zeta<\omega_1$ case. 

Now suppose that $\zeta=\omega_1$ and assume that $\sup_{\zeta<\omega_1}\Gamma(\zeta, k, \varsigma,M_0)<C$.  Fix $L\in [M_0]$.    For any $\zeta<\omega_1$, there exists $M\in [L]$ such that $\mathcal{F}_\zeta^M[\mathcal{P}]\subset \mathfrak{G}(k, C, \varsigma)$, from which it follows that $$\text{rank}(\mathfrak{G}(k, C, \varsigma)\upp L)>\text{rank}(\mathcal{F}_\zeta^M[\mathcal{P}])=\omega^\xi \zeta+1.$$  Since this holds for any $\zeta<\omega_1$, $\text{rank}(\mathfrak{G}(k, C, \varsigma)\upp L)=\infty$. By Proposition \ref{tree}, $\mathfrak{G}(k, C, \varsigma)\upp L$ is ill-founded. From this it follows that there exists $M\in[\nn]$ such that for all $t\in \nn$, $(M(n))_{n=1}^t \in \mathfrak{G}(k, C, \varsigma)\upp L$.   This obviously implies that $M\in [L]$. Moreover, since $\mathfrak{G}(k,C, \varsigma)\upp L$ is hereditary and contains all initial segments of $M$, it contains all subsets of $M$.  This means that $\mathcal{F}_{\omega_1}^M[\mathcal{P}]= [M]^{<\omega}\subset \mathfrak{G}(k, C, \varsigma)$. Since $L\in[M_0]$ was arbitrary, this yields that $\Gamma(\omega_1, k, \varsigma,M_0)\leqslant C$.   The remainder of $(ii)$ follows from taking the appropriate suprema.

$(iii)$  This follows from the fact that $\mathcal{F}_0=\{\varnothing\}$ and $\varnothing$ is vacuously $\mathfrak{G}(k, 0, \varsigma)$-good for all $k<\omega$ and $\varsigma\in R$.

$(iv)$ Fix $\varsigma\in B_R$ and  suppose that $\Gamma(\zeta+p, k, \varsigma)<C$. Fix $L\in[\nn]$.  We may fix $N\in [L]$ such that $\mathcal{F}^N_{\zeta+p}[\mathcal{P}]\subset \mathfrak{G}(k, C, \varsigma)$.   Let $I= \cup_{n=1}^p \text{supp}(\mathbb{P}_{N,n})$ and $M=N\setminus I$. We claim that $\mathcal{F}^M_\zeta[\mathcal{P}]\subset \mathfrak{G}(k+p, C, \varsigma)$.   Indeed, fix $G\in \mathcal{F}^M_\zeta[\mathcal{P}]$ and $F=_\mathcal{P}\cup_{n=1}^t F_n\subset G$.  Then $$I\cup F=_\mathcal{P} \Bigl(\bigcup_{n=1}^p \text{supp}(\mathbb{P}_{N,n})\Bigr)\cup \Bigl(\bigcup_{n=1}^t F_n \Bigr)\in \mathcal{F}^N_{\zeta+p}[\mathcal{P}] \subset \mathfrak{G}(k, C, \varsigma)$$ and \begin{align*} s_{k+p}\Bigl(\sum_{n=1}^t \mathbb{E}_F \varsigma(n)\otimes e_n\Bigr) &  = s_{k+p}\Bigl(\sum_{n=1}^t \mathbb{E}_{I\cup F} \varsigma(n+p)\otimes e_n\Bigr) =s_k\Bigl(\sum_{n=1}^t \mathbb{E}_{I\cup F} \varsigma(n+p)\otimes e_{n+p}\Bigr) \\ & = s_k\Bigl(\sum_{n=1+p}^{t+p} \mathbb{E}_{I\cup F}\varsigma(n)\otimes e_n\Bigr) \leqslant s_k\Bigl(\sum_{n=1}^{t+p} \mathbb{E}_{I\cup F}\varsigma(n)\otimes e_n\Bigr) \leqslant C.  \end{align*}    Therefore for any $L\in[\nn]$, there exists $M\in [L]$ such that $\mathcal{F}^M_\zeta[\mathcal{P}]\subset \mathfrak{G}(k+p, C, \varsigma)$.  This yields that $\Gamma(\zeta, k+p, \varsigma) \leqslant C$, and since $C>\Gamma(\zeta+p, k, \varsigma)$ was arbitrary, $\Gamma(\zeta, k+p, \varsigma)\leqslant \Gamma(\zeta+p, k, \varsigma)$.    

Now suppose that $\Gamma(\zeta, k+p, \varsigma)<C$ and fix $L\in[\nn]$. We may fix $M\in [L]$ such that $\mathcal{F}^M_\zeta[\mathcal{P}]\subset \mathfrak{G}(k+p, C, \varsigma)$.  We claim that $\mathcal{F}^M_{\zeta+p}[\mathcal{P}]\subset \mathfrak{G}(k,p+ C, \varsigma)$, which will finish the proof of the first inequality of $(iv)$, since $L\in[\nn]$ and $C>\Gamma(\zeta, k+p, \varsigma)$ are arbitrary.    Fix $G\in \mathcal{F}^M_{\zeta+p}[\mathcal{P}]$ and $F=_\mathcal{P}\cup_{n=1}^t F_n \subset G $.    If $t\leqslant p$, then $$s_k\Bigl(\sum_{n=1}^t \mathbb{E}_F \varsigma(n)\otimes e_n\Bigr)\leqslant \sum_{n=1}^t \|\mathbb{E}_F\varsigma(n)\|\leqslant t\leqslant p\leqslant p+C.$$   If $t>p$, let $H_n=F_{n+p}$ for each $1\leqslant n\leqslant t-p$.   Let $H=\cup_{n=1}^{t-p} H_n\in \mathcal{F}^M_\zeta[\mathcal{P}]$ and note that $H=_\mathcal{P}\cup_{n=1}^{t-p}H_n$ and  \begin{align*} s_k\Bigl(\sum_{n=1}^t \mathbb{E}_F \varsigma(n)\otimes e_n\Bigr) & =s_k\Bigl(\sum_{n=1}^p \mathbb{E}_F \varsigma(n)\otimes e_n+ \sum_{n=p+1}^t \mathbb{E}_F \varsigma(n)\otimes e_n\Bigr) \\ & \leqslant \sum_{n=1}^p \|\mathbb{E}_F \varsigma(n)\| + s_k\Bigl(\sum_{n=p+1}^t \mathbb{E}_F \varsigma(n)\otimes e_n\Bigr)\\ &  \leqslant p+ s_k\Bigl(\sum_{n=p+1}^t \mathbb{E}_F \varsigma(n)\otimes e_n\Bigr) \\ & =p+s_k\Bigl(\sum_{n=1}^{t-p} \mathbb{E}_H \varsigma(n)\otimes e_{n+p}\Bigr) \\ & = p+s_{k+p}\Bigl(\sum_{n=1}^{t-p} \mathbb{E}_H \varsigma(n)\otimes e_n\Bigr)\leqslant p+C. \end{align*} 

The remainder of $(iv)$ follows from taking the appropriate suprema.

$(v)$  Let $\zeta=\beta+n$ be the unique expression of $\zeta$ as the sum of a non-successor ordinal $\beta$ and $n<\omega$.    If $\beta=0$, then $\zeta=n$.  In this case, by $(iii)$ and $(iv)$, $\Gamma(\zeta)\leqslant n+\Gamma(0)=n<\infty$.  Now assume $\beta>0$, which means $\zeta\geqslant \omega$.    For $q<\omega$, if $q\leqslant p$, we write $p=k+q$ and use $(iv)$ and the fact that $\Gamma(\cdot, q)$ is non-decreasing to find that for any $\mu<\beta$, $$\Gamma(\mu, q)\leqslant \Gamma(\mu+k, q) \leqslant k+\Gamma(\mu, p)\leqslant p+\Gamma(\mu,p).$$ If $q\geqslant p$, then we write $q=k+p$ and use $(iv)$ to find that $$\Gamma(\mu, q)=\Gamma(\mu, k+p) \leqslant \Gamma(\mu+k, p)\leqslant p+\Gamma(\mu+k, p).$$ In either case, we obtain the inequality $$\Gamma(\mu, q)\leqslant p+\Gamma(\mu+(q-p)^+, p),$$ where $m^+=\max\{0, m\}$.

By continuity of $\Gamma(\cdot)$,  \begin{align*} \Gamma(\beta) & = \sup_{\mu<\beta}\Gamma(\mu) = \sup_{\mu<\beta}\sup_{q<\omega} \Gamma(\mu, q) \leqslant p+\Gamma(\mu+(q-p)^+, p) \leqslant p+\Gamma(\beta,p).  \end{align*}  Combining this with  $(iv)$, we find that  $$\Gamma(\zeta)\leqslant n+\Gamma(\beta)\leqslant p+n+\Gamma(\beta,p).$$

For $\beta<\omega_1$, the second statement of $(v)$ follows  from the first with $p=n=0$.  From continuity of the functions $\zeta\mapsto \Gamma(\zeta)$ and $\zeta\mapsto \Gamma( \zeta,0)$, $$\Gamma(\omega_1)=\sup\{\Gamma(\beta): \beta<\omega_1, \beta\text{\ a limit ordinal}\}= \sup\{\Gamma(\beta,0): \beta<\omega_1, \beta\text{\ a limit ordinal}\}=\Gamma(\omega_1, 0).$$

$(vi)$ Fix $\zeta\leqslant  \omega_1$. If $\zeta$ is zero or a limit ordinal, $\Gamma(\zeta)=\Gamma(\zeta, 0)$ by $(iii)$ if $\zeta=0$ and by $(v)$ if $\zeta$ is a limit ordinal, and the result is trivial in this case. If $\zeta$ is a successor ordinal,  then $\zeta<\omega_1$ and there exist a non-successor ordinal $\beta<\omega_1$ and  $p<\omega$ such that $\zeta=\beta+p$. By $(iv)$ and $(v)$, \begin{align*} \Gamma(\zeta) & = \sup_{k<\omega}\Gamma(\zeta, k) = \sup_{k<\omega}\Gamma(\beta+p, k) \leqslant p+\sup_{k<\omega}\Gamma(\beta, p+k) \\ & \leqslant p+\Gamma(\beta)=p+\Gamma(\beta, 0) \leqslant p+\Gamma(\zeta,0). \end{align*} 

\end{proof}

\begin{rem}\upshape Since $\Gamma(\zeta)$ is a non-decreasing, continuous function of $\zeta\in \omega_1+$, $\Gamma=\sup_{\zeta\leqslant \omega_1}\Gamma(\zeta)=\sup_{\zeta<\omega_1} \Gamma(\zeta)=\Gamma(\omega_1)$.   Moreover, we have the following, which is a direct application of Proposition \ref{ord1}.

\label{wbmw}
\end{rem}

\begin{corollary} Either $\Gamma<\infty$ or there exists a countable limit ordinal $\gamma$ such that $\Gamma(\gamma)=\infty$ and for each $k<\omega$, $\{\Gamma(\zeta, k): \zeta<\gamma\}$ is  an unbounded subset of $[0,\infty)$.

\label{pen2}
\end{corollary}

To give further information about the minimum ordinal $\gamma$ such that $\Gamma(\gamma)=\infty$, we have the following.

\begin{lemma}\begin{enumerate}[(i)]\item Let $\mathcal{F}, \mathcal{G}$ be regular families with the same non-zero rank. Then for $k<\omega$, $L\in[\nn]$, $0<C<\infty$, and $\varsigma\in B_R$, there exists $M\in[L]$ such that $\mathcal{F}^M[\mathcal{P}] \subset \mathfrak{G}(k,C, \varsigma)$ if and only if there exists $M\in[L]$ such that $\mathcal{G}^M[\mathcal{P}]\subset \mathfrak{G}(k,C, \varsigma)$.  \item  For $\zeta, \mu<\omega_1$, $\Gamma(\zeta+\mu)\leqslant \Gamma(\zeta)+\Gamma(\mu)$.  \item If $\Gamma=\infty$, then there exists $0<\gamma<\omega_1$ such that $\min \{\xi: \Gamma(\xi)=\infty\}=\omega^\gamma$. \end{enumerate}

\label{work3}

\end{lemma}

\begin{proof}$(i)$ The statement of $(i)$ is symmetric in $\mathcal{F}$ and $\mathcal{G}$, so we only need to prove one direction of the implication. Assume there exists $M\in[L]$ such that $\mathcal{F}^M[\mathcal{P}]\subset \mathfrak{G}(k,C, \varsigma)$. By Proposition \ref{reg}, there exists $K\in[\nn]$ such that $\mathcal{G}(K)\subset \mathcal{F}$.  Let $N(n)=M(K(n))$, so $N\in[M]\subset [L]$. We claim that $\mathcal{G}^N[\mathcal{P}]\subset \mathcal{F}^M[\mathcal{P}]$, from which it will follow that $\mathcal{G}^N[\mathcal{P}]\subset \mathfrak{G}(k,C,\varsigma)$.    Fix $\varnothing\neq F \in \mathcal{G}^N[\mathcal{P}]$. Then we can write $F=\cup_{n=1}^t F_n$ for some $F_1<\ldots <F_t$, $\varnothing\neq F_n\in \mathcal{P}\upp N\subset \mathcal{P}\upp M$, and $(\min F_n)_{n=1}^t \in \mathcal{G}(N)$. The last condition means that $(\min F_n)_{n=1}^t = N(G)=M(K(G))$ for some $G\in \mathcal{G}$. Then $K(G)\in \mathcal{G}(K)\subset \mathcal{F}$, from which it follows that $(\min F_n)_{n=1}^t = M(K(G))\in \mathcal{F}(M)$.  Therefore we have shown that $F=\cup_{n=1}^t F_n$ for some $F_1<\ldots <F_t$, $\varnothing\neq F_n\in \mathcal{P}$, and $(\min F_n)_{n=1}^t\in \mathcal{F}(M)$, so $F\in \mathcal{F}^M[\mathcal{P}]$. Since $F$ was an arbitrary non-empty member of $\mathcal{G}^N[\mathcal{P}]$, and since $\varnothing\in \mathcal{F}^M[\mathcal{P}]$, we are done.

$(ii)$ If $\Gamma(\zeta)=\infty$ or $\Gamma(\mu)=\infty$, the result is trivial, so assume both quantities are finite and fix $\Gamma(\zeta)<B<\infty$ and $\Gamma(\mu)<C<\infty$. Fix  $k<\omega$, $L\in [\nn]$, and $\varsigma\in B_R$.  Since $\Gamma(\zeta,k)\leqslant \Gamma(\zeta)<B$, we may select $L_0\in [L]$ such that $\mathcal{F}^{L_0}_\zeta[\mathcal{P]}\subset \mathfrak{G}(k, B, \varsigma)$.    Since $\sup_{l<\omega} \Gamma(\mu, l)=\Gamma(\mu)<C$, we may recursively select $L_0\supset L_1\supset L_2\supset \ldots$ such that for all $q\in\nn$, $\mathcal{F}_{\mu}^{L_q}[\mathcal{P}]\subset \mathfrak{G}(k+q, C, \varsigma)$.  For each $q\in\nn$, let $I_{q,n}=\text{supp}(\mathbb{P}_{L_q, n})$ and $I_q=\cup_{n=1}^q I_{q,n}$. Let $I_0=\varnothing$. Choose $M(1)<M(2)<\ldots$ such that for each $q\in\nn$, $I_q<M(q)\in L_q$. 

Let $$\mathcal{F}=\{F\cup G: F<G, F\in\mathcal{F}_\mu, G\in \mathcal{F}_\zeta\},$$ and note that $\text{rank}(\mathcal{F})=\zeta+\mu+1=\text{rank}(\mathcal{F}_{\zeta+\mu})$ by Proposition \ref{reg}.  We claim that $$\mathcal{F}^M[\mathcal{P}]\subset \mathfrak{G}(k, B+C, \varsigma).$$   By $(i)$, this will imply that there exists $N\in[M]$ such that $\mathcal{F}_{\zeta+\mu}^N[\mathcal{P}]\subset \mathfrak{G}(k, B+C, \varsigma)$. Since this holds for arbitrary $k<\omega$, $\varsigma\in B_R$, $L\in[\nn]$, $B>\Gamma(\zeta)$, and $C>\Gamma(\mu)$, we will be done with $(ii)$ once we show the inclusion $\mathcal{F}^M[\mathcal{P}]\subset \mathfrak{G}(k, B+C, \varsigma)$.

Fix $\varnothing\neq E\in \mathcal{F}^M[\mathcal{P}]$ and $H=_\mathcal{P}\cup_{n=1}^t E_n\subset E$. Note that $(\min E_n)_{n=1}^t \in \mathcal{F}(M)$, as noted in Remark \ref{triv}.  Therefore there exist $F<G$, $F\in \mathcal{F}_\mu$, and $G\in \mathcal{F}_\zeta$ such that $$M(F)\cup M(G)= M(F\cup G)=(\min E_n)_{n=1}^t.$$ Note that either of the sets $F,G$ could be empty.  With $q=|F|$, $$M(F)=(\min E_n)_{n=1}^q \in \mathcal{F}_\mu(M)$$ and $$M(G)=(\min E_n)_{n=q+1}^t\in \mathcal{F}_\zeta(M).$$     Let $$T=I_q\cup \Bigl(M\setminus \bigcup_{n=1}^q \text{supp}(\mathbb{P}_{M,n})\Bigr)\in [L_q].$$  Since $q=|F|$ and $F<G$, $\min G>q$, and $$M(G)=T(G)\in \mathcal{F}_\zeta(T),$$ and $$H_2:=\bigcup_{n=q+1}^t E_n \in \mathcal{F}^T_\zeta[\mathcal{P}]\subset \mathcal{F}^{L_q}_\zeta[\mathcal{P}]\subset \mathfrak{G}(k+q, C, \varsigma).$$  Since $M(F)\in \mathcal{F}_\mu(M)$, $$H_1:=\bigcup_{n=1}^q E_n\in \mathcal{F}^M_\mu [\mathcal{P}] \subset \mathcal{F}^{L_0}_\mu[\mathcal{P}] \subset \mathfrak{G}(k, B, \varsigma).$$   Therefore \begin{align*} s_k(\mathbb{E}_H \varsigma) & = s_k\Bigl(\sum_{n=1}^t \mathbb{E}_H \varsigma(n)\otimes e_n\Bigr) =s\Bigl(\sum_{n=1}^t \mathbb{E}_H \varsigma(n)\otimes e_{n+k}\Bigr) \\ & = s\Bigl(\sum_{n=1}^q \mathbb{E}_{H_1}\varsigma\otimes e_{n+k} + \sum_{n=q+1}^t \mathbb{E}_{H_2}\varsigma(n-q)\otimes e_{n+k}\Bigr) \\ & \leqslant s\Bigl(\sum_{n=1}^q \mathbb{E}_{H_1}\varsigma\otimes e_{n+k}\Bigr) + s\Bigl(\sum_{n=1}^{t-q} \mathbb{E}_{H_2}\varsigma(n)\otimes e_{n+k+q}\Bigr) \\ & = s_k(\mathbb{E}_{H_1}\varsigma)+s_{k+q}(\mathbb{E}_{H_2}\varsigma) <B+C. \end{align*}

$(iii)$ If $\Gamma=\infty$, then by Corollary \ref{pen2}, there exists a minimum countable ordinal $\lambda$ such that $\Gamma(\lambda)=\infty$. Since $\Gamma(0)=0$, $0<\lambda<\omega_1$.  By the minimality of $\lambda$, if  $\zeta, \mu<\lambda$, then $\Gamma(\zeta)+\Gamma(\mu)<\infty$.  By $(ii)$, $\Gamma(\zeta+\mu)\leqslant \Gamma(\zeta)+\Gamma(\mu)<\infty$, and $\zeta+\mu<\lambda$. Therefore we have shown that $0<\lambda<\omega_1$ is such that if $\zeta, \mu<\lambda$, then $\zeta+\mu<\lambda$. It is a standard fact about ordinals that $\lambda=\omega^\gamma$ for some $0\leqslant \gamma<\omega_1$.  Since $\mu(1)\leqslant 1+\mu(0)=1$ by Lemma \ref{work2}$(iii)$ and $(iv)$, $1<\lambda$ and $0<\gamma$.

\end{proof}

\begin{proposition} If $\Gamma(\omega_1)<\infty$, then $(\mathfrak{P}, \mathcal{P})$ and $R,S$ satisfy $(6)$ of Theorem \ref{main2}. Moreover, $\Gamma(\omega_1)$ is the infimum of $C$ such that $(6)$ is satisfied with constant $C$. 

\label{home2}
\end{proposition}

\begin{proof} Assume that $\Gamma(\omega_1)<\infty$ and fix $\Gamma(\omega_1)<C<\infty$.  Note that $\sup_{\zeta<\omega_1} \Gamma(\zeta)=\Gamma(\omega_1)<C$.     Fix $\varsigma\in B_R$ and $L\in [\nn]$.    For each $\zeta<\omega_1$,  $\Gamma(\zeta)<C$, from which it follows that there exists $M_\zeta\in [\nn]$ such that $\mathcal{F}^{M_\zeta}_\zeta[\mathcal{P}]\subset \mathfrak{G}(0, C, \varsigma)$.  From this it follows that $$\text{rank}(\mathfrak{G}(0, C, \varsigma) \upp L)\geqslant \text{rank}(\mathcal{F}^{M_\zeta}_\zeta[\mathcal{P}])=\text{rank}(\mathcal{F}_\zeta[\mathcal{P}])\geqslant \text{rank}(\mathcal{F}_\zeta)=\zeta+1.$$  Since this holds for any $\zeta<\omega_1$, it follows that $\mathfrak{G}(0, C, \varsigma)\upp L$ is ill-founded. Therefore there exists an infinite subset $M$ of $\nn$ such that for every $t\in \nn$, $(M(n))_{n=1}^t \in \mathfrak{G}(0, C, \varsigma)\upp L$. This clearly implies that $M\in [L]$.   Since $\mathfrak{G}(0, C, \varsigma)$ is hereditary, $[M]^{<\omega}\subset \mathfrak{G}(0, C, \varsigma)$.  Now for any $N\in [M]$ and $t\in \nn$, if $H_t=\cup_{n=1}^t \text{supp}(\mathbb{P}_{N, n})$, then $H_t\in [M]^{<\omega}\subset \mathfrak{G}(0, C, \varsigma)$ and $s(\mathbb{E}_{H_t} \varsigma)\leqslant C$.    By the properties of $s$ and $S$, $\mathbb{E}_N \varsigma\in CB_S$.    Since $\varsigma\in B_R$ and $L\in[\nn]$ were arbitrary, condition $(6)$ of Theorem \ref{main2} is satisfied. Moreover, since $C>\Gamma(\omega_1)$ was arbitrary, the infimum of $C$ such that $(\mathfrak{P}, \mathcal{P})$ and $R,S$ satisfy condition $(6)$ is not more than $\Gamma(\omega_1)$.

It remains to show that this infimum is not less than $\Gamma(\omega_1)$. To that end, assume $(\mathfrak{P}, \mathcal{P})$ and $R,S$ satisfy condition $(6)$ of Theorem \ref{main2} with constant $C$. Then for any $\varsigma\in B_R$ and $L\in [\nn]$, there exists $M_1\in[L]$ such that $\mathbb{E}_N\varsigma\in CB_S$ for all $N\in[M_1]$.  Fix $k\in\nn$ and let $I=\cup_{n=1}^k \text{supp}(\mathbb{P}_{M_1,n})$ and $M=M_1\setminus I$.  Here, $I=\varnothing$ if $k=0$.     Fix $F=_\mathcal{P}\cup_{n=1}^t F_n\in [M]^{<\omega}$. Fix $N\in[M]$ such that $F\prec N$ and note that $I\cup N\in [M_1]$ and $\mathbb{E}_{I\cup N}\varsigma(n+k)= \mathbb{E}_N \varsigma(n)$ for all $n\in\nn$ by the permanence property.  By the bimonotone property of $s$, \begin{align*} s_k(\mathbb{E}_F \varsigma) & \leqslant s_k(\mathbb{E}_N\varsigma) \leqslant s(\mathbb{E}_{I\cup N}\varsigma)\leqslant C. \end{align*}  Therefore $\mathcal{F}_{\omega_1}^M[\mathcal{P}]=[M]^{<\omega}\subset \mathfrak{G}(k, C, \varsigma)$.   Since $\varsigma\in B_R$, $k<\omega$, $L\in[\nn]$ were arbitrary, $$\Gamma=\Gamma(\omega_1) =\sup_{\varsigma\in B_R} \sup_{k<\omega}\sup_{L\in[\nn]}\Gamma(\omega_1,k, \varsigma, L) \leqslant C.$$

\end{proof}

The preceding result shows the connection between the conditions appearing the $\zeta=\omega_1$ case of Theorem \ref{main3} and the function $\Gamma$. We now wish to illustrate the connection between the conditions in Theorem \ref{main3} and the function $\Gamma$ in the $\zeta<\omega_1$ case. 

\begin{proposition} Fix $\zeta<\omega_1$,  $k<\omega$,  $\varsigma\in B_R$, and $L\in[\nn]$. \begin{enumerate}[(i)]\item Fix $0<C<\infty$. If $L\in[\nn]$ is such that for each $M_1\in [L]$, there exists $M\in[M_1]$ such that for all $N\in[M]$, $\mathbb{E}_{N|\mathcal{F}_\zeta[\mathcal{P}]} \varsigma\in CB_{S_k}$, then $\Gamma(\zeta,k, \varsigma, L)\leqslant C$.\item If $0<C<\infty$, $L\in[\nn]$, $k<\omega$ are such that $\Gamma(\zeta, k, \varsigma, L)>C$, then there exists $M\in [L]$ such that for all $N\in [M]$, $\mathbb{E}_{N|\mathcal{F}_\zeta[\mathcal{P}]}\varsigma\notin CB_{S_k}$.  In particular, $MAX(\mathcal{F}_\zeta[\mathcal{P}]\upp M) \cap \mathfrak{G}(k, C, \varsigma)=\varnothing$.  \item If $\Gamma(\zeta+1,k, \varsigma,L)<C$, then for any $M_1\in [L]$, there exists $M\in[M_1]$ such that for all $N\in [M]$, $\mathbb{E}_{N|\mathcal{F}_\zeta[\mathcal{P}]}\varsigma\in CB_{S_k}$. \end{enumerate}

\label{home3}
\end{proposition}

\begin{proof}$(i)$ Fix $M_1\in [L]$ and let $M\in[M_1]$ be such that for all $N\in [M]$, $E_{N|\mathcal{F}_\zeta[\mathcal{P}]} \varsigma\in CB_{S_k}$.  For any $F\in \mathcal{F}_\zeta[\mathcal{P}]\upp M$ and $G=_\mathcal{P}\cup_{n=1}^t G_n\subset F$, there exists $N\in[M]$ such that $G\prec N|\mathcal{F}_\zeta[\mathcal{P}]$. By the properties of $M$,  $\mathbb{E}_{N|\mathcal{F}_\zeta[\mathcal{P}]}\varsigma\in CB_{S_k}$. By bimonotonicity and the permanence property, together with the fact that $G\prec N|\mathcal{F}_\zeta[\mathcal{P}]$, $\mathbb{E}_G\varsigma\in CB_{S_k}$. Since $G=_\mathcal{P}\cup_{n=1}^tG_n \subset F$ was arbitrary, $F\in \mathfrak{G}(k, C, \varsigma)$. Since $F\in \mathcal{F}_\zeta[\mathcal{P}]\upp M$ was arbitrary, $$\mathcal{F}^M_\zeta[\mathcal{P}]\subset \mathcal{F}_\zeta[\mathcal{P}]\upp M\subset \mathfrak{G}(k, C, \varsigma).$$  Since $M_1\in [L]$ was arbitrary, we are finished with $(i)$.

$(ii)$ The condition $\Gamma(\zeta, k, \varsigma, L)>C$ implies the existence of some $L_0\in [L]$ such that there does not exist $M\in [L_0]$ such that $(\varsigma, M)$ is $(\zeta, k, C)$-stable. We let $$\mathcal{V}=\{M\in [\nn]: \mathbb{E}_{M|\mathcal{F}_\zeta[\mathcal{P}]}\varsigma\in CB_{S_k}\}.$$  Since this is a closed, and in fact clopen, set, Theorem \ref{Ramsey} guarantees the existence of some $M\in [L_0]$ such that either $[M]\subset \mathcal{V}$ or $[M]\cap \mathcal{V}=\varnothing$. We will be done once we show that the inclusion $[M]\subset \mathcal{V}$ cannot hold.  If the first alternative were to hold, we could choose $G\in \mathcal{F}_\zeta^M[\mathcal{P}]\subset \mathcal{F}_\zeta[\mathcal{P}]\upp M$ and $F=_\mathcal{P}\cup_{n=1}^t F_n\subset G$. There exists $N\in [M]$ such that $F\prec N|{\mathcal{F}_\zeta[\mathcal{P}]}$, so by bimonotonicity, $$s_k(\mathbb{E}_F\varsigma) \leqslant s_k(\mathbb{E}_{N|{\mathcal{F}_\zeta[\mathcal{P}]}} \varsigma)\leqslant C.$$  Since $G\in \mathcal{F}_\zeta^M[\mathcal{P}]$ and $F=_\mathcal{P}\cup_{n=1}^t F_n\subset G$ were arbitrary, $(\varsigma, M)$ is $(\zeta, k, C)$-stable.  This contradiction finishes the first part of $(ii)$. For the second part of $(ii)$, since for any $N\in[M]$, $N|\mathcal{F}_\zeta[\mathcal{P}]=_\mathcal{P}\cup_{n=1}^t F_n$ for some $F_n$ and $t$, and since $s_k(\mathbb{E}_{N|\mathcal{F}[\mathcal{P}]}\varsigma)>C$, $N|\mathcal{F}_\zeta[\mathcal{P}]\notin \mathfrak{G}(k, C, \varsigma)$. We conclude $(ii)$ by noting that $$MAX(\mathcal{F}_\zeta[\mathcal{P}]\upp M)=\{N|\mathcal{F}_\zeta[\mathcal{P}]: N\in [M]\}.$$

%We apply Proposition \ref{dich} with $\mathcal{F}=\mathcal{F}_\zeta[\mathcal{P}]$ and $\mathcal{G}=\mathfrak{G}(k, C, \varsigma)$.  There exists $M\in[L_0]$ such that either $\mathcal{F}_\zeta[\mathcal{P}]\upp M\subset \mathfrak{G}(k,C, \varsigma)$ or $MAX(\mathcal{F}_\zeta[\mathcal{P}]\upp M)\cap \mathfrak{G}(k, C, \varsigma)=\varnothing$. We will be done once we show that the first alternative cannot hold. If the first alternative were to hold, then $$\mathcal{F}^M_\zeta[\mathcal{P}]\subset \mathcal{F}_\zeta[\mathcal{P}]\upp M\subset \mathfrak{G}(k, C, \varsigma),$$ from which it follows that $(\varsigma, M)$ is $(\zeta, k, C)$-stable. This is a contradiction of the properties of $L_0$. 

$(iii)$  Fix $M_1\in [L]$. Since $\Gamma(\zeta+1, k, \varsigma, L)<C$, there exists $M_0\in[M_1]$ such that $\mathcal{F}^{M_0}_{\zeta+1}[\mathcal{P}]\subset \mathfrak{G}(k, C, \varsigma)$.   By Theorem \ref{gasp}, there exists $M\in [M_0]$ such that either $\mathcal{F}_{\zeta+1}^{M_0}[\mathcal{P}]\upp M\subset \mathcal{F}_\zeta[\mathcal{P}]$ or $\mathcal{F}_\zeta[\mathcal{P}]\upp M\subset \mathcal{F}^{M_0}_{\zeta+1}[\mathcal{P}]\subset \mathfrak{G}(k, C, \varsigma)$.    By comparing the rank of $\mathcal{F}^{M_0}_{\zeta+1}[\mathcal{P}]\upp M$, which is $\omega^\xi (\zeta+1)+1$, to the rank of $\mathcal{F}_\zeta[\mathcal{P}]\upp M$, which is $\omega^\xi \zeta+1<\omega^\xi(\zeta+1)+1$, we see that the second inclusion must hold.  Therefore $\mathcal{F}_\zeta[\mathcal{P}]\upp M\subset \mathcal{F}^{M_0}_{\zeta+1}[\mathcal{P}]\subset \mathfrak{G}(k, C, \varsigma)$.   From this we deduce that for any $N\in[M]$, since $N|\mathcal{F}_\zeta[\mathcal{P}]\in \mathcal{F}_\zeta[\mathcal{P}]\upp M\subset \mathfrak{G}(k, C, \varsigma)$, $\mathbb{E}_{N|\mathcal{F}_\zeta[\mathcal{P}]}\varsigma\in CB_{S_k}$. 

\end{proof}

One of the major increases in difficulty in the case of a $\xi$-homogeneous probability block with $0<\xi<\omega_1$ is that the convex coefficients a vector in a sequence receives depends upon its position in the sequence.  We now turn to the process of overcoming this difficulty. The idea is, once we obtain some hereditary badness, we can make that badness independent of the position of a vector in the sequence by witnessing the badness with linear functionals.  Then moving the vector within the sequence does not lose the badness.  Our next two results compare $\Gamma$ values between two probability blocks, so we use superscripts to distinguish.

\begin{lemma} Let $\xi<\omega_1$ and $P=(\mathfrak{P}, \mathcal{P})$ be our fixed, $\xi$-homogeneous probability block. Assume $\varsigma=(x_n)_{n=1}^\infty\in B_{\ell_\infty(X)}$,  $k<\omega$, $\zeta<\omega_1$, and $L\in[\nn]$ are such that for each $F\in MAX(\mathcal{F}_\zeta[\mathcal{P}])\upp L$, $s_k(\mathbb{E}^P_F \varsigma) \geqslant D$. Then with $\varrho=(x_{L(n)})_{n=1}^\infty$, for any $\upsilon \leqslant \xi$, any $\upsilon$-homogeneous probability block $Q=(\mathfrak{Q}, \mathcal{Q})$, and $N\in [\nn]$, $$\Gamma^Q(\zeta+1, k, \varrho, N)\geqslant D.$$   

\label{beer}

\end{lemma}

\begin{proof} We first prove the trivial $\xi=0$ case. This case is trivial, since it implies that $P$ and $Q$ must both be equal to the Dirac probability block.  In this case, $\mathcal{F}_\zeta[\mathcal{P}]=\mathcal{F}_\zeta[\mathcal{Q}]=\mathcal{F}_\zeta$ and $\mathcal{F}_{\zeta+1}[\mathcal{Q}]=\mathcal{F}_{\zeta+1}$. Assume $0<C<\infty$ and $M\in[\nn]$ are such that $\mathcal{F}^M_{\zeta+1}[\mathcal{Q}]\subset \mathfrak{G}^Q(k,C, \varrho)$.     We can then recursively choose $i_n, j_n\in\nn$ such that $i_1=M(1)$ and for all $n\in\nn$, $j_n=L(i_n)$ and $M(j_n)<i_{n+1}\in M$.  By compactness, there exists $t\in\nn$ such that $F:=(j_n)_{n=1}^t\in MAX(\mathcal{F}_\zeta)=MAX(\mathcal{F}_\zeta[\mathcal{P}])$.    Since $j_n=L(i_n)$ for all $n\in\nn$, $F\in MAX(\mathcal{F}_\zeta[\mathcal{P}])\upp L$. Since $i_n\in M$ for all $n\in\nn$, we can write $E:=(i_n)_{n=1}^t=M(G)$ for some $G\in[\nn]^{<\omega}$. Since $i_{n+1}>M(j_n)$ for each $1\leqslant n<t$, $$M(G\setminus (\min G))=(i_n)_{n=2}^t=(i_{n+1})^{t-1}_{n=1}$$ is a spread of $M((j_n)_{n=1}^{t-1})\subset M(F)$. From this it follows that $G\setminus (\min G)$ is a spread of a subset of $F$, and therefore $G\setminus (\min G)\in \mathcal{F}_\zeta$. Therefore $$G=(\min G)\cup (G\setminus (\min G))\in \mathcal{F}_{\zeta+1}$$ and $$E=(i_n)_{n=1}^t=M(G)\in \mathcal{F}_{\zeta+1}^M=\mathcal{F}_{\zeta+1}^M[\mathcal{Q}]\subset \mathfrak{G}^Q(k,C, \varrho).$$   Then since $E=_\mathcal{Q}\cup_{n=1}^t (i_n)$ and  \begin{displaymath}
   \mathbb{E}^Q_E\varrho(n) = \left\{
     \begin{array}{rr}
       x_{L(i_n)} & : 1\leqslant n \leqslant t\\
       0 & : t<n,
     \end{array}
   \right.
\end{displaymath} $s_k(\sum_{n=1}^t x_{L(i_n)}\otimes e_n)\leqslant C$.  But since $F=_\mathcal{P}\cup_{n=1}^t(j_n)=\cup_{n=1}^t(L(i_n))\in MAX(\mathcal{F}_\zeta[\mathcal{P}])\upp L$ and \begin{displaymath}
   \mathbb{E}^P_E\varsigma(n) = \left\{
     \begin{array}{rr}
       x_{j_n}=x_{L(i_n)} & : 1\leqslant n \leqslant t\\
       0 & : t<n,
     \end{array}
   \right. \end{displaymath} it follows from our hypotheses that $s_k(\sum_{n=1}^t x_{L(i_n)}\otimes e_n)\geqslant D>C$, a contradiction. This contradiction finishes the $\xi=0$ case. 

We next prove the $0<\xi<\omega_1$ case. Assume that $0<C<\infty$ and $M\in[\nn]$ are such that $\mathcal{F}^M_{\zeta+1}[\mathcal{Q}]\subset \mathfrak{G}^Q(k, C, \varrho)$.  In order to prove the lemma, we must show that $C\geqslant D$. To that end, we will assume $C<D$ and reach a contradiction.  Fix a sequence $(\ee_n)_{n=1}^\infty \subset (0,1)$ such that $C+\sum_{n=1}^\infty \ee_n<D$.  For each $n\in\nn$, let $\delta_n=\ee_n/4$. Note that for any $n\in\nn$, any real number $b\leqslant 2$, and any $0\leqslant c\leqslant \delta_n$, $(1-c)(b-\delta_n)-c \geqslant b-\ee_n$.

  For each $F\in MAX(\mathcal{F}_\zeta[\mathcal{P}])\upp L$, we may choose $f_F\in B_{S_k^*}$ such that $$f_F(\mathbb{E}^P_F \varsigma)=\text{Re\ }f_F(\mathbb{E}^P_F \varsigma)=s_k(\mathbb{E}^P_F \varsigma)\geqslant D.$$   We define $g:L\ominus \mathcal{F}_\zeta[\mathcal{P}]\to [-1,1]$. Fix $(i,F)\in L\ominus \mathcal{F}_\zeta[\mathcal{P}]$, write $F=_\mathcal{P}\cup_{n=1}^t F_n$, let $1\leqslant j\leqslant t$ be such that $i\in F_j$, and define $$g(i,F)= \text{Re\ }f_F(x_i\otimes e_j)\in [-1,1].$$   Then for $F=_\mathcal{P} \cup_{n=1}^t F_n\in MAX(\mathcal{F}_\zeta[\mathcal{P}])$, \begin{align*} \mathbb{E}_F^P g(\cdot, F) & = \sum_{n=1}^t \mathbb{E}^P_{F_n} g(\cdot, F) = \sum_{n=1}^t \sum_{i\in F_n}\mathbb{P}_{F,n}(i) \text{Re\ }f_F(x_i\otimes e_n)= \text{Re\ }f_F\Bigl(\sum_{n=1}^t \bigl(\sum_{i\in F_n} \mathbb{P}_{F,n}(i)x_i\bigr)\otimes e_n\Bigr) \\ & = \text{Re\ }f_F(\mathbb{E}^P_F \varsigma)=s_k(\mathbb{E}^P_F \varsigma) \geqslant D.\end{align*}

Let us consider the case $0<\upsilon$.    By Theorem \ref{god2}, we may find $F=_\mathcal{Q} \cup_{n=1}^t F_n\in \mathcal{F}^M_{\zeta+1}[\mathcal{Q}]$,  $b_1, \ldots, b_t\in \rr$, $H_1<\ldots <H_t$, $H_n\in MAX(\mathcal{P})\upp L$, $H\in MAX(\mathcal{F}_\zeta[\mathcal{P}])\upp L$ such that \begin{enumerate}[(i)]\item for each $1\leqslant n\leqslant t$ and each $i\in\nn$, $\mathbb{Q}_{F,n}(i)\leqslant \delta_n$, \item for each $1\leqslant n\leqslant t$, $L(F_n\setminus (\min F_n))\subset H_n$, \item $H=_\mathcal{P} \cup_{n=1}^t H_n$, \item for each $1\leqslant n\leqslant t$ and $m\in F_n\setminus (\min F_n)$, $g(L(m), H)\geqslant b_n-\delta_n$,  \item $D\leqslant \sum_{n=1}^t b_n$. \end{enumerate}    Since $F\in \mathcal{F}^M_{\zeta+1}[\mathcal{Q}]\subset \mathfrak{G}(k, C, \varrho)$, $s_k(\mathbb{E}^Q_F \varrho)\leqslant C$.  However, \begin{align*} s_k(\mathbb{E}^Q_F \varrho) &  \geqslant \text{Re\ } f_H(\mathbb{E}^Q_F \varrho) = \text{Re\ }f_H\Bigl(\sum_{n=1}^t \bigl(\sum_{i\in F_n} \mathbb{Q}_{F,n}(i) x_{L(i)}\bigr)\otimes e_n\Bigr) \\ & = \sum_{n=1}^t \sum_{i\in F_n} \mathbb{Q}_{F,n}(i)\text{Re\ }f_H(x_{L(i)} \otimes e_n) \\ & = \sum_{n=1}^t \mathbb{Q}_{F,n}(\min F_n)\text{Re\ }f_H(x_{L(\min F_n)}\otimes e_n) + \sum_{n=1}^t \sum_{i\in F_n\setminus (\min F_n)} \mathbb{Q}_{F,n}\text{Re\ }f_H(x_{L(i)}\otimes e_n) \\ & = \sum_{n=1}^t \mathbb{Q}_{F,n}(\min F_n)\text{Re\ }f_H(x_{L(\min F_n)}\otimes e_n) + \sum_{n=1}^t \sum_{i\in F_n\setminus (\min F_n)} \mathbb{Q}_{F,n}(i)g(L(i), H).\end{align*} 

For the last equality, we have used the fact that for $1\leqslant n\leqslant t$ and $i\in F_n\setminus (\min F_n)$, $L(i)\in H_n$, from which it follows that $$g(L(i), H)= \text{Re\ }f_H(x_{L(i)}\otimes e_n).$$   Now continuing the inequality, \begin{align*} s_k(\mathbb{E}^Q_F \varrho) &  \geqslant \sum_{n=1}^t \mathbb{Q}_{F,n}(\min F_n)\text{Re\ }f_H(x_{L(\min F_n)}\otimes e_n) + \sum_{n=1}^t \sum_{i\in F_n\setminus (\min F_n)} \mathbb{Q}_{F,n}(i)g(L(i), H) \\ & \geqslant  \sum_{n=1}^t \mathbb{Q}_{F,n}(\min F_n)(-1) + \sum_{n=1}^t (1- \mathbb{Q}_{F,n}(\min F_n)) (b_n-\delta_n) \\ & = \sum_{n=1}^t \bigl[(1-\mathbb{Q}_{F,n}(\min F_n))(b_n-\delta_n) - \mathbb{Q}_{F,n}(\min F_n)\bigr] \\ & \geqslant \sum_{n=1}^t b_n-\ee_n \geqslant D-\sum_{n=1}^\infty  \ee_n>C. \end{align*} This is a contradiction, and finishes the $0<\upsilon $ case.

We now consider the $\upsilon=0$ case. Note that in this case, $Q$ is the Dirac block and $\mathcal{F}_{\zeta+1}^M[\mathcal{Q}]=\mathcal{F}_{\zeta+1}(M)$.  By Theorem \ref{demigod2}, there exist $F=(m_n)_{n=1}^t \in \mathcal{F}_{\zeta+1}(M)$, $L\in [L]$,  $b_1, \ldots, b_t\in\rr$,  and  sets $H_1, \ldots, H_t\in \mathcal{P}$, $H\in MAX(\mathcal{F}_\zeta[\mathcal{P}])$, such that  $L(m_n)\in H_n$,    $H=_\mathcal{P}\bigcup_{n=1}^t H_n$,  for each $1\leqslant n\leqslant t$, $g(L(m_n), H) \geqslant b_n-\delta_n$, and $\sum_{n=1}^t b_n\geqslant D$.   Let $\varrho=(x_{L(n)})_{n=1}^\infty$.     Then \begin{align*} s_k\Bigl(\sum_{n=1}^t x_{L(m_n)}\otimes e_n\Bigr) & \geqslant \text{Re\ }f_H\Bigl(\sum_{n=1}^t x_{L(m_n)}\otimes e_n\Bigr) =\sum_{n=1}^t g(L(m_n), H)\geqslant \sum_{n=1}^t b_n-\delta_n \geqslant D-\sum_{n=1}^\infty \ee_n>C. \end{align*} Since $\mathbb{E}^Q_F\varrho=\sum_{n=1}^t x_{m_n}\otimes e_n$ and $s_k(\mathbb{E}^Q_F\varrho)\leqslant C$, this is a contradiction.  This contradiction finishes this case.

\end{proof}

\begin{corollary} Let $\xi<\omega_1$ and $P=(\mathfrak{P}, \mathcal{P})$ be our fixed, $\xi$-homogeneous probability block. Fix $\upsilon\leqslant \xi$ and any $\upsilon$-homogeneous probability block $Q=(\mathfrak{Q}, \mathcal{Q})$. \begin{enumerate}[(i)]\item For any $\zeta<\omega_1$ and $k<\omega$,  $\Gamma^P(\zeta, k)\leqslant \Gamma^Q(\zeta+1, k)$ and $\Gamma^P(\zeta)\leqslant \Gamma^Q(\zeta+1)$. \item If $\zeta\leqslant \omega_1$ is a limit ordinal, $\Gamma^P(\zeta)\leqslant \Gamma^Q(\zeta)$.  \item If $\upsilon=\xi$ and if $\zeta\leqslant \omega_1$ is a limit ordinal, $\Gamma^P(\zeta)= \Gamma^Q(\zeta)$.\end{enumerate}

\label{shift1}
\end{corollary}

\begin{proof}$(i)$ If $\xi=0$, then $\upsilon=0$ and $P=Q$ is the Dirac probability block. Assume $0<\xi$. Suppose that $\Gamma^P(\zeta, k)>D$. There exists $\varsigma\in B_R$ such that $\Gamma^P(\zeta, k, \varsigma, \nn)>D$.   This means there exists $L\in[\nn]$ such that there does not exist $M\in[L]$ such that $(\varsigma, M)$ is $(k, D, \varsigma)$-stable.  By Proposition \ref{home3}$(ii)$, there exists $L\in[\nn]$ such that for each $N\in [L]$, $\mathbb{E}_{N|\mathcal{F}_\zeta[\mathcal{P}]} \varsigma\notin DB_{S_k}$.     By Lemma \ref{beer}, $\Gamma^Q(\zeta+1, k)\geqslant D$.  Since this holds for any $D<\Gamma^P(\zeta, k)$, $\Gamma^P(\zeta, k) \leqslant \Gamma^Q(\zeta+1, k)$.

The second inequality of $(i)$ follows from the first by taking the supremum over $k$.

$(ii)$ By $(i)$ and Lemma \ref{work2}$(v)$, for a limit ordinal $\zeta\leqslant \omega_1$, $$\Gamma^P(\zeta)=\Gamma^P(\zeta, 0)= \sup_{\mu<\zeta} \Gamma^P(\mu, 0) \leqslant \sup_{\mu<\zeta}\Gamma^Q(\mu+1, 0)=\Gamma^Q(\zeta, 0)=\Gamma^Q(\zeta).$$  

$(iii)$ This follows from $(ii)$, since the inequality holds in both directions.

\end{proof}

We are now ready to prove the remaining implications $(1)\Rightarrow (3)\Rightarrow (4)$ of Theorem \ref{main3}, which we now recall.

\begin{corollary} Fix $\zeta\leqslant \omega_1$. \begin{enumerate}[(i)]\item If for every $\varsigma\in B_R$, there exist $C$ and $M\in[\nn]$ such that for all $N\in[M]$, $\mathbb{E}_{N|\mathcal{F}_\zeta[\mathcal{P}]}\varsigma\in CB_S$, then for every $\varsigma\in B_R$ and $L\in[\nn]$, there exist $C$ and $M\in[L]$ such that for all $N\in [M]$, $\mathbb{E}_{N|\mathcal{F}_\zeta[\mathcal{P}]}\varsigma\in CB_S$. \item If for each $\varsigma\in B_R$ and $L\in[\nn]$, there exist $M\in [L]$ and $0<C<\infty$ such that $\mathbb{E}_{N|\mathcal{F}_\zeta[\mathcal{P}]}\varsigma\in B_S$, then there exists a constant $C$ such that for all $\varsigma\in B_R$ and $L\in[\nn]$, there exists $M\in[L]$ such that for all $N\in [M]$, $\mathbb{E}_{N|\mathcal{F}_\zeta[\mathcal{P}]}\varsigma \in B_S$.  \end{enumerate}

\label{boardtown}
\end{corollary}

\begin{proof} $(i)$ We first prove the $\zeta<\omega_1$ case by contraposition. Suppose there exist $\varsigma\in B_R$ and $L\in[\nn]$ such that for all $M\in[L]$ and $0<C<\infty$, there exists $N\in[M]$ such that $\mathbb{E}_{N|\mathcal{F}_\zeta[\mathcal{P}]}\varsigma\notin CB_S$.  By Proposition \ref{home3}$(iii)$, $\inf_{M_0\in [L]}\Gamma(\zeta+1, 0, \varsigma, M_0)=\infty$. Indeed, if there were some $M_0\in [L]$ and $0<C<\infty$ such that $\Gamma(\zeta+1, 0, \varsigma, M_0)<C$, then by Proposition \ref{home3}$(iii)$, there exists $M\in [M_0]$ such that for all $N\in [M]$, $\mathbb{E}_{N|\mathcal{F}_\zeta[\mathcal{P}]}\in CB_S$, a contradiction.   Therefore $\Gamma(\zeta+1, 0, \varsigma, M_0)=\infty$ for all $M_0\in [L]$.   Let $\beta$ be the largest non-successor ordinal not exceeding $\zeta$ and note that by Lemma \ref{work2}$(v)$, $\Gamma(\beta, k, \varsigma, M_0)=\infty$ for all $M_0\in [L]$ and $k<\omega$.   Since $\Gamma(p)\leqslant p$ for any $p<\omega$ by Lemma \ref{work2}, this implies that $\zeta\geqslant\omega$. 

\begin{claim} For any $k<\omega$ and $M_0\in [L]$, there exist $\mu<\beta$ and $M_1\in [M_0]$ such that for all $N\in [M_1]$, $\mathbb{E}_{N|\mathcal{F}_\mu[\mathcal{P}]}\varsigma\notin k B_{S_k}$.\label{claim1}\end{claim}

  Let us prove the claim. Suppose that for some $k<\omega$ and $M_0\in [L]$, no such $\mu<\beta$ and $M_1\in [M_0]$ exist.  Fix $\mu<\beta$ and  define the closed set $$\mathcal{V}_k=\{M\in[\nn]: \mathbb{E}_{M|\mathcal{F}_\mu[\mathcal{P}]}\varsigma\in k B_{S_k}\}.$$ Then for any $K\in [M_0]$, we use  Theorem \ref{Ramsey} to deduce the existence of some $M\in [K]$ such that either $[M]\subset \mathcal{V}_k$ or $[M]\cap \mathcal{V}_k=\varnothing$.   By our contradiction hypothesis in the proof of the claim, it follows that the alternative $[M]\cap \mathcal{V}_k=\varnothing$ cannot hold, otherwise $M_1=M$ and $\mu$ are as in the conclusion of the claim.  From this it follows that $[M]\subset \mathcal{V}_k$.  Therefore we have shown that for any $K\in [M]$ and $\mu<\beta$, there exist $M\in [K]$ such that for all $N\in[M]$, $\mathbb{E}_{N|\mathcal{F}_\mu[\mathcal{P}]}\varsigma\in kB_{S_k}$. It follows from Proposition \ref{home3}$(i)$ that $\Gamma(\mu, k, \varsigma, M_0)\leqslant k+1$.  Since this holds for any $\mu<\beta$, by Lemma \ref{work2}$(ii)$, $\Gamma(\beta, k, \varsigma, M_0)\leqslant k+1$, contradicting the fact that $\Gamma(\beta, k, \varsigma, M_0)=\infty$ established in the paragraph before the claim. This is the necessary contradiction, and we have proved the claim. 

We apply Claim \ref{claim1} recursively to select $L\supset M_1\supset M_2\supset \ldots$ and $\mu_1, \mu_2, \ldots<\beta$ such that for all $1\leqslant k<\omega$ and $N\in [M_k]$, $\mathbb{E}_{N|\mathcal{F}_{\mu_k}[\mathcal{P}]}\varsigma\notin k B_{S_k}$.  At each stage of the recursion, if necessary we may replace $M_k$ with a subset thereof and use the almost monotone property of the fine Schreier families together with the fact that $\mu_k<\beta\leqslant \zeta$ to deduce that $\mathcal{F}_{\mu_k}[\mathcal{P}]\upp M_k\subset \mathcal{F}_\zeta[\mathcal{P}]$.   Choose $M(1)<M(2)<\ldots$ such that  $M(k)\in M_k$.   Fix $1\leqslant k<\omega$ and $N\in[M]$. Let $I=\cup_{n=1}^k \text{supp}(\mathbb{P}_{N,n})$ and note that $(N|{\mathcal{F}_{\mu_k+k}[\mathcal{P}]})\setminus I=(N\setminus I)|{\mathcal{F}_{\mu_k}[\mathcal{P}]}$ by the permanence property. Also, note that, since $I$ contains at least the first $k$ members of $N$, $(N|{\mathcal{F}_{\mu_k+k}[\mathcal{P}]})\setminus I \subset M_k$, from which it follows that $(N|{\mathcal{F}_{\mu_k+k}[\mathcal{P}]})\setminus I \in MAX(\mathcal{F}_{\mu_k}[\mathcal{P}]\upp M_k)$.     Also by the permanence property, $\mathbb{E}_{N|\mathcal{F}_{\mu_k+k}[\mathcal{P}]}\varsigma (n+k)= \mathbb{E}_{(N\setminus I)|{\mathcal{F}_{\mu_k}[\mathcal{P}]}}\varsigma(n)$ for all $n\in\nn$.    By the bimonotone property of $s$, \begin{align*} s(\mathbb{E}_{N|\mathcal{F}_{\mu_k+k}[\mathcal{P}]}\varsigma) & = s\Bigl(\sum_{n=1}^\infty \mathbb{E}_{N|\mathcal{F}_{\mu_k+k}[\mathcal{P}]}\varsigma(n)\otimes e_n\Bigr) \geqslant s\Bigl(\sum_{n=1}^\infty \mathbb{E}_{(N\setminus I)|{\mathcal{F}_{\mu_k}[\mathcal{P}]}}\varsigma(n)\otimes e_{n+k}\Bigr) \\ & =s_k(\mathbb{E}_{(N\setminus I)|{\mathcal{F}_{\mu_k}[\mathcal{P}]}}\varsigma)>k.\end{align*}  This implies that for any $N\in [M]$, $\Gamma(\mu_k+k+1 , 0, \varsigma, N)\geqslant k$ by Proposition \ref{home3}.   Therefore $M\in [L]$ and $\mu_1, \mu_2,\ldots <\beta$ have the property that for any $1\leqslant k<\omega$ and any $N\in [M]$, $s(\mathbb{E}_{N|\mathcal{F}_{\mu_k+k}} \varsigma) \geqslant k$.      By another application of Lemma \ref{work2}$(i)$, for any $N\in[\nn]$, $$\Gamma(\zeta, 0, \varrho, N)\geqslant \sup_k \Gamma(\mu_k+k+1, 0, \varrho, N)=\infty.$$   By another application of Proposition \ref{home3}$(i)$, there cannot exist $N_0\in [\nn]$ and $0<C<\infty$ such that for all $N\in [N_0]$, $\mathbb{E}_{N|\mathcal{F}_\zeta[\mathcal{P}]}\varrho \in CB_S$.  This finishes the proof by contraposition in the $ \zeta<\omega_1$ case. 

It remains to prove the $\zeta=\omega_1$ case.  Fix $\varsigma\in B_R$ and $L\in[\nn]$.  Our goal is to show that there exist $n\in\nn$ and $M\in[L]$ such that $[M]^{<\omega}\subset \mathfrak{G}(0,n,\varsigma)$.  If condition $(1)$ of Theorem \ref{main3} is satisfied for $\zeta=\omega_1$, then it is also satisfied for all $\zeta<\omega_1$.  This means that for each $\zeta<\omega_1$, there exist $n_\zeta\in\nn$ and $M_\zeta\in[L]$ such that $\mathcal{F}_\zeta[\mathcal{P}]\upp M_\zeta\subset \mathfrak{G}(0, n_\zeta, \varsigma)\upp L$.  This implies that $$\sup_{n\in\nn} \text{rank}(\mathfrak{G}(0, n, \varsigma)\upp L) \geqslant \omega_1,$$ from which it follows that for some $n\in\nn$, $\mathfrak{G}(0, n, \varsigma)\upp L$ is ill-founded.  Therefore there exists some $M\in [L]$ such that $\mathfrak{G}(0,n, \varsigma)\upp L$ contains all initial segments of $M$. Since $\mathfrak{G}(0, n, \varsigma)\upp L$ contains all subsequences of its members, $[M]^{<\omega}\subset \mathfrak{G}(0, n, \varsigma)\upp L$. This completes the $\zeta=\omega_1$ case.

$(ii)$ First let us note that, by Lemma \ref{work2}$(v)$ and Proposition \ref{home3}$(iii)$, it is sufficient to prove that under the hypothesis of $(ii)$, $\Gamma(\zeta)<\infty$. We prove this by contradiction. 

Assume that for each $\varsigma\in B_R$ and $L\in[\nn]$, there exist $M\in[L]$ and $0<C<\infty$ such that for all $N\in [M]$, $\mathbb{E}_{N|\mathcal{F}_\zeta[\mathcal{P}]}\varsigma\in CB_S$, and assume also that $\Gamma(\zeta)=\infty$.   Let $$\beta=\min \{\mu\leqslant \omega_1: \Gamma(\beta)=\infty\}\leqslant \zeta.$$  Note that by Corollary \ref{pen2}, $\beta$ is a limit ordinal and for any $k<\omega$, $\{\Gamma(\mu, k): \mu<\beta\}$ is an unbounded subset of $[0, \infty)$.

Note that for any $0<D<\infty$ and $k<\omega$, there exist $\varsigma\in B_R$ and  $\mu<\beta$ such that $\Gamma(\mu, k, \varsigma)>D+1$.   By Proposition \ref{home3}$(ii)$, there exists $M\in [\nn]$ such that for all $N\in [M]$, $\mathbb{E}_{N|\mathcal{F}_\zeta[\mathcal{P}]}\varsigma\notin (D+1)B_{S_k}$.   Let $\varsigma=(x_n)_{n=1}^\infty$ and $\varrho=(x_{M(n)})_{n=1}^\infty$. By Lemma \ref{beer} applied with $Q=P$, $$\inf_{N\in [\nn]} \Gamma(\mu+1, k, \varrho, N)\geqslant D+1>D.$$   By another application of Proposition \ref{home3}$(ii)$, this implies that for any $L\in [\nn]$, there exists $M\in [L]$ such that for any $N\in [M]$, $\mathbb{E}_{N|\mathcal{F}_{\mu+1}[\mathcal{P}]} \varrho\notin DB_{S_k}$.    Therefore in this paragraph we have shown that for any $0<D<\infty$ and $k<\omega$, there exist $\mu<\beta$ and $\varrho\in B_R$ such that for any $L\in [\nn]$, there exists $M\in [L]$ such that  for all $N\in [M]$, $s_k(\mathbb{E}_{N|\mathcal{F}_\mu[\mathcal{P}]} \varrho)>D$.

Fix $D_1>4$.  Fix $\varrho_1\in B_R$ and $\mu_1<\beta$ such that for any $L\in[\nn]$, there exists $M\in [L]$ such that for all $N\in [M]$, $s_0(\mathbb{E}_{N|\mathcal{F}_{\mu_1}[\mathcal{P}]} \varrho_1)>D_1$.    Fix $M_1\in [\nn]$ and $0<C_1<\infty$ such that for any $N\in [M_1]$, $s(\mathbb{E}_{N|\mathcal{F}_\zeta[\mathcal{P}]}\varrho_1) \leqslant C_1$. 

Now assume that for some $l\in\nn$, constants $D_1, \ldots, D_l, C_1, \ldots, C_l$, sequences $\varrho_1, \ldots, \varrho_l\in B_R$, $M_1\supset \ldots \supset M_l$, and $\mu_1, \ldots, \mu_l<\beta$ have been chosen. Choose $D_{l+1}>4^{l+1}$ so large that \begin{enumerate}[(a)]\item $\sum_{k=1}^l \frac{k+C_k}{D_k^{1/2}} < D_{l+1}^{1/2}/3$, \item $\underset{1\leqslant k\leqslant l}{\max} \frac{k+\Gamma(\mu_k+1)}{D_k^{1/2}D_{l+1}^{1/2}}<\frac{1}{3\cdot 2^{l+1}}$. \end{enumerate}   We may select $\varrho_{l+1}\in B_R$ and $\mu_{l+1}<\beta$ such that for any $L\in [\nn]$, there exists $M\in[L]$ such that for all $N\in [M]$, $s_l(\mathbb{E}_{N|\mathcal{F}_{\mu_{l+1}}[\mathcal{P}]} \varrho_{l+1})>D_{l+1}$.  Choose $K\in [M_l]$ such that for all $N\in [K]$, $s_l(\mathbb{E}_{N|\mathcal{F}_{\mu_{l+1}}[\mathcal{P}]} \varrho_{l+1})>D_{l+1}$.    By hypothesis, we may fix $K_0\in [K]$ and $C_{l+1}$ such that for all $N\in [K_0]$,  $s( \mathbb{E}_{N|\mathcal{F}_\zeta[\mathcal{P}]}\varrho_{l+1}) \leqslant C_{l+1}$.    Now using Proposition \ref{home3}$(iii)$ recursively, we may find $K_0\supset K_1\supset \ldots \supset K_l$ such that for each $1\leqslant k\leqslant l$ and each $N\in [K_k]$, $s_l(\mathbb{E}_{N|\mathcal{F}_{\mu_k}[\mathcal{P}]}\varrho_{l+1}) \leqslant \Gamma(\mu_k+1)+1$.    Let $M_{l+1}=K_l$.    This completes the recursive construction.

Now let $\varrho=\sum_{l=1}^\infty \frac{1}{D_l^{1/2}}\varrho_l$. Since $$\sum_{l=1}^\infty \frac{1}{D_l^{1/2}}\leqslant \sum_{l=1}^\infty \frac{1}{2^l}=1$$ and since $R$ is a Banach space, the series above converges and $\varrho\in B_R$.  Fix $L(1)<L(2)<\ldots$ such that for all $l\in\nn$, $L(l)\in M_l$.    By hypothesis, there exist $M\in [L]$ and $0<C<\infty$ such that for all $N\in [M]$, $s(\mathbb{E}_{N|\mathcal{F}_\zeta[\mathcal{P}]}) \leqslant C$. By the bimonotone property, this implies that for any $G=_\mathcal{P}\cup_{n=1}^t G_n\in \mathcal{F}_\zeta[\mathcal{P}]\upp N$, $s(\mathbb{E}_G \varrho)\leqslant C$.

Choose $k$ so large that $D_k^{1/2}/3>C$. By the almost monotone property, there exists $j\in\nn$ such that if $j\leqslant G\in \mathcal{F}_{\mu_k+k-1}$, then $G\in \mathcal{F}_\zeta$. From this it follows that if $j\leqslant G\in\mathcal{F}_{\mu_k+k-1}[\mathcal{P}]$, then $G\in \mathcal{F}_\zeta[\mathcal{P}]$.     For each $l\in\nn$, let $J_l=\cup_{n=1}^{l-1}J^l_n$, where $J^l_1, \ldots, J^l_{l-1}\subset \nn$ are any sets such that $j\leqslant J^l_1 <\ldots <J^l_{l-1}$ and $J^l_n\in MAX(\mathcal{P})\upp M_l$.       Choose $j\leqslant N(1)<N(2)<\ldots $such that for each $l\in\nn$, $J_l<N(l)$.   Let $G=N|\mathcal{F}_{\mu_k+k-1}[\mathcal{P}]\in MAX(\mathcal{F}_{\mu_k+k-1}[\mathcal{P}])\upp N$. Since $j\leqslant G\in \mathcal{F}_{\mu_k+k-1}[\mathcal{P}]\upp N$ and $N\in [M]$, $G\in \mathcal{F}_\zeta[\mathcal{P}]\upp M$, and $s(\mathbb{E}_G\varrho)\leqslant C$. The rest of the proof involves providing  estimates for $s(\mathbb{E}_G\varrho_l)$ for each $l\in\nn$, and combining these estimates to contradict $s(\mathbb{E}_G\varrho)\leqslant C$.  We perform these estimates for $l=k$, $l<k$, and $l>k$.

In the remainder of the proof, we use the convention that for sets $I_1, \ldots, I_r$, $\cup_{n=i}^j I_n=\varnothing$ if $i>j$.  Write $G=_\mathcal{P} \cup_{n=1}^t G_n$.  Note that since $G=_\mathcal{P}\cup_{n=1}^t G_n\in MAX(\mathcal{F}_{\mu_k+k-1}[\mathcal{P}])$, it follows that $\cup_{n=k}^t G_n\in MAX(\mathcal{F}_{\mu_k}[\mathcal{P}])$, and by heredity, $\cup_{n=l}^t G_n\in \mathcal{F}_{\mu_k}[\mathcal{P}]$ for each $l\geqslant k$.

For each $l\in\nn$, let $$E_l=\bigcup_{n=1}^{\min \{t, l-1\} } G_n \text{\ \ and\ \ }F_l=\bigcup_{n=l}^t G_n.$$  Note that $E_l<F_l$, $G=E_l\cup F_l$,  and by the triangle inequality, $$s(\mathbb{E}_{E_l}\varrho_l)\leqslant \sum_{n=1}^{\min \{t, l-1\}} s(\mathbb{E}_{E_l} \varrho_l(n)) \leqslant l-1.$$ Note also that since $M(l)\leqslant F_l\in [M]^{<\omega}$, $F_l\in [M_l]^{<\omega}$. By the previous paragraph, $F_l\in \mathcal{F}_{\mu_k}[\mathcal{P}]$ for each $l\geqslant k$, from which it follows that $F_l\in \mathcal{F}_{\mu_k}[\mathcal{P}]\upp M_l$ for each $l\geqslant k$.    Also, by the previous paragraph, $F_k\in MAX(\mathcal{F}_{\mu_k}[\mathcal{P}])$, so $F_k\in MAX(\mathcal{F}_{\mu_k}[\mathcal{P}])\upp M_k$.

  For convenience, if $F_l=\varnothing$, let $\mathbb{E}_{F_l}\varrho_l$ be the zero sequence.  By the permanence property, $\mathbb{E}_{F_l}\varrho_l(n)=\mathbb{E}_G\varrho_l(n+l-1)$ for each $n\in\nn$, so that \begin{align*} s(\mathbb{E}_G \varrho_l) &  =s\Bigl(\sum_{n=1}^\infty \mathbb{E}_G \varrho_l(n)\otimes e_n\Bigr) \leqslant s\Bigl(\sum_{n=1}^{\min \{t, l-1\}} \mathbb{E}_{E_l} \varrho_l(n)\otimes e_n\Bigr) + s\Bigl(\sum_{n=l}^\infty \mathbb{E}_{F_l}\varrho_l(n)\otimes e_{n+l-1}\Bigr) \\ & = s(\mathbb{E}_{E_l}\varrho_l)+s_{l-1}(\mathbb{E}_{F_l}\varrho_l) \leqslant l-1 +s_{l-1}(\mathbb{E}_{F_l} \varrho_l).\end{align*} If $l>k$, our choice of $M_l$  and $F_l\in \mathcal{F}_{\mu_k}[\mathcal{P}]\upp M_l$ can be combined with the previous inequality to find that $$s(\mathbb{E}_G\varrho_l) \leqslant l-1+s_{l-1}(\mathbb{E}_{F_l}\varrho_l) \leqslant l-1+\Gamma(\mu_k+1)+1=l+\Gamma(\mu_k+1).$$  Combining this with (b), $$s(\mathbb{E}_G\varrho_l) \leqslant \frac{D_l^{1/2}D^{1/2}_k}{3\cdot 2^l}.$$    	If $k=l$, since $F_k\in MAX(\mathcal{F}_{\mu_k}[\mathcal{P}])\upp M_k$, our choice of $M_k$ and bimonotonicity yield that \begin{align*} D_k &  < s_{k-1}(\mathbb{E}_{F_k}\varrho_k)=s\Bigl(\sum_{n=1}^\infty \mathbb{E}_{F_k} \varrho_k(n)\otimes e_{n+k-1}\Bigr) \\ & = s\Bigl(\sum_{n=k}^\infty \mathbb{E}_G \varrho_k(n)\otimes e_n\Bigr) \leqslant s\Bigl(\sum_{n=1}^\infty \mathbb{E}_G\varrho_k(n)\otimes e_n\Bigr)= s(\mathbb{E}_G\varrho_k). \end{align*} This implies that $t\geqslant l$.

Now consider $l<k$. Recall the set $J_l$ chosen prior to choosing $N$. The set $J_l$ is the union of $l-1$ consecutive, maximal members of $\mathcal{P}\upp M_l$, $j\leqslant J_l$, and $\max J_l < N(l)\leqslant F_l\in [M]^{<\omega}$. Since $M(l)\leqslant N(l)\leqslant F_l\in [M]^{<\omega}$, $F_l\in [M_l]^{<\omega}$.  Then \begin{align*} j & \leqslant J_l\cup F_l= \Bigl[J_l\cup \bigl(\bigcup_{n=l}^{k-1}G_n\bigr)\Bigr]\cup F_k \in \mathcal{F}_{\mu_k+k-1}[\mathcal{P}]\upp M_l \subset \mathcal{F}_\zeta[\mathcal{P}]\upp M_l. \end{align*} Here we have used that $J_l<\cup_{n=l}^{k-1}G_n<F_k$, each of these three sets is a subset of $M_l$, $F_k\in \mathcal{F}_{\mu_k}[\mathcal{P}]$, and $J_l\cup (\cup_{n=l}^{k-1}G_n)$ is a union of $k-1$ successive members of $MAX(\mathcal{P})$.

This implies that  $s(\mathbb{E}_{J_l\cup F_l} \varrho_l)\leqslant C_l$ by our choice of $M_l$.   By the permanence property, for all $n\geqslant l$, $$\mathbb{E}_G\varrho_l(n)=\mathbb{E}_{E_l\cup F_l} \varrho_l(n)=\mathbb{E}_{F_l}\varrho_l(n+l-1) = \mathbb{E}_{J_l\cup F_l}\varrho_l(n).$$   Therefore \begin{align*} s(\mathbb{E}_G\varrho_l) & =s\Bigl(\sum_{n=1}^\infty \mathbb{E}_G\varrho_l(n)\otimes e_n\Bigr)\leqslant l-1+s\Bigl(\sum_{n=l}^\infty \mathbb{E}_G\varrho_l(n)\otimes e_n\Bigr) \\ &  = l-1+s\Bigl(\sum_{n=l}^\infty \mathbb{E}_{J_l\cup F_l} \varrho_l(n)\otimes e_n\Bigr) \leqslant  l+s(\mathbb{E}_{J_l\cup F_l} \varrho_l) \leqslant l+C_l. \end{align*}

Combining these estimates and using (a),  we find that \begin{align*}  C & \geqslant s(\mathbb{E}_G\varrho) \geqslant \frac{s(\mathbb{E}_G\varrho_k)}{D_k^{1/2}}-\sum_{l=1}^{k-1} \frac{s(\mathbb{E}_G \varrho_l)}{D_l^{1/2}} - \sum_{l=k+1}^\infty \frac{s(\mathbb{E}_G\varrho_l)}{D_l^{1/2}} \\ & > D_k^{1/2} - D_k^{1/2}/3 - \sum_{l=k+1}^\infty \frac{D_k^{1/2}}{3\cdot 2^l} \\ & \geqslant D_k^{1/2}/3>C.  \end{align*}  This contradiction finishes the proof.

\end{proof}

\section{Examples}

In this section, we wish to discuss the distinction of these properties for distinct $\xi$. Distinguishing $0<\xi<\omega_1$ from $\xi=0$ will establish the distinctness of our properties from that studied by Freeman and Knaust/Odell.   

For the remainder of this work, for a Banach space $X$ and $1<p\leqslant \infty$, we let $S_p$ denote the space of $(x_n)_{n=1}^\infty$ such that there exists a constant $C$ such that for any $(a_n)_{n=1}^\infty\in c_{00}$, $\|\sum_{n=1}^\infty  a_nx_n\|\leqslant C\|\sum_{n=1}^\infty a_ne_n\|_{\ell_p}$. We let $s_p((x_n)_{n=1}^\infty)$ denote the infimum of such $C$. As noted in Section $3$, $s_p$ is a bimonotone norm on $c_{00}(X)$ and $S_p$ is its natural domain. 

We also let $R=c_0^w(X)$, the space of weakly null sequences in $X$, endowed with the $\|\cdot\|_{\ell_\infty(X)}$ norm.  Since $R$ is closed in $\ell_\infty(X)$, $R$ is a subsequential space.   We let $\Gamma_p$ denote the $\Gamma$ function from the previous sections with this choice of $R$ and $s=s_p$, $S=S_p$.    Of course, these notations should depend on $X$, but this omission will cause no confusion.    For $\xi<\omega_1$, we let $\gamma_p(\xi)$ denote the minimum ordinal $\zeta$  such that for any $\xi$-homogeneous probability block $P$, $\Gamma^P_p(\zeta,0)=\infty$ if any such $\zeta$ exists, and $\gamma_p(\xi)=\omega_1$ if no such $\zeta$ exists.  We note that by Corollary \ref{shift1}, in order to compute $\gamma_p(\xi)$, it is sufficient to consider only $P=(\mathfrak{S}_\xi, \mathcal{S}_\xi)$, the repeated averages hierarchy. For concreteness, we consider only these probability blocks for the remainder.  We let $\mathbb{E}^\xi_N$ denote the convex blockings with respect to the probability block $(\mathfrak{S}_\xi, \mathcal{S}_\xi)$. The function $\Gamma^\xi$ is defined similarly. 

We recall that for a Banach space $X$ and $0<\xi<\omega_1$, we say a sequence $(x_n)_{n=1}^\infty$ is an $\ell_1^\xi$-\emph{spreading model} provided that it is bounded and $$0<\inf\Bigl\{\|x\|: E\in \mathcal{S}_\xi, 1=\sum_{n\in E} |a_n|, x=\sum_{n\in E} a_nx_n\Bigr\}.$$   We say the sequence $(x_n)_{n=1}^\infty$ is $\xi$-\emph{weakly null} provided it is weakly null and has no subsequence which is an $\ell_1^\xi$-spreading model. If $\varsigma=(x_n)_{n=1}^\infty\subset X$ is $\xi$-weakly null, then for any $\ee>0$ and $L\in[\nn]$, using \cite[Theorem $4.12$]{CN} as in the proof of \cite[Proposition $4.13$]{CN}, we can find $M\in[\nn]$ such that for all $N\in[M]$ and $n\in\nn$, $\|\mathbb{E}^\xi_N \varsigma(n)\|\leqslant \ee/2^n$.  Therefore for any $N\in[M]$, $$s_p(\mathbb{E}^\xi_N\varsigma)\leqslant \sum_{n=1}^\infty \|\mathbb{E}^\xi_N \varsigma(n)\|\leqslant \ee.$$  This shows that $\Gamma^\xi_p(\omega_1, 0, \varsigma)=0$.    That is, we have trivially small behavior with respect to blockings of $\varsigma$ at level $\xi$ of the repeated averages hierarchy when $\varsigma$ is $\xi$-weakly null. The analogy for the $\xi=0$, sequence/subsequence case would be the case that $\varsigma$ is a norm null sequence.  Therefore if $X$ has the property that every weakly null sequence in $X$  is $\xi$-weakly null (that is, if $X$ has the $\xi$-\emph{weak Banach-Saks property}), then $\Gamma^\xi_p(\omega_1)=0$. Thus the study of $\xi$-convex blocks in a $\xi$-weak Banach-Saks space is trivial. 

On the other hand, if $X$ is a Banach space which admits a weakly null sequence which is not $\xi+1$-weakly null (that is, if $X$ fails to have the $\xi+1$-weak Banach-Saks property), then $X$ admits a weakly null $\ell_1^{\xi+1}$-spreading model, say $\varsigma=(x_n)_{n=1}^\infty$.     Then $$0<\ee := \inf\Bigl\{\|x\|: E\in\mathcal{S}_{\xi+1}, 1=\sum_{n\in E}|a_n|, x=\sum_{n\in E}a_nx_n\Bigr\}.$$  Assume that $\mathcal{F}_k^M[\mathcal{S}_\xi]\subset \mathfrak{G}(0,C,\varsigma)$ for some $k\in\nn$, $M\in[\nn]$, and $0<C<\infty$.  Since $\mathcal{F}_k(M)=\mathcal{F}_k\upp M$, this simply means that for any $F_1<\ldots<F_k$, $F_n\in \mathcal{S}_\xi\upp M$, $\cup_{n=1}^k F_n\in \mathfrak{G}(0,C,\varsigma)$.     Then if $N\in[M]$ has $k\leqslant N$ and $\cup_{n=1}^k F_n=N|\mathcal{F}_k[\mathcal{S}_\xi]\in\mathcal{S}_{\xi+1}$, $$C\geqslant \Bigl\|\sum_{n=1}^k \frac{1}{k^{1/p}} \mathbb{E}_N\varsigma(n)\Bigr\|\geqslant \ee k/k^{1/p}.$$  Since this holds for any $k\in\nn$ and $M\in[\nn]$, it follows that $\Gamma^\xi_p(k)\geqslant \ee k /k^{1/p}\underset{k\to \infty}{\to}\infty$ and $\gamma_p(\xi)=\omega$.  This is the smallest possible value of $\gamma_p(\xi)$, and we see the opposite behavior to that in the previous paragraph.   

Therefore the only spaces with interesting behavior of the function $\gamma_p(\xi)$ are spaces in which every weakly null sequence is $\xi+1$-weakly null, but in which there exists a weakly null sequence which is not $\xi$-weakly null.   We now discuss a general method for constructing such a space.  Suppose that $H$ is a Banach space which is the completion of some norm $c_{00}$ such that the canonical $c_{00}$ basis is normalized, $1$-unconditional, shrinking basis for $H$. For $x=\sum_{n=1}^\infty a_ie_i\in c_{00}$, we let $Ex=\sum_{i\in E} a_ie_i$.  For $x\in c_{00}$, we also let $\ran(x)$ be the smallest integer interval containing $\{i\in\nn: a_i\neq 0\}$.   Let us define $H_\xi$ to be the completion of $c_{00}$ with respect to the norm $$\|x\|_{H_\xi}=\sup\Bigl\{\Bigl\|\sum_{n=1}^\infty \|E_nx\|_{\ell_1}e_{\max E_n}\Bigr\|_H: E_1<E_2<\ldots, E_n\in \mathcal{S}_\xi\Bigr\}.$$   It is easy to see that $(e_n)_{n=1}^\infty$ is a normalized, $1$-unconditional basis for $H_\xi$.

Let us also assume that the canonical basis of $H$ is \emph{block stable} and $1$-\emph{left dominant}.  By block stable, we mean there exists a constant $B\geqslant 1$ such that for any normalized block sequences $(x_n)_{n=1}^\infty, (y_n)_{n=1}^\infty$ such that $$\max\{\max \ran(x_n), \max \ran(y_n)\} < \min \{\min \ran(x_{n+1}), \min \ran(y_{n+1})\}$$ for all $n\in\nn$, then  $$\frac{1}{B}\|\sum_{n=1}^\infty a_nx_n\Bigr\|\leqslant \Big\|\sum_{n=1}^\infty a_ny_n\Bigr\|\leqslant B\Bigl\|\sum_{n=1}^\infty a_nx_n\Bigr\|$$ for all $(a_n)_{n=1}^\infty\in c_{00}$. If we wish to emphasize the constant $B$, we say $H$ is $B$-\emph{block stable}.  By $1$-left dominant, we mean that for any increasing sequences $(m_n)_{n=1}^\infty, (l_n)_{n=1}^\infty$ of positive integers such that $m_n\leqslant l_n$ for all $n\in\nn$, it follows that  $$\Bigl\|\sum_{n=1}^\infty a_ne_{l_n}\Bigr\|\leqslant \Bigl\|\sum_{n=1}^\infty a_ne_{m_n}\Bigr\|$$ for all $(a_n)_{n=1}^\infty\in c_{00}$.

We will use the following.

\begin{lemma} Suppose that the canonical basis of $H$ is normalized, $1$-unconditional, shrinking, $B$-block stable, and $1$-left dominant.  Then for $\xi<\omega_1$, $C>B$, any $\varsigma=(x_n)_{n=1}^\infty\subset B_{H_\xi}$ which is either a weakly null sequence or a block sequence, and any $A,L\in[\nn]$, there exists $M\in[L]$ such that for any $N\in[M]$, $\mathbb{E}_N^\xi \varsigma(n)$ is $C$-dominated by $(e_{A(n)})_{n=1}^t$. 

\label{exam1}
\end{lemma}

\begin{proof} Fix $\varsigma=(x_n)_{n=1}^\infty \subset B_{H_\xi}$ which is either a weakly null sequence or a block sequence. Fix $A,L\in[\nn]$.  Fix $\ee>0$ such that $2\ee+B<C$ and  a decreasing sequence $(\ee_n)_{n=1}^\infty \subset (0,1)$ such that $2\sum_{n=1}^\infty \sum_{m=n}^\infty \ee_m<\ee$.      If $(x_n)_{n=1}^\infty$ is weakly null, then by replacing $L$ with a subset thereof, we may assume that there exists a  sequence $(y_n)_{n=1}^\infty \subset B_{H_\xi}$ such that $(y_{L(n)})_{n=1}^\infty$ is a block sequence,  for any $n\in\nn$, $\|x_n-y_n\|_{H_\xi}<\ee_n$, and $\min \ran(y_{L(n)})>A(n)$ for all $n\in\nn$.    If $(x_n)_{n=1}^\infty$ is already a block sequence, we let $y_n=x_n$ and, by replacing $L$ with a subset thereof, assume $\min \ran(y_{L(n)})>A(n)$.   In either case, we have a  sequence $(y_n)_{n=1}^\infty \subset B_{H_\xi}$ such that $(y_{L(n)})_{n=1}^\infty$ is a block sequence, and for each $n\in\nn$,   $\|y_n-x_n\|_{H_\xi}<\ee_n$ and $\min \ran(y_{L(n)})>A(n)$.  Let $\varrho=(y_n)_{n=1}^\infty$.

Let $$\mathcal{G}=\{E: (\max \supp(y_{L(n)}))_{n\in E}\in \mathcal{S}_\xi\},$$ which is a regular family with rank $\omega^\xi+1$ \cite[Proposition $3.1$]{Con}. Let $M_1=L$ and let $r_1\in\nn$ be such that $M_1(1)=L(r_1)$.

Now assume that $M_1\supset \ldots \supset M_t\in[L]$, $r_1<\ldots <r_t$ have been chosen.   Combining \cite[Lemma $4.3$]{CN} and \cite[Corollary $4.8$]{CN}, there exists $M_{t+1}\in [M_t]$ such that $$\sup \{\mathbb{S}^\xi_{N,1}(G): G\in \mathcal{G}, N\in [M_{t+1}], \min G\leqslant r_t \} \leqslant \ee_{t+1}.$$ Let $r_{t+1}\in \nn$ be such that $M_{t+1}(t+1)=L(r_{t+1})$.      This completes the recursive process.  Let $M(t)=M_t(t)$ for each $t\in\nn$ and note that $M\in [L]$.

Fix $N\in[M]$.     Fix scalars $(a_n)_{n=1}^\infty\in c_{00}$. Let $$x= \sum_{n=1}^\infty a_n\mathbb{E}^\xi_N \varsigma(n)$$   and $$y=\sum_{n=1}^\infty a_n\mathbb{E}^\xi_N \varrho(n).$$  We omit the trivial case $0=a_1=a_2=\ldots$ and assume that $a:=\max_{n\in\nn}|a_n|>0$.   Let $I_1=[1, \max \ran \mathbb{E}^\xi_N \varrho(1)]$ and for $n\in\nn$, $$I_{n+1}=(\max \ran( \mathbb{E}^\xi_N \varrho(n)), \max \ran(\mathbb{E}^\xi_N \varrho(n+1)].$$    Fix $E_1<E_2<\ldots$, $E_j\in \mathcal{S}_\xi$, such that $$\|y\|_{H_\xi}=\Bigl\|\sum_{j=1}^s \|E_jy\|_{\ell_1} e_{\max E_j}\Bigr\|_H.$$   Note that by omitting superfluous $E_j$, we may assume that for each $1\leqslant j\leqslant s$, there exists at least one value of $n\in\nn$ such that $E_j\mathbb{E}^\xi_N \varrho(n)\neq 0$.

 For each $n\in\nn$, let $$T_n=\{j\in \nn: \min E_j\in I_n\}.$$   Note that for each $n\in\nn$ and $j\in T_n$, $E_j\mathbb{E}^\xi_N \varrho(m)=0$ for all $m<n$.  Note also that for each $n\in\nn$ and $j\in T_n\setminus (\max T_n)$, $E_j\mathbb{E}^\xi_N \varrho(m)=0$ for all $m>n$.    From this it follows that \begin{align*} \|x\|_{H_\xi} & \leqslant a \ee+\|y\|_{H_\xi}  \\ & \leqslant a\ee+\Bigl\|\sum_{n=1}^\infty |a_n|\sum_{j\in T_n} \|E_j\mathbb{E}^\xi_N \varrho(n)\|_{\ell_1} e_{\max E_j}\Bigr\|_H  + a\sum_{n=1}^\infty\sum_{m=n+1}^\infty \|E_{\max T_n} \mathbb{E}^\xi_N \varrho(m)\|_{\ell_1}. \end{align*}  Here we agree to the convention that if $T_n=\varnothing$, $E_{\max T_n}$ denotes the zero projection.

\begin{claimp} $\sum_{n=1}^\infty \sum_{m=n+1}^\infty \|E_{\max T_n} \mathbb{E}^\xi_N \varrho(m)\|_{\ell_1}\leqslant \sum_{n=1}^\infty \sum_{m=n+1}^\infty \ee_m$, from which it follows that $$a\sum_{n=1}^\infty \sum_{m=n+1}^\infty \|E_{\max T_n} \mathbb{E}^\xi_N \varrho(m)\|_{\ell_1} \leqslant a\ee \leqslant \ee\Bigl\|\sum_{n=1}^\infty a_ne_{A(n)}\Bigr\|_H.$$   

\end{claimp}

\begin{claimp} Let $h_n=\sum_{j\in T_n} \|E_j\mathbb{E}^\xi_N \varrho(n)\|_{\ell_1} e_{\max E_j}$. Then $$ \Bigl\|\sum_{n=1}^\infty |a_n|\sum_{j\in T_n} \|E_j\mathbb{E}^\xi_N \varrho(n)\|_{\ell_1} e_{\max E_j}\Bigr\|_H  = \Bigl\|\sum_{n=1}^\infty |a_n|h_n\Bigr\|_H \leqslant B\Bigl\|\sum_{n=1}^\infty a_n e_{A(n)}\Bigr\|_H.$$  

\end{claimp}

Let us assume the claims and see how this finishes the proof, assuming that $\ee>0$ was chosen small enough that $2\ee+B<C$. We note that since the basis of $H$ is normalized and bimonotone, $a\ee\leqslant \|\sum_{n=1}^\infty a_n e_{A(n)}\|_H$, so the estimate above combined with the two claims  yields that \begin{align*} \|x\|_{H_\xi} & \leqslant a \ee+\|y\|_{H_\xi}  \\ & \leqslant a\ee+\Bigl\|\sum_{n=1}^\infty |a_n|\sum_{j\in T_n} \|E_j\mathbb{E}^\xi_N \varrho(n)\|_{\ell_1} e_{\max E_j}\Bigr\|_H  + a\sum_{n=1}^\infty \sum_{m=n+1}^\infty \|E_{\max T_n} \mathbb{E}^\xi_N \varrho(m)\|_{\ell_1} \\ & \leqslant (2\ee+B)\Bigl\|\sum_{n=1}^\infty a_ne_{A(n)}\Bigr\|_H.  \end{align*}

We now prove Claim $1$.   In the proof, recall that $r_t\in\nn$ has the property that $M(t)=M_t(t)=L(r_t)$ for each $t\in\nn$.  Note that it is sufficient to show that for each $m,n\in\nn$ with $m<n$, $\|E_{\max T_n} \mathbb{E}^\xi_N \varrho(m)\|_{\ell_1}\leqslant \ee_m$.    To that end, fix such an $m,n$. The result is trivial if $T_n=\varnothing$ or if $E_{\max T_n}\mathbb{E}^\xi_N \varrho(m)=0$, so assume $E_{\max T_n}\mathbb{E}^\xi_N\varrho(m)\neq \varnothing$.     Let $t\in\nn$ be such that $$M(t+1)=\min \supp(\mathbb{S}_{N,m})$$ and note that $t+1\geqslant m$. Let $N_0=N\setminus \cup_{i=1}^{m-1}\supp(\mathbb{S}^\xi_{N,i})\in [M_{t+1}]$ and note that by the permanence property, $\mathbb{S}^\xi_{N,m}=\mathbb{S}^\xi_{N_0,1}$.       Let $$J=\{i\in\supp(\mathbb{S}_{N,m}): E_{\max T_n}y_i\neq 0\}\in [N]$$ and note that, since $N\subset L$, $J=L(J_0)$ for some $J_0\in [\nn]^{<\omega}$. For each $j\in J$, fix some $m_j\in E_{\max T_n}\cap \ran(y_j)$.       By definition of $T_n$, $$\min E_{\max T_n}  \leqslant \max I_n = \max\{\max \ran(y_r): r\in \supp(\mathbb{S}^\xi_{N,n})\} \leqslant \max \ran(y_{M(t)})=\max \ran(y_{L(r_t)}).$$   Let $G=(r_t)\cup J_0$ and note that $$(\max \ran(y_{L(j)}))_{j\in G} = (\max \ran(y_{L(r_t)}))\cup (\max \ran(y_i):j\in J),$$ which is a spread of $$(\min E_{\max T_n})\cup (m_j:j\in J)\subset E_{\max T_n}\in \mathcal{S}_\xi.$$  Therefore $G\in \mathcal{G}$ and $\min G\leqslant r_t$. By our choice of $M_{t+1}$ and since $N_0\in [M_{t+1}]$, $$\mathbb{S}^\xi_{N,m}(G)=\mathbb{S}^\xi_{N_0,1}(G)\leqslant \ee_{t+1}\leqslant \ee_m.$$    Since $E_{\max T_m}\in \mathcal{S}_\xi$ and $y_i\in B_{H_\xi}$ for all $i\in\nn$, \begin{align*} \|E_{\max T_n} \mathbb{E}^\xi_{N,m}\|_{\ell_1} & = \sum_{j\in J} \mathbb{S}^\xi_{N,m}(j)\|E_{\max T_n}y_j\|_{\ell_1} \leqslant \sum_{j\in J} \mathbb{S}^\xi_{N,m}(j) \\ & = \sum_{j\in J_0}\mathbb{S}^\xi_{N,m}(L(j))=\mathbb{S}^\xi_{N,m}(G)\leqslant\ee_m. \end{align*}

We now prove Claim $2$.  Let $h_n$ be defined as in Claim $2$.  For $n\in\nn$ and $j\in T_n$, let $G_j=E_j\cap I_n$.  Note that since $\text{supp}(\mathbb{E}^\xi_N\varrho(n))\subset I_n$, $G_j\mathbb{E}^\xi_N\varrho(n)=E_j\mathbb{E}^\xi_N\varrho(n)$ for all $n\in\nn$ and $j\in T_n$. Furthermore, it follows from the definition of $T_n$ that for $j\in T_n$, $\min E_j\in T_n$, so that  $G_j\neq \varnothing$.    Also,  $\max G_j\leqslant \max E_j$ for each $j\in T_n$.   For $n\in\nn$, let $$S_n=\{j\in T_n: E_j\mathbb{E}^\xi_N\varrho(n)\neq 0\}$$    and note that $$h_n=\sum_{j\in T_n} \|E_j\mathbb{E}^\xi_N \varrho(n)\|_{\ell_1} e_{\max E_j}= \sum_{j\in S_n}\|E_j\mathbb{E}^\xi_N\varrho(n)\|_{\ell_1} e_{\max E_j}.$$  For $n\in\nn$, let $$g_n=\sum_{j\in S_n}\|G_j\mathbb{E}^\xi_N\varrho(n)\|_{\ell_1} e_{\max G_j}$$ and note that by $1$-left dominance, $$\Bigl\|\sum_{n=1}^\infty |a_n|h_n\Bigr\|_H \leqslant \Bigl\|\sum_{n=1}^\infty |a_n|g_n\Bigr\|_H.$$    Note that since $\ran(g_n)\subset I_n$, $(g_n)_{n=1}^\infty$ (after omitting any zero vectors if necessary) is a block sequence in $H$. Note also that $$\|g_n\|_H =\Bigl\|\sum_{j\in S_n}\|G_j\mathbb{E}^\xi_N\varrho(n)\|_{\ell_1} e_{\max G_j}\Bigr\|_H \leqslant \|\mathbb{E}^\xi_N \varrho(n)\|_{H_\xi}\leqslant 1.$$  

For each $n\in\nn$ and $j\in S_n$, since $N\in[\nn]$, $L(r)\in \cup_{s=1}^r \supp(\mathbb{S}^\xi_{N,n})$ for each $r\in\nn$. Therefore for each $n\in\nn$ and $j\in S_n$,  \begin{align*} \max E_j & \geqslant \min \{\min \ran(y_r): r\in \text{supp}(\mathbb{S}^\xi_{N,n})\} \geqslant \min \ran(y_{\min \supp(\mathbb{S}^\xi_{N,n})}) \\ & \geqslant \min \ran(y_{L(n)}) >A(n).\end{align*} Therefore $\min \ran(g_n)\geqslant A(n)$.  By $B$-block stability and $1$-unconditionality, $(g_n)_{n: g_n\neq 0}$ is $B$-dominated by $(e_{\min \ran(g_n)})_{n:g_n\neq 0}$, which is $1$-dominated by $(e_{A(n)})_{n: g_n\neq 0}$ by $1$-left dominance.  Therefore  \begin{align*}  \Bigl\|\sum_{n=1}^\infty |a_n|\sum_{j\in T_n} \|E_j\mathbb{E}^\xi_N \varrho(n)\|_{\ell_1} e_{\max E_j}\Bigr\|_H & = \Bigl\|\sum_{n=1}^\infty |a_n|h_n\Bigr\|_H \leqslant \Bigl\|\sum_{n=1}^\infty |a_n|g_n\Bigr\|_H \\ & \leqslant B\Bigl\|\sum_{n=1}^\infty |a_n|e_{A(n)}\Bigr\|. \end{align*} This gives Claim $2$. 

\end{proof}

\begin{corollary} If $\zeta$, $C$,  $1<p\leqslant \infty$ are as in Lemma \ref{exam1}, then the canonical basis of $H_\xi$ is shrinking. Furthermore,  the canonical basis of $H_\xi$ is weakly null and not $\xi$-weakly null.

\label{shrink}
\end{corollary}

\begin{proof} Lemma \ref{exam1} yields that any bounded block sequence in $H_\xi$ has a convex block sequence which is dominated by a weakly null sequence, which implies that any bounded block sequence in $H_\xi$ is weakly null. Therefore the canonical basis of $H_\xi$ is shrinking. 

It follows from the previous paragraph that the canonical basis of $H_\xi$ is weakly null. It is evident that the canonical basis of $H_\xi$ is such that for any $E\in \mathcal{S}_\xi$ and scalars $(a_n)_{n\in E}$, $$\Bigl\|\sum_{n\in E}a_ne_n\Bigr\|_{H_\xi} \geqslant \Bigl\|\bigl(\sum_{n\in E}|a_n|\bigr)e_{\max E_n}\Bigr\|_H = \sum_{n\in E}|a_n|.$$   Therefore the canonical $H_\xi$ basis is an $\ell_1^\xi$-spreading model, and therefore not $\xi$-weakly null.

\end{proof}

We recall that for an ordinal $1\leqslant \mu<\omega_1$ and $0<\vartheta<1$, the \emph{Figiel-Johnson Tsirelson space} $T_{\mu, \vartheta}$ is the space which is the completion of $c_{00}$ with respect to the implicitly defined norm $$\|x\|_{T_{\mu,\vartheta}}= \max\Bigl\{\|x\|_{c_0}, \sup\bigl\{\vartheta\sum_{n=1}^t \|I_nx\|_{T_{\mu,\vartheta}}: I_1<\ldots <I_n, (\min I_n)_{n=1}^t\in \mathcal{S}_\mu\bigr\}\Bigr\}.$$   We let $1/p+1/q=1$ and let $T^{(q)}_{\mu, \vartheta}$ denote the $q$-convexification of $T_{\mu,\vartheta}$. It is easy to see that for any block sequence $(x_n)_{n=1}^t\in T^{(q)}_{\mu,\vartheta}$ such that $(\min \ran(x_n))_{n=1}^t\in \mathcal{S}_\mu$, $$\vartheta^{1/q}\Bigl(\sum_{n=1}^t \|x_n\|_{T_{\mu,\vartheta}}^q\Bigr)^{1/q} \leqslant \Bigl\|\sum_{n=1}^t x_n\Bigr\|_{T_{\mu,\vartheta}}\leqslant \Bigl(\sum_{n=1}^t \|x_n\|_{T_{\mu,\vartheta}}^q\Bigr)^{1/q}.$$   Recall that for non-empty regular families $\mathcal{F}, \mathcal{G}$, $$\mathcal{F}[\mathcal{G}]=\{\varnothing\}\cup \Bigl\{\bigcup_{n=1}^t F_n: F_1<\ldots <F_t, \varnothing\neq F_n\in\mathcal{G}, (\min F_n)_{n=1}^t \in \mathcal{F}\Bigr\}.$$ If $\text{rank}(\mathcal{F})=\mu_1+1$ and $\text{rank}(\mathcal{G})=\mu_2+1$, then $\text{rank}(\mathcal{F}[\mathcal{G}])=\mu_2\mu_1+1$.     We also define $\mathcal{F}^{\otimes 1}=\mathcal{F}$ and $\mathcal{F}^{\otimes m+1}= \mathcal{F}[\mathcal{F}^{\otimes m}]$ for $m\in\nn$. Since $\text{rank}(\mathcal{S}_\mu)=\omega^\mu+1$, it follows that $\text{rank}(\mathcal{S}_\mu^{\otimes m})=\omega^{\mu m}+1$.  It is easy to see by induction that for any $m\in\nn$ and any block sequence $(x_n)_{n=1}^t\subset T^{(q)}_{\mu,\vartheta}$ such that $(\min \ran(x_n))_{n=1}^t\in \mathcal{S}_\mu^{\otimes m}$, $$\vartheta^{m/q}\Bigl(\sum_{n=1}^t \|x_n\|_{T_{\mu,\vartheta}^{(q)}}^q\Bigr)^{1/q} \leqslant \Bigl\|\sum_{n=1}^t x_n\Bigr\|_{T_{\mu,\vartheta}^{(q)}}\leqslant \Bigl(\sum_{n=1}^t \|x_n\|_{T_{\mu,\vartheta}^{(q)}}^q\Bigr)^{1/q}.$$

An easy duality argument yields that, if $H=(T_{\mu, \vartheta}^{(q)})^*$, then for any $m\in\nn$ and any block sequence $(x_n)_{n=1}^t\in H$ such that $(\min \ran(x_n))_{n=1}^t\in \mathcal{S}_\mu^{\otimes m}$, $$\Bigl(\sum_{n=1}^t \|x_n\|_H^p\Bigr)^{1/p} \leqslant \Bigl\|\sum_{n=1}^t x_n\Bigr\|_H\leqslant \vartheta^{-m/q}\Bigl(\sum_{n=1}^t \|x_n\|_H^p\Bigr)^{1/p},$$ with the $\ell_p$ norm replaced by the maximum if $p=\infty$.

Before completing our examples, we isolate the following piece of bookeeping.   

\begin{lemma}  Fix $\xi<\omega_1$,  For any $M\in[\nn]$, there exists $T\in [M]$ such that for any $F=_{\mathcal{S}_\xi}\cup_{n=1}^t F_n\in \mathcal{F}_{\omega^\mu}^T[\mathcal{S}_\xi] $, there exist $N\in[M]$ and $H\in \mathcal{F}_{\omega^\mu}$ such that $\cup_{n=1}^t F_n=\cup_{n\in H}\supp(\mathbb{S}^\xi_{N,n})$.

\label{veryeasy}
\end{lemma}

\begin{proof} Let $K(1)=1$ and $K(p+1)=K(p)+p+1> p+1$.     Fix $E_1<E_2<\ldots$ with $E_n\in MAX(\mathcal{S}_\xi)\upp M$. Let $T(n)=\min E_{K(n)}\in M$ for all $n\in\nn$.   Note that $T\in [M]$.   Fix $F=_{\mathcal{S}_\xi}\cup_{n=1}^t F_n\in \mathcal{F}_{\omega^\mu}^T[\mathcal{S}_\xi]$ and note that if $\min F_n= T(i_n)$, then $H=(i_n)_{n=1}^t\in \mathcal{F}_{\omega^\mu}$, since $(\min F_n)_{n=1}^t\in \mathcal{F}_{\omega^\mu}^T$.

For $n=1, \ldots, t$, let $j_n$ be such that $T(j_n)=\max F_n$ and note that $i_1\leqslant j_1<i_2\leqslant j_2<\ldots$.    Since $\min F_1=T(i_1)=\min E_{K(i_1)} \geqslant K(i_1)\geqslant i_1$, we can select $J_1\subset (1, K(i_1))$ with $|J_1|=i_1-1$. Let $G_1=\cup_{i\in J_1} E_i$ and note that $G_1$ is a union of $i_1-1$ consecutive, maximal members of $\mathcal{S}_\xi\upp M$, and $G_1<F_1$.    Therefore $G_1\cup F_1$ is a union of $i_1$ consecutive, maximal members of $\mathcal{S}_\xi\upp M$ and $F_1$ is the last of those consecutive sets.   

Next, suppose that $G_1<F_1<\ldots <G_n<F_n$ have been chosen such that $G_1\cup F_1\cup \ldots \cup G_n\cup F_n$ is the union of $i_n$ consecutive, maximal members of $\mathcal{S}_\xi\upp M$ and $F_m$ is the $i_m^{th}$ of those sets for each $m=1, \ldots, n$.  If $n=t$, let $N_0\in [M]$ be such that $F_t<N_0$, and we are done with the recursive process.  In this case, we let $N=G_1\cup F_1\cup \ldots \cup G_t\cup F_t\cup N_0$.  If $n<t$, we complete the recursive step as follows: Since $j_n<i_{n+1}$, $$K(i_{n+1})-K(j_n)\geqslant K(i_{n+1})-K(i_{n+1}-1)=i_{n+1}.$$   Therefore we can choose a subset $J_{n+1}$ of $(K(j_n), K(i_{n+1}))$ of cardinality $i_{n+1}-i_n-1$. Let $G_{n+1}=\cup_{i\in J_{n+1}}E_i$ and note that $F_n<G_{n+1}<F_{n+1}$ and $G_{n+1}$ is a union of $i_{n+1}-i_n-1$ consecutive, maximal members of $\mathcal{S}_\xi\upp M$.  From this it follows that $G_1\cup F_1\cup \ldots \cup F_n\cup G_{n+1}\cup F_{n+1}$ is a union of $i_{n+1}$ consecutive, maximal members of $\mathcal{S}_\xi\upp M$.    

Now if $N=G_1\cup F_1\cup \ldots \cup G_n\cup F_n\cup N_0\in[M]$ as above, then $\cup_{n=1}^t F_n=\cup_{n=1}^t \supp(\mathbb{S}^\xi_{N,i_n})=\cup_{n\in H}\supp(\mathbb{S}^\xi_{N,n})$, as desired.

\end{proof}

\begin{corollary} Fix $1<p\leqslant \infty$ and let $1/p+1/q=1$. \begin{enumerate}[(i)]\item If $H=\ell_p$ (resp. $c_0$ if $p=\infty$), then with $X=H_\xi$, $R=c_0^w(X)$, and $S=S_p$,  $\gamma_p(\xi)=\omega_1$.  \item For $0<\vartheta<1$, $1\leqslant \mu<\omega_1$,  and $H=(T^{(q)}_{\mu, \vartheta})^*$, $\omega^\mu<\gamma_p(\xi)<\infty$. \end{enumerate}

\label{exam2}
\end{corollary}

\begin{rem}\upshape The examples from Corollary \ref{exam2} yield that we can obtain uncountably many distinct values of $\gamma_p(\xi)$ for different choices of $H$.  

\end{rem}

\begin{proof}[Proof of Corollary \ref{exam2}]In the proof, to avoid repetition, we leave it to the reader to make the appropriate substitution of $\ell_p$ norms with maxima and $\ell_p$ with $c_0$ in the $p=\infty$ case.

$(i)$ In the case $H=\ell_p$, then the canonical basis of $H$ is normalized, $1$-unconditional, shrinking, $1$-block stable, and $1$-left dominant. By Lemma \ref{exam1}, for any $C>1$ and any $L\in[\nn]$, there exists $M\in[\nn]$ such that for any $N\in[M]$, $\mathbb{E}^\xi_N\varsigma$ is $C$-dominated by the $\ell_p$ basis.  This yields that $\gamma_p(\xi)=\omega_1$, and in fact $\Gamma^\xi_p(\omega_1)\leqslant 1$.   In this case, the space $H_\xi$ is the higher order Baernstein space $X_{\xi,p}$.

$(ii)$  It was shown in \cite{LT} that $T_{\mu, \vartheta}$ is block stable. From this it easily follows that $T^{(q)}_{\mu, \vartheta}$ and its dual $H=(T^{(q)}_{\mu, \vartheta})^*$ are block stable. Let $B$ be such that $H$ is $B$-block stable.    Since the rank of $\mathcal{S}_\mu$ is $\omega^\mu+1$, there exists $A\in [\nn]$ such that for any $G\in \mathcal{F}_{\omega^\mu}$, $A(G)\in \mathcal{S}_\mu$.

 By Lemma \ref{exam1}, for any $C>B$, $\varsigma\in B_{c_0^w(H_\xi)}$, and $L\in[\nn]$, there exists $M\in[L]$ such that for any $N\in[\nn]$, $\mathbb{E}^\xi_N\varsigma$ is $C$-dominated by $(e_{A(n)})_{n=1}^\infty$.    Let $T\in [M]$ be as in the conclusion of Lemma \ref{veryeasy}.   Then for any $F=_{\mathcal{S}_\xi}\cup_{n=1}^t F_n \in \mathcal{F}_{\omega^\mu}^T[\mathcal{S}_\xi]$, there exist $H\in \mathcal{F}_{\omega^\mu}$ and $N\in[M]$ such that $F=\cup_{n\in H}\supp(\mathbb{S}^\xi_{N,n})$.     Write $H=(i_n)_{n=1}^t$.    Fix any scalars $(a_n)_{n=1}^t$ and let $b_{i_n}=a_n$ for $n=1, \ldots, t$. Let $b_n=0$ for $n\in \nn\setminus H$.    Then since $A(H)\in \mathcal{S}_\mu$, \begin{align*}\Bigl\|\sum_{n=1}^t a_n\mathbb{E}^\xi_F \varsigma(n)\Bigr\|_{H_\xi} & =\Bigl\|\sum_{n=1}^\infty b_n \mathbb{E}^\xi_N \varsigma(n)\Bigr\|_{H_\xi} \leqslant C\Bigl\|\sum_{n=1}^\infty b_n e_{A(n)}\Bigr\|_H \\ & =C\Bigl\|\sum_{n\in H} b_ne_{A(n)}\Bigr\|_H \leqslant  C\vartheta^{-1/q}\Bigl(\sum_{n\in H} |b_n|^p\Bigr)^{1/p} = C\vartheta^{-1/q}\Bigl(\sum_{n=1}^t |a_n|^p\Bigr)^{1/p}. \end{align*}  This yields that  $\Gamma^\xi_p(\mu)\leqslant B/\vartheta^{1/q}$ and $\gamma_p(\xi)>\mu$.

  However, if we take the canonical basis $\varsigma=(e_n)_{n=1}^\infty\in B_{c_0^w(H_\xi)}$, then for any $M\in[\nn]$, if $M=\cup_{n=1}^\infty E_n$, then for any $(a_n)_{n=1}^\infty\in c_{00}$, \begin{align*} \Bigl\|\sum_{n=1}^\infty a_n \mathbb{E}^\xi_M (n)\Bigr\|_{H_\xi} & \geqslant \Bigl\|\sum_{n=1}^\infty \Bigl(\bigl\|\supp(\mathbb{S}^\xi_{M,n})\sum_{m=1}^\infty a_m \mathbb{E}^\xi_M (m) \bigr\|_{\ell_1}\Bigr)e_{\max \supp(\mathbb{S}^\xi_{M,n})}\Bigr\|_H \\ & = \Bigl\|\sum_{n=1}^\infty |a_n|e_{\max \supp(\mathbb{S}^\xi_{M,n})}\Bigr\|_H.\end{align*}   Therefore $\mathbb{E}_M\varsigma$ $1$-dominates the subsequence $(e_{\max \supp(\mathbb{S}^\xi_{M,n})})_{n=1}^\infty$ of the $H$ basis. Since the $H=(T^{(q)}_{\mu,\vartheta})^*$ basis does not have a subsequence dominated by the $\ell_p$ basis, this shows that there cannot exist any $M\in[\nn]$ such that $\mathbb{E}_M\varsigma\in S_p$. This shows that $\gamma_p(\xi)<\omega_1$. 

To see that the basis of $H$ has no subsequence dominated by the $\ell_p$ basis, note that since the canonical basis of $H$ dominates the $\ell_p$ basis, any subsequence of the basis of $H$ which is dominated by the $\ell_p$ basis must be equivalent to the $\ell_p$ basis. If such a sequence existed, then $H$ would admit a complemented copy of $\ell_p$. Then $T^{(q)}_{\mu, \vartheta}$ would contain an isomorphic copy of $\ell_q$, and $T_{\mu, \vartheta}$ would contain an isomorphic copy of $\ell_1$.  But $T_{\mu, \vartheta}$ famously contains no isomorphic copy of $\ell_1$.

\end{proof}

%Let $H_\gamma$ be a Banach space which is the completion of $c_{00}$ with respect to some norm $\|\cdot\|_H$ making the canonical $c_{00}$ basis shrinking, normalized, $1$-unconditional, block stable, and $1$-left dominant.  

%Our first example of such a space is the Baernstein space $X_{\xi,p}$. 

%\begin{example}\upshape For $0\leqslant \xi<\omega_1$ and $1<p\leqslant \infty$, we let $X_{\xi,p}$ denote the completion of $c_{00}$ with respect to the norm $$\|x\|_{\xi,p}=\sup\Bigl\{\Bigl(\sum_{n=1}^\infty \|E_nx\|_{\ell_1}^p\Bigr)^{1/p}: E_1<E_2<\ldots, E_n\in \mathcal{S}_\xi\Bigr\}.$$ When $p=\infty$, we replace the $\ell_p$ norm with the maximum. In this case, $X_{\xi,\infty}$ is the \emph{Schreier space} $X_\xi$.   When $\xi=0$, $X_{0,p}=\ell_p$ for $1<p<\infty$ and $X_0=c_0$. 

%\end{example}

\end{document}